\documentclass{theoretics} 
\usepackage{mathtools}
\usepackage[english]{babel}
\usepackage{nicefrac}
\usepackage{mathdots,url}
\usepackage{dsfont,bm}
\usepackage{thm-restate}
\usepackage{calc}

\ThCSnewtheostd{question}

\DeclareMathOperator*{\argmax}{arg\,max}

\DeclareMathOperator{\rank}{rk} 

\newcommand{\R}{{{\mathbb{R}}}}
\newcommand{\Q}{{{\mathbb{Q}}}}
\newcommand{\N}{{{\mathbb{N}}}}
\newcommand{\Z}{{{\mathbb{Z}}}}
\newcommand{\K}{{{\mathbb{K}}}}
\newcommand{\1}{{{\mathds{1}}}}

\newcommand{\supp}{{{\mathrm{supp}}}}
\newcommand{\nat}{{{$\sp{\natural}$}}}

\newcommand\SetOf[2]{\left\{\left.#1\vphantom{#2}\ \right|\ #2\vphantom{#1}\right\}}

\newcommand{\cl}{\operatorname{cl}}
\newcommand{\dom}{{{\mathrm{dom}}}}
\newcommand{\Trop}{\R \cup \{-\infty\}}

\newcommand{\tempinducednew}[3]{\Phi(#1,#2,#3)}
\newcommand{\networkinducednew}[4]{\Phi_{#4}(#1,#2,#3)}

\newcommand{\temptrafonew}[3]{\Psi(#1,#2,#3)}
\newcommand{\free}[1]{{\rm fr}_{#1}}
\newcommand{\layer}[2]{\ell^{#1}\left({#2}\right)}

\newcommand{\Lor}[2]{{\rm{L}}^{#1}_{#2}}
\newcommand{\Mco}[2]{\text{M}^{#1}_{#2}}
\newcommand{\MAP}{{\rm{MAP}}}
\newcommand{\pseries}[2]{{#1}\{\!\{#2\}\!\}}

\newcommand{\actson}{\!\curvearrowright\!}

\newcommand{\cG}{{\mathcal{G}}} 
\newcommand{\E}{E} 
\newcommand{\M}{{\mathcal{M}}} 
\newcommand{\cN}{\mathcal{N}} 
\newcommand{\I}{\mathcal{I}} 
\newcommand{\J}{\mathcal{J}} 
\newcommand{\B}{\mathcal{B}} 
\newcommand{\cP}{\mathcal{P}} 
\newcommand{\cH}{\mathcal{H}} 
\newcommand{\cF}{\mathcal{F}} 
\newcommand{\Xs}{X^*} 

\newcommand{\co}{c} 
\newcommand{\Rnat}{\text{R}^{\natural}} 
\newcommand\restr[2]{{
  \left.\kern-\nulldelimiterspace 
  #1 
  \vphantom{\big|} 
  \right|_{#2} 
}}

\newcommand{\trop}{{\rm{trop}}} 

\makeatletter
\newcommand{\myitem}[1]{%
\item[#1]\protected@edef\@currentlabel{#1}%
}
\makeatother



\newcommand{\RestateRemark}[1]{{\normalfont\bfseries #1}}
\newcommand{\RestateInit}[1]{\newcommand{#1}{}}
\newcommand{\RestateGo}[1]{\renewcommand{#1}{(Restated)}}

\title{On complete classes of valuated matroids}

\ThCSauthor[London]{Edin Husi\'{c}}{edinehusic@gmail.com}[0000-0002-6708-5112]
\ThCSauthor[Twente]{Georg Loho}{g.loho@utwente.nl}[0000-0001-6500-385X]
\ThCSauthor[Lancaster]{Ben Smith}{b.smith9@lancaster.ac.uk}[0000-0002-3066-1534]
\ThCSauthor[London]{L{\'{a}}szl{\'{o}} A. V{\'{e}}gh}{lvegh@uni-bonn.de}[0000-0003-1152-200X]
\ThCSaffil[London]{London School of Economics and Political Science}
\ThCSaffil[Twente]{University of Twente}
\ThCSaffil[Lancaster]{Lancaster University}

\ThCSthanks{An extended abstract of this work was published in the proceedings of the ACM-SIAM Symposium on Discrete Algorithms in 2022 (SODA2022).This project has received funding from the European Research Council (ERC) under the European Union's Horizon 2020 research and innovation programme (grant agreement ScaleOpt--757481).}
\ThCSshortnames{E.\ Husi\'{c}, G.\ Loho, B.\ Smith, L.\ V{\'{e}}gh}  
\ThCSshorttitle{On complete classes of valuated matroids}
\ThCSyear{2024}
\ThCSarticlenum{24}
\ThCSreceived{Dec 29, 2022}
\ThCSrevised{May 16, 2024}
\ThCSaccepted{Jul 21, 2024}
\ThCSpublished{Nov 18, 2024}
\ThCSdoicreatedtrue
\ThCSkeywords{matroid, valuated matroid, gross substitutes, M-convex function, matroid minors, induction by network, Lorentzian polynomial}

\begin{document}
\maketitle

\begin{abstract}
We characterize a rich class of valuated matroids, called \emph{R-minor valuated matroids} that includes the indicator functions of matroids, and is closed under operations such as taking minors, duality, and induction by network. 
We exhibit a family of valuated matroids that are not R-minor based on sparse paving matroids.

Valuated matroids are inherently related to gross substitute valuations in mathematical economics. By the same token we refute the Matroid Based Valuation Conjecture by Ostrovsky and Paes Leme (Theoretical Economics 2015) asserting that every gross substitute valuation arises from weighted matroid rank functions by repeated applications of merge and endowment operations. Our result also has implications in the context of Lorentzian polynomials: it reveals the limitations of known construction operations.
\end{abstract}

\section{Introduction}

Valuated (generalized) matroids capture a quantitative version of the exchange axiom(s) for matroids, first introduced by Dress and Wenzel~\cite{DressWenzel:1992}.
Later, Murota~\cite{Murotabook} identified them as a fundamental concept in discrete convex analysis. 
They play important roles across different areas of mathematics and computer science, with several applications in algorithmic game theory.

Valuated matroids and valuated generalized matroids
can be defined in many equivalent ways: in tropical geometry~\cite[Theorem 4.1.3]{Frenk:2013}, via the interplay of price and demands in economics~\cite{Leme2017}, or with various exchange properties~\cite{Murotabook}.
We follow~\cite{FujishigeHirai:2020,MurotaShioura:2018}, 
and say that a function $f \colon 2^V \to \Trop$ is a \emph{valuated generalized matroid} if two properties hold:
\begin{subequations}
  \begin{equation} \label{eq:Mnat-concave}
    \begin{aligned}
      &\forall X,Y \subseteq V \text{ with } |X| < |Y|:  \\
      &f(X) + f(Y) \leq \max_{j \in Y \setminus X} \{f(X + j) + f(Y - j)\}
    \end{aligned}
\end{equation}
  \begin{equation} \label{eq:M-concave}
    \begin{aligned}
      &\forall X,Y \subseteq V \text{ with } |X| = |Y| \text{ and } \forall i \in X \setminus Y: \\
      &f(X) + f(Y) \leq \max_{j \in Y \setminus X} \{f(X - i + j) + f(Y + i - j)\} .
    \end{aligned}
\end{equation}
\end{subequations}
For fixed $r \leq |V|$, those set functions $\binom{V}{r} \to \R\cup\{-\infty\}$ fulfilling~\eqref{eq:M-concave} are \emph{valuated matroids}.
This means that each layer of a valuated generalized matroid is a valuated matroid.
Conversely, one can represent all valuated generalized matroids by valuated matroids, see Appendix~\ref{sec:vgm-to-vm}.

The axiom~\eqref{eq:M-concave} can be seen as a quantitative version of the strong basis exchange property of matroids. 
Valuated matroids with codomain $\{0, -\infty\}$ correspond to matroids: 
the sets taking value $0$ form the bases of a matroid, and conversely, every matroid gives rise to such a valuated matroid. We call these \emph{trivially valuated matroids}.

\paragraph{R-minor valuated matroids} 
We are interested in the following classes of valuated matroids arising from independent matchings in bipartite graphs.
The term pays tribute to Richard Rado, who introduced the induction of matroids through bipartite graphs in \cite{Rado1942}.

\begin{definition}[R-minor, R-induced] \label{def:intro-R-minor}
Let $G = (V\cup W, U;E)$ be a bipartite graph with disjoint vertex sets $V\cup W$ and $U$, edges $E$ with edge weights $\co\in \R^E$, and let $\M$ be a matroid on $U$ of rank $d + |W|$ for some $d\in \mathbb{N}$.
Such a graph is displayed in Figure~\ref{fig:RminorRepresentation}.
We define an \emph{R-minor} valuated matroid $f \colon {V\choose d}\to \R$ for $X\in {V\choose d}$ as follows. 

The value $f(X)$ is the maximum weight of a matching in $G$ whose endpoints in $V\cup W$ are $X\cup W$, 
and the endpoints in $U$ form a basis in $\M$.
For $W= \emptyset$, the function $f$ is called an \emph{R-induced} valuated matroid.

Observe that every R-minor function $f$ on $V$ arises as a (valuated) \emph{contraction} of an R-induced function $g$ on $V \cup W$, i.e. $f(X) = g(X \cup W)$.
\end{definition}
\begin{figure}
\centering
\includegraphics[scale=1]{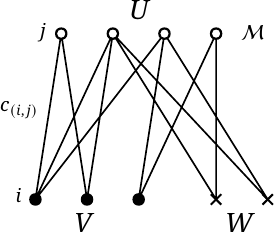}
\caption{A bipartite graph $G = (V\cup W, U;E)$ with edge weights $\co\in \R^E$ and matroid $\M$ on vertex set $U$.
This gives rise to an R-minor valuated matroid on $V$, as described in Definition~\ref{def:intro-R-minor}.}
\label{fig:RminorRepresentation}
\end{figure}

This concept naturally extends to valuated generalized matroids: the endpoints in $U$ should not form a basis but a set in a \emph{generalized matroid.}\footnote{These are defined as the effective domain of a $\{0,-\infty\}$-valued valuated generalized matroid, see Section~\ref{sec:valuated+generalized+matroids}. The canonical examples are independent sets of matroids.}
In 2003, Frank~\cite{lecture,problemSheet} asked if all valuated matroids arise as \emph{R-induced} valuated matroids. 
The corresponding version of this conjecture for valuated generalized matroids has been recently disproved by Garg et al.~\cite{garg2020approximatingNSW,nswRado}, 
in which they observed that valuated (generalized) matroids are closed under contraction, whereas R-induced valuated (generalized) matroids are not.

Noting that R-minor valuated matroids are precisely the contractions of R-induced valuated matroids, 
this suggests a natural refinement of the original conjecture:
\begin{quote}\emph{Do all valuated matroids arise as  R-minor valuated matroids?}
\end{quote}
A variant of this conjecture on valuated generalized matroids was proposed in \cite{garg2020approximatingNSW}. 
The main contributions of this paper are {\em (i)} showing that R-minor valuated matroids form a \emph{complete class} of valuated matroids, a family closed under several fundamental operations, yet {\em (ii)} not all valuated matroids arise in this form, disproving the above conjecture. We then derive implications for gross substitute valuations and for Lorentzian polynomials.

\paragraph{Complete classes of matroids}
Let us consider R-induced and R-minor valuated matroids where $\M$ is the free matroid, the matroid all of whose subsets are independent, and $c$ is weight zero on all edges.
The trivially valuated matroids arising this way are \emph{transversal matroids} and \emph{gammoids}, respectively.
In 1977, Ingleton~\cite{Ingleton:1977} studied representations of transversal matroids and gammoids.
He observed that gammoids arise via this simple construction yet form a rich class closed under several fundamental matroid operations.
This motivated the definition of a \emph{complete class} of matroids by requiring closure under the operations \emph{restriction, dual, direct sum, principal extension}, four key operations that gammoids are closed under.
Ingleton showed that gammoids arise as the smallest complete class by taking the closure of the matroid on a single element under these four operations.
Moreover, complete classes are closed under a number of other natural matroid operations, including contraction, matroid union and truncation.
An important example of such an operation is \emph{induction by bipartite graph}, given by fixing a matroid on one node set of a bipartite graph and inducing a matroid structure on the other node set via matchings to independent sets.
Induction encompasses many other natural matroid operations, and
closure under this operation is what creates the rich structure of
complete classes.\newpage 

The theory of complete classes was further developed in Bonin and Savitsky~\cite{BoninSavitsky:2016} who also collected the necessary properties to define a complete class. 
Brualdi~\cite{Brualdi:1971} revealed that if a matroid is \emph{base orderable}\footnote{A matroid is base orderable, if, for any two bases $B_1$ and $B_2$, there is a bijection $\sigma \colon B_1 \to B_2$ such that $B_1 - x + \sigma(x)$ and $B_2 - \sigma(x) + x$ are bases for every $x \in B_1$. }, then so is each matroid induced from it.
As gammoids are base orderable but the graphical matroid of $K_4$ is not, one can see that not all matroids are gammoids and that there are larger complete classes strictly containing gammoids.

\paragraph{Gross substitutes} 
A somewhat surprising application of valuated generalized matroids arises in mathematical economics. 
Gross substitutability captures the following type of interaction between prices and demands for goods.
At given prices, an agent would like to buy a certain amount of goods.
If the price of a single good increases then we expect that demand for this good decreases.
Gross substitutability dictates that previously desirable goods whose price is unaffected should remain desirable.
This concept is crucial for
equilibrium existence and computation~\cite{arrow1959stability,arrow1958stability,DBLP:conf/stoc/CodenottiMV05}, 
auction algorithms~\cite{garg2019auction,gul1999walrasian,DBLP:journals/geb/LehmannLN06}, 
and mechanism design~\cite{ausubel2002ascending,HatfieldMilgrom:2005}.

In the case of discrete (indivisible) goods, 
an agent determines her demand by maximizing a \emph{valuation function}: a monotone set function taking value $0$ on the empty set.
Hence, gross substitutability is a property of a function. 
It turns out that the functions with the gross substitute property (GS functions) are exactly valuated generalized matroids~\cite{MurotaShioura:2018}.

For indivisible goods, the property was first formalized by Kelso and Crawford~\cite{kelso1982job} to show that a natural auction-like price adjustment process converges to an equilibrium. 
We also point out that Gul and Stacchetti~\cite{gul1999walrasian} showed that the so-called Walrasian equilibrium exists whenever agents' valuations satisfy the gross substitute property and that, in a sense, the converse also holds. 
For further results on gross substitutability, we refer to~\cite[Chapter 11]{nisan2007algorithmic} and a survey by Paes Leme~\cite{Leme2017}. 

A classical example of GS functions ($=$ valuated generalized matroids) are \emph{assignment (OXS) valuations} introduced by Shapley~\cite{shapley1962optimal}. 
For a graph $G = (V, U; E)$ with edge-weights $c\in\R_{\geq 0}^E$, 
the value $v(X)$ for $X\subseteq V$ is defined as the maximum weight matching with endpoints in $X$.\footnote{
   Shapley introduces the valuations as follows. 
   Assume that $V$ are workers and $U$ is the set of jobs within a company. 
   The edge set represents the possibilities (willingness) of assigning workers to jobs, 
   and the weight $c_{ij}$ is the value the company gets by assigning worker $i$ to job $j$. 
   Then the value of a subset $X\subseteq V$ of workers for the company is the maximum
  possible value the company gets by assigning workers $X$ to jobs $U$.
}

\paragraph{Constructions of substitutes}
By the equivalence with valuated generalized matroids, 
functions with the gross substitute property can be described in many different ways. 
Balkanski and Paes Leme~\cite{BalkanskiLeme:2020} mention eight characterizations of GS functions.
Nevertheless, finding a constructive description of all GS functions/valuations remained elusive.

The first attempt to ``construct'' all GS valuations was by Hatfield and Milgrom~\cite{HatfieldMilgrom:2005}.
After observing that most examples of GS valuations arising in applications are built from assignment valuations 
and the endowment operation, they asked if this is true for all GS valuations.
Ostrovsky and Paes Leme~\cite{OstrovskyPaesLeme:2015} showed that this is not the case: 
some matroid rank functions cannot be constructed as endowed assignment valuations 
while all (weighted) matroid rank functions are GS valuations.
Instead, Ostrovsky and Paes Leme proposed the 
\emph{matroid based valuations (MBV)}  conjecture.
Matroid based valuations are those that arise from weighted matroid rank functions by repeatedly applying the operations of \emph{merge} and \emph{endowment}. Tran~\cite{Tran:2019} showed that using only merge but no endowment operations does not suffice, but the conjecture remained open.

Originally, interest for such conjectures stemmed from auction design. 
They are an attempt at designing a language in which
agents can represent their valuations in a compact and expressible way~\cite{Leme2017}.
Moreover, valuations with a constructive description facilitate more algorithmic techniques, especially linear programming (see Section~\ref{section:LPs} and e.g.,~\cite{nswRado}).
The quest for succinct representations (of matroids) is also intimately connected to questions in parametrized complexity, 
see e.g.~\cite{KratschWahlstroem:2020}.
In this paper, we analyze and disprove the MBV conjecture through the lens of complete classes.

\smallskip
\paragraph{Sparse paving matroids} A crucial tool for our counterexamples to the conjectures are valuated matroids arising from the well-known class of sparse paving matroids. A matroid of rank $d$ is \emph{paving} if all circuits are of size $d$ or $d+1$, and \emph{sparse paving} if in addition the intersection of any two $d$-element circuits is of size at most $d-2$.
Knuth~\cite{Knuth:1974} gave an elegant construction of a doubly exponentially large family of sparse paving matroids; this is essentially  the strongest lower bound on the number of matroids on $n$ elements.
In fact,  it was conjectured in \cite{Mayhew2011} that asymptotically almost all matroids are sparse paving; weaker versions were proved in  \cite{BansalPendavinghVanderPol:2015} and \cite{Pendavingh2015number}. Our main  valuated matroid construction is based on sparse paving matroids that arise from  Knuth's construction.

\smallskip
\paragraph{Lorentzian polynomials}
Br{\"a}nd\'en and Huh~\cite{BraendenHuh:2020} recently introduced \emph{Lorentzian polynomials} generalizing stable polynomials in optimization theory and volume polynomials in algebraic geometry.
They act as a bridge between discrete and continuous convexity.
In particular, their domains form discrete convex sets, generalizing earlier work connecting matroids and polynomials, e.g., by Choe et al.~\cite{ChoeOxleySokalWagner:2004}.
Their connection to continuous convexity is via their equivalence to strongly log-concave polynomials discovered by Gurvits~\cite{Gurvits:2010} and completely log-concave polynomials which were used by Anari et al.~\cite{AnariGharanVinzant:2018} in their breakthrough work for efficiently sampling bases of matroids. 
This connection has lead to applications in numerous areas such as combinatorial optimization~\cite{AnariLiuGharanVinzant:2019,AnariMaiGharanVazirani:2018}.
Furthermore, they are intimately connected to valuated matroids via tropical geometry: Br{\"a}nd\'en and Huh showed that the space of valuated matroids arises as the tropicalization of squarefree Lorentzian polynomials. 

There is ongoing research regarding the space of Lorentzian polynomials~\cite{Branden:2020}.
They are closed under many natural operations analogous to valuated matroids, therefore a natural question is whether one can construct the space of Lorentzian polynomials from certain ``building block'' functions closed under these operations.
We use our techniques to deduce limitations of such constructions.

\subsection{Our contributions}
\paragraph{Complete classes of valuated matroids}
We introduce the notion of complete classes of valuated matroids. 
These are classes of valuated matroids closed under the valuated generalizations of the fundamental operations \emph{restriction, dual, direct sum}, and \emph{principal extension}.
The crucial ingredient going beyond the basic operations already introduced in~\cite{DressWenzel:1992} is (valuated) principal extension. 
This is a special case of \emph{transformation by networks}~\cite[Theorem~9.27]{Murotabook}.
These operations appeared as `linear maps' and `linear extensions' in tropical geometry~\cite{Frenk:2013,Mundinger:2018}. 
Right from the definition, valuated gammoids are seen to form the smallest complete class of valuated matroids (Theorem~\ref{thm:complete+class}). 

The valuated matroids arising as building blocks in both Frank's question and the MBV conjecture arise as R-minor valuated matroids.
As both discuss closure under fundamental operations, the study of their complete classes gives rise to a common framework.
We can also consider existing results from different fields in a unified manner: 
The proof of Ostrovsky and Paes Leme~\cite{OstrovskyPaesLeme:2015} that endowed assignment valuations do not exhaust all GS functions is based on a valuated analogue of (strongly) base orderable matroids.
Also the work on \emph{Stiefel tropical linear spaces} in tropical geometry~\cite{FinkOlarte:2021,FinkRincon:2015} can be considered as the study of representations in the complete class of valuated gammoids.

\paragraph{Complete class containing trivially valuated matroids}
After introducing complete classes, an immediate question arises: 
does the smallest complete class that contains all trivially valuated matroids cover all valuated matroids?
Or in other words, does the smallest class that contains all trivially valuated matroids and is closed for deletion, contraction, duality, truncation, and principal extension exhaust all valuated matroids? 

We show that the smallest class of valuated matroids that contains all trivially valuated matroids 
and is closed for the above operations is exactly the class of \emph{R-minor} valuated matroids. 
Thus, the above question asks if every valuated matroid is an R-minor valuated matroid. 
We can use an information-theoretic argument to show that not all valuated matroids are \emph{R-induced} by constructing valuated matroids with many independent values (Appendix~\ref{sec:size-bound-R-rep}).
However, such an argument does not seem to extend to R-minor valuated matroids, since the size of the contracted set $W$ may be arbitrarily large.
Thus, the construction disproving the more general claim relies on a well-chosen family of valuated matroids.

\paragraph{Non-R-minor valuated matroids}
The most challenging part of our paper is proving that there are valuated matroids that are not R-minor valuated matroids.
In particular, we show that none of the valuated matroids in the following family is R-minor.

\RestateInit{\restatefamily}
\begin{restatable}{definition}{family}\label{def:class}\RestateRemark{\restatefamily}
For $n\ge 3$, we define $\cF_n$ as the following family of functions $\binom{[2n]}{4}\to\R$.
Let $V = [2n]$, $P_i = \{2i - 1, 2i\}$ for $i\in [n]$, and let
\begin{equation}\label{eq:parity-pairs}
\cH = \SetOf{P_i \cup P_j}{ij \equiv 0 \mod 2}
\tag{$\mathcal{H}$-def}
\end{equation}
i.e. we take pairs such that at least one of $i$ and $j$ are even.
Let $\Xs=P_1\cup P_2=\{1,2,3,4\}$.
A function $h:{V \choose 4}\to \R\cup\{-\infty\}$ is in the family $\cF_n$ if and only if the following hold:
\begin{itemize}
  \item $h(X)=0$ if $X\in {V \choose 4}\setminus  {\mathcal{H}}$,
  \item $h(X)<0$ if $X\in {\mathcal{H}}$, and
  \item $h(\Xs)$ is the unique largest nonzero value of the function.
\end{itemize}
\end{restatable}
\RestateGo{\restatefamily}

\begin{theorem}[Main]
\label{thm:non-r-minor-main}
If $n\ge 3$, then all functions in $\cF_n$ are valuated matroids.
If $n\ge 16$, then no function in $\cF_n$ arises as an  $R$-minor function.
\end{theorem}
The functions in $\cF_n$ are derived from sparse paving matroids; our construction was inspired by Knuth's~\cite{Knuth:1974} work. The construction also resembles the \emph{V\'amos matroid} which is an example of a matroid that is not representable over any field, see~\cite[Proposition 2.2.26]{oxley}. 
 We note that if $\B$ is the family of bases of a sparse paving matroid of rank $d$, then any function $h:{V \choose d}\to \R\cup\{-\infty\}$ with $h(X)=0$ if $X\in {\mathcal{B}}$ and $h(X)<0$ otherwise gives a valuated matroid, see Appendix~\ref{section:functionH}. This implies in particular that all functions in $\cF_n$ are valuated matroids.

As our family allows still for quite some flexibility and it is conjectured that almost all matroids are sparse paving~\cite{Mayhew2011}, one could guess that even \emph{almost all} valuated matroids might not be R-minor.
But the development of the framework for making such a statement goes beyond the scope of this paper.

\paragraph{Refuting the Matroid Based Valuation Conjecture}
Building on Theorem~\ref{thm:non-r-minor-main}, 
we also refute the MBV conjecture by Ostrovsky and Paes Leme~\cite{OstrovskyPaesLeme:2015}.
This is done by considering valuated generalized matroids corresponding to R-minor valuated matroids and reduce to Theorem~\ref{thm:non-r-minor-main} by considering their layers. 

First, we show that every function that can be obtained from weighted matroid rank functions by repeatedly applying merge and endowment 
is an R\nat-minor valuated generalized matroid --- the class of valuated generalized matroids arising by contraction and induction from generalized matroids. Garg et al.~~\cite{garg2020approximatingNSW} proposed the conjecture that all valuated generalized matroids have an R\nat-minor representation.

Then, we show that the function $h^\natural \colon 2^V \to \R_{\geq 0}$ defined as follows is a valuated generalized matroid but not R\nat-minor (Theorem~\ref{thm:mbv+counterexample}). 
This disproves the conjecture in \cite{garg2020approximatingNSW}, as well as the MBV conjecture.
For an arbitrary valuated matroid $h \in \cF_n$ taking values only in $(-1,0]$, we define 
\begin{equation*}
  h^\natural (X) := \begin{cases}
    |X| & \text{ for } |X| \leq 3, \\ 
    4 + h(X) & \text{ for } |X| = 4,\\
    4 & \text{ for } |X| \geq 5. \\
  \end{cases}
\end{equation*}
We achieve this by focusing on the function restricted to all $4$-subsets of $V$.
This is an R-minor valuated matroid and therefore allows us to apply Theorem~\ref{thm:non-r-minor-main}. 
Note that the function $h^\natural$ has the additional structure of being \emph{monotone} and only taking \emph{non-negative finite} values, as the MBV conjecture refers to valuations. 
Finally, we note that while matroid based valuations form a subset of R\nat-minor valuated generalized matroids $f$ that are monotone and $f(\emptyset)=0$, it is unclear whether the containment is strict or if these two classes coincide.

\paragraph{Lorentzian polynomials}

A fundamental operation which preserves Lorentzian polynomials is by the multiplicative action of non-negative matrices in the argument~\cite[Theorem~2.10]{BraendenHuh:2020}. 
This means that, given a Lorentzian polynomial $p$ in $n$ variables, a non-negative matrix $A \in \R^{n \times k}$ and a variable vector $(y_1,\dots,y_k)$, the polynomial given by $p(A\cdot y) \in \R[y_1,\dots,y_k]$ is also Lorentzian.

We demonstrate that several basic operations for Lorentzian polynomials translate to the basic operations considered for valuated matroids via \emph{tropicalization}, where one replaces the polynomial by a map from the exponents to the respective coefficients.  Most notably, the multiplicative action of non-negative matrices on Lorentzian polynomials translates to induction via bipartite graphs for valuated matroids (Theorem~\ref{thm:Lorentzian-transformation}). 

Taking polynomials which correspond to our family of counterexamples given in Definition~\ref{def:class} via tropicalization, we can deduce limitations on constructions of Lorentzian polynomials.  
Explicitly, we show that not all Lorentzian polynomials over the Puiseux series can be realized by the action of non-negative matrices on \emph{generating functions of matroids} (Theorem~\ref{thm:matroid+induced+subclass}).
We then show a weaker restriction for Lorentzian polynomials over the reals, namely that there exist Lorentzian polynomials that would require arbitrarily large matroids to be realized by this construction.
The proof is based on the relation between polynomials over real-closed fields via Tarski's principle.

\subsection{Organization of the paper}

In Section~\ref{sec:preliminaries}, we define the operations on valuated matroids: restriction (deletion), contraction, dual, principal extension, induction by network, and induction by bipartite graph. 
Complete classes of valuated matroids are introduced in Section~\ref{sec:completeClasses}, 
where we also prove that R-minor valuated matroids form a complete class. 

Theorem~\ref{thm:non-r-minor-main} is proved in two parts. 
The proof that functions in $\cF_n$ are valuated matroids follows from simple case analysis given in Appendix~\ref{section:functionH}.
We prove that no function in $\cF_n$ arises as an R-induced minor function in Section~\ref{sec:R-minor+not+cover}; the proof uses several lemmas on Rado representations of matroids given in Section~\ref{section:structureOfMatroids}, and lemmas on the LP representation of R-minor valuated matroids given in Section~\ref{section:LPs}.
A roadmap to the main proof is given in Section~\ref{sec:roadmap-new}.

Valuated generalized matroids and the MBV conjecture are treated in Section~\ref{sec:valuated+generalized+matroids}.
Section~\ref{sec:lorentzian+polynomials} presents implications of our work for constructions of Lorentzian polynomials.

\subsection{Notation}

We denote a matroid $\mathcal{M}$ by $\mathcal{M} = (U,r)$ where $U$ is the ground set of the matroid and $r = r_{\mathcal{M}}$ is the rank function of the matroid.
We assume familiarity with the basic concepts in matroid theory; we refer the reader to Oxley's book~\cite{oxley} for an introduction to matroids.
Our notation of the major operations on matroids follows the notation of valuated matroids introduced in Section~\ref{sec:preliminaries}, as these are special cases of the valuated operations.

A function $\rho: 2^V \to \R$ is \emph{submodular} if for every $A, B\in 2^V$ it holds $\rho(A) + \rho(B) \ge \rho(A\cap B) + \rho(A\cup B)$.
The rank function of a matroid is well-known to be submodular.

Given a set $V$, we denote its set of subsets of cardinality $d$ by $\binom{V}{d}$.
Given two sets $X,Y$, we denote their disjoint union by $X \dot{\cup} Y$.

We denote a bipartite graph $G$ by $G = (V, U; E)$, where $V$ and $U$ are the two parts of the node set and $E$ is the edge set.
For a subset of nodes $Y\subseteq U\cup V$, we denote the set of neighbours of $Y$ by $\Gamma_G(Y)$ or $\Gamma_E(Y)$. 
When the graph is clear from the context, we drop the subscript.
Given a set of edges $\mu\subseteq E$ and a subset of nodes $Y$, 
we let $\partial_Y(\mu)$ denote the nodes in $Y$ incident to the subgraph induced by $\mu$.
For a cost function $c\in\R^E$, we let $c(\mu):=\sum_{e\in \mu} c_e$ denote the cost of the edge set $\mu$.
By a \emph{network} or \emph{directed network}, we will refer to directed graphs.

\section{Operations on valuated matroids}
\label{sec:preliminaries}

For a valuated matroid $f$, its \emph{(effective) domain} $\dom(f)$ is formed by those sets $X$ on which $f(X) > -\infty$. 
The exchange property implies that it forms the set of bases of a matroid.
The \emph{rank} $\rank(f)$ of a valuated matroid $f$ is the rank of the underlying matroid $\dom(f)$. 

\begin{definition}\label{def:M+operations}
Let $f \colon \binom{V}{d} \to \Trop$ be a valuated matroid with $d = \rank(f)$, and $Y\subset V$ some subset of $V$.
\begin{enumerate}[label=(\roman*)]
\item If $V \setminus Y$ has full rank in $\dom(f)$ then the \emph{deletion} of $f$ by $Y$ is the function $f\setminus Y \colon \binom{V\setminus Y}{d} \to \Trop $ defined as
  \[
  (f \setminus Y)(X) = f(X), \quad \forall X \in \binom{V\setminus Y}{d}. 
  \]
  This is also called the \emph{restriction} to $V \setminus Y$ and denoted by $f| (V \setminus Y)$.
  If $V \setminus Y$ does not have full rank in $\dom(f)$, the deletion is the function attaining only $-\infty$. 
\item If $Y$ is independent in $\dom(f)$, the \emph{contraction} of $f$ by $Y$ is the function $f/Y \colon \binom{V\setminus Y}{d-|Y|} \to \Trop$ defined as
  \[
  (f/Y)(X) = f(X \cup Y), \quad \forall X \in \binom{V\setminus Y}{d-|Y|}.
  \]
  If $Y$ is not independent in $\dom(f)$, the contraction is the function attaining only $-\infty$. 
\item The \emph{dual} of $f$ is the function  $f^* \colon \binom{V}{|V| - d} \to \Trop$ defined as
  \[
  f^*(X) = f(V \setminus X), \quad \forall X\in \binom{V}{|V| - d} .
  \]
\item The \emph{truncation} of $f$ is the function $f^{(1)} \colon \binom{V}{d -1} \to \Trop$ defined as
   \[
  f^{(1)}(X) = \max_{v \in V \setminus X}f(X \cup v), \quad \forall X\in \binom{V}{d-1} .
  \]
  The iterated truncation for $1 \leq r \leq d-1$ is given by $f^{(r+1)} = (f^{(r)})^{(1)}$. 
\item For a weight function $w \in (\Trop)^V$, the \emph{principal extension} $f^w$ of $f$ with respect to $w$ is the valuated matroid on $V \cup p$ of rank $d$, for an additional element $p$, with $f^w|V = f$ and
\begin{equation*}
  f^w(X \cup p) = \max_{v \in V \setminus X} \left(f(X \cup v) + w_v\right) \quad \text{ for all } \quad X \in \binom{V}{d-1} \enspace .
\end{equation*}

\end{enumerate}
\end{definition}

\begin{remark}
Our definition of deletion and contraction differs from the usual definition, e.g. in~\cite{DressWenzel:1992}, in that we impose these rank conditions.  The usual definition of deletion (and dually contraction) for matroids could equally be formulated by first performing a truncation (to the rank of the remaining set) and then a deletion.   While for unvaluated matroids this is the same, for valuated matroids the naive deletion, where the remaining set does not have full rank, would result in a function only taking $-\infty$ as value. 
  Our reason to be more restrictive with deletion and contraction is that these definitions allow for simple `layer-wise' extensions to valuated generalized matroids in Section~\ref{sec:valuated+generalized+matroids} and they tie in more naturally with operations on polynomials as we demonstrate in Section~\ref{sec:lorentzian+polynomials}.
\end{remark}

\begin{example} \label{ex:trivially-valuated-matroids}
  The most basic examples of valuated matroids are those with \emph{trivial valuation}, 
  where only the values $0$ and $-\infty$ are attained (following naming as in~\cite{FinkOlarte:2021}).
  Such valuated matroids can be identified with the underlying matroid.
   Each of the operations listed in Definition~\ref{def:M+operations} are valuated analogues of matroid operations, see~\cite{oxley} for formal definitions.
These operations can also be recovered by restricting the valuated operations to trivially valuated matroids.
\end{example}

\begin{example}\label{ex:transversal+VM}
  Valuated matroids corresponding to the layers of the assignment valuations are \emph{transversally valuated matroids}.
  For a bipartite graph $G= (V, U; E)$ with edge weights $c\in\R^E$, we define a transversally valuated matroid $f\colon{V\choose |U|} \rightarrow \Trop$ by setting $f(X)$ to the maximum weight of a matching whose endpoints in $V$ are exactly $X$, for $X\in {V\choose |U|}$; if no such matching exists then we set $f(X) = - \infty$.    
  
  Let $V = [4]$ and consider the valuated matroid $f\colon{V\choose 2} \rightarrow \Trop$ defined as
  \[
  f(12) = -\infty \ , \ f(13) = 0 \ , \ f(14) = 0 \ , \ f(23) = 1 \ , \ f(24) = 1 \ , \ f(34) = 1 \, .
  \]
  This valuated matroid is transversally valuated as it can be realized via the weighted bipartite graph shown in Figure~\ref{fig:transversal+VM}.
\end{example}
\begin{figure}
\centering
\includegraphics[width=0.25\textwidth]{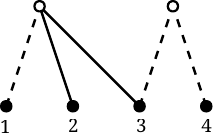}
\caption{The bipartite graph realising the transversally valuated matroid from Example~\ref{ex:transversal+VM}. The dashed edges have weight zero and the solid edges have weight one.}
\label{fig:transversal+VM}
\end{figure}

\begin{example}
  One source of valuated matroids arises from matrices with polynomial entries.
  Let $A$ be a matrix with $d$ rows and columns labelled by $V$, whose entries are univariate polynomials over a field.
  For $J \subseteq V$, we denote by $A[J]$ the submatrix formed by the columns labelled by $J$.
  The valuated matroid induced by $A$ is defined to be
  \[
  f(J) = \deg \det A[J] \, ,
  \]
  where $f(J) = -\infty$ if $\det A[J] = 0$ or $A[J]$ is non-square, see~\cite[Section 2.4.2]{Murotabook} for further details.
  
  Recall the valuated matroid from Example~\ref{ex:transversal+VM}.
  Observe that we can also represent this matrix via the polynomial matrix
  \[
  A = \begin{bmatrix}
  1 & t & t & 0 \\ 0 & 0 & 1 & 1
  \end{bmatrix} \, 
  \]
  e.g. $f(23) = \deg(t) = 1$.
\end{example}

Recall the following matroid operations for combining matroids.
Consider two matroids $\M_1$ and $\M_2$ on not-necessarily disjoint ground sets $U_1,U_2$ with sets of independent sets $\I_1, \I_2$.
Their \emph{matroid union} is the matroid $\M_1 \vee \M_2$ on the ground set $U_1 \cup U_2$ with independent sets given by $\I = \{I_1 \cup I_2 : i_i \in \I_i\}$.
If $U_1,U_2$ are disjoint, then we call $\M_1 \oplus \M_2 := \M_1 \vee \M_2$ their \emph{direct sum}.
These operations also have valuated analogues.

\begin{definition}[Valuated matroid union]
\label{def:valuatedSumUnion}
  Let $f_1$ and $f_2$ be valuated matroids on ground sets $V_1$ and $V_2$ with ranks $d_1$ and $d_2$, and let $V=V_1\cup V_2$. The \emph{(valuated) matroid union} of $f_1$ and $f_2$ is $(f_1 \vee f_2) \colon \binom{V}{d_1 + d_2} \to \Trop$, where
    \[
    (f_1 \vee f_2)(X) = \max\SetOf{f_1(Y) + f_2(X \setminus Y)}{Y \subseteq X \, , \, Y \in \binom{V_1}{d_1} \, , \, X\setminus Y \in \binom{V_2}{d_2}} .
    \]
  Undefined sets get the value $-\infty$. 
\end{definition}
Note that for the special case when $V_1$ and $V_2$ are disjoint, the only sets with finite value are $X$ such that $|X\cap V_1|=d_1$ and $|X\cap V_2|=d_2$. For this case, we can write a simpler formula. This will also be called the  \emph{direct sum} of $f_1$ and $f_2$ and denoted as $(f_1 \oplus f_2) \colon \binom{V}{d_1 + d_2} \to \Trop$. Thus,
    \[
    (f_1 \oplus f_2)(X) = f_1(X\cap V_1) + f_2(X\cap V_2) 
    \]
with value $-\infty$ unless $|X\cap V_1|=d_1$ and $|X\cap V_2|=d_2$. 

\begin{example}\label{ex:extension-by-coloops}
  Given a matroid on some ground set, it is often useful to extend that ground set to a larger ground set by adding coloops, elements contained in all bases.
  The same construction can be generalized to valuated matroids in the following way.
  
  Let $f$ be a valuated matroid on ground set $V$, and $W$ a disjoint set from $V$.
  We define the \emph{free valuated matroid} $\free{W}$ on $W$ to take the value 0 on $W$ (and all sets of smaller size get value $-\infty$); so the ground set itself is the only basis of the underlying matroid. 
  Then the direct sum of $f$ with $\free{W}$ is given by
  \[
  (f \oplus \free{W})(X) = \begin{cases} f(Y) & X = Y \cup W \\ -\infty & \text{otherwise} \end{cases} \, .
  \]
  In particular, note that $f = (f \oplus \free{W})/W$.
  This construction of adding coloops to a valuated matroid will be useful throughout.
\end{example}

\subsection{Induction by networks} The next operation is very powerful 
and can be seen as a vast generalization of Rado's theorem (Theorem~\ref{thm:RadosTheorem}). 
Somewhat surprisingly, we show that it can be modelled by the basic operations defined so far. This can be seen as a generalization 
of transversal valuated matroids (Example~\ref{ex:transversal+VM}). Instead of finding a maximum weight matching in a bipartite graph, we embed $V$ and $U$ into a directed network, with a valuated matroid $g$ on $U$. For a subset $X$ of $V$, we consider the maximum weight set of node-disjoint paths from $X$ to a subset $Y\subseteq U$, plus $g(Y)$. As a special case when $g$ is the $0$/$-\infty$ indicator function of a matroid, this means that we need to find node-disjoint paths to an independent set.

\begin{definition} \label{def:induction-by-networks}
  Let $N=(T, A)$ be a directed network with a weight function $\co \in \R^A$. Let $V, U\subseteq T$ be two non-empty subsets of nodes of $N$.
  Let $g$ be a valuated matroid on $U$ of rank $d$.
  Then the \emph{induction of $g$ by $N$} is a function $\networkinducednew{N}{g}{\co}{V} \colon \binom{V}{d} \to \Trop$.
  For $X \in \binom{V}{d}$, one sets
\[
\networkinducednew{N}{g}{\co}{V}(X) = \max\SetOf{\sum_{a \in \mathcal{P}} \co(a) + g(Y)}{\begin{aligned}&\text{node-disjoint paths } \mathcal{P} \text{ in } N \colon \\ &\partial_V({\cP}) = X \wedge \partial_U({\cP}) = Y \end{aligned}} .
\]
Note that the maximization can also result in $-\infty$ if  no node-disjoint paths exist from $X$ to a set with finite value. It is even possible that $\dom(\networkinducednew{N}{g}{\co}{V})=\emptyset$.

In the special case that the directed network is bipartite with the edges directed from $V$ to $U$, we can also consider this as an undirected weighted bipartite graph $G$ and call the corresponding operation \emph{induction by bipartite graphs}.
In this case, we drop the dependence on $V$ and just use the notation $\tempinducednew{G}{g}{\co}$.
\end{definition}
\begin{theorem}[{\cite[Theorem~9.27]{Murotabook}}]\label{thm:induction-network}
  Let $N, g, \co$ and $V$ as in Definition~\ref{def:induction-by-networks}.
  If $\dom(\networkinducednew{N}{g}{\co}{V})$ is non-empty, the induced function is a valuated matroid.
\end{theorem}

\begin{figure}
\centering\includegraphics[width=0.25\textwidth]{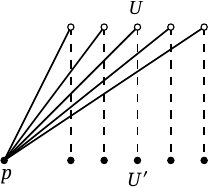}
\caption{Given a valuated matroid $g$ on $U$ and $w \in (\Trop)^U$, the principal extension $g^w$ is realized as the induction of $g$ via the above bipartite graph, as given in Remark~\ref{rem:principle+extension}.
The dashed edges are weighted zero, while the solid edges $(p,u)$ are weighted $w_u$.}
\label{fig:principal-ext-graph}
\end{figure}

While it is a special case of induction by networks, induction by bipartite graphs is an extremely powerful operation.
Many of the operations introduced so far can be modelled using induction by bipartite graphs, which is a key observation in the proof of Theorem~\ref{thm:VM+closed+operations}.

\begin{remark} \label{rem:principle+extension}
Recall from Definition~\ref{def:M+operations} the principal extension $g^w$ of a valuated matroid $g$ on ground set $U$ with respect to a weight vector $w \in \Trop^U$.
  We write the ground set of $g^w$ as $U' \cup \{p\}$ where $p$ is a new element and $U'$ a copy of $U$, with a copy $u'\in U'$ for each $u\in U$.
  Let $G = (V, U; E)$ be the bipartite graph with $V = U' \cup \{p\}$ where the edge set $E$ consists of $(u',u)$ and $(p,u)$ for all $u \in U$, and the weight function $\co$ takes the value zero on $(u',u)$ and $w_u$ on $(p,u)$; this graph is displayed in Figure~\ref{fig:principal-ext-graph}.
  Then the principal extension $g^w$ is precisely $\tempinducednew{G}{g}{\co}$, the induction of $g$ through $G$ onto $U' \cup \{p\}$.
  More details why this holds are provided in Lemma~\ref{lem:induction-principle+extension}.
\end{remark}

Furthermore, the following lemma shows we can realize induction by a network as induction by a bipartite graph followed by a contraction.
Given the power of this operation, it shall be a key construction throughout.

\RestateInit{\restatenetworkbipartite}
\begin{restatable}{lemma}{networkbipartite} \label{lem:network-bipartite+contraction}\RestateRemark{\restatenetworkbipartite}
  Let $N$ be a directed network with weight function $d$ and $g$ a valuated matroid on $U$ such that $f = \networkinducednew{N}{g}{d}{V}$ is again a valuated matroid. 
  Then there is a bipartite graph $G = (V \cup W, U \cup W';E)$ with weight function $\co \in \R^{V \cup W}$ and a valuated matroid $h$ on $U \cup W'$ such that $f = (\tempinducednew{G}{h}{\co}) / W$. 
\end{restatable}
\RestateGo{\restatenetworkbipartite}

We end this section by stating that valuated matroids are closed under all the operations introduced so far.
We defer the proof of this and the previous lemma to Appendix~\ref{sec:VM-operations}.

\RestateInit{\restateoperations}
\begin{restatable}{theorem}{operations} \label{thm:VM+closed+operations}\RestateRemark{\restateoperations}
  The class of valuated matroids is closed under the operations deletion, contraction, dualization, truncation, principal extension, direct sum, matroid union. 
\end{restatable}
\RestateGo{\restateoperations}

\section{Classes of valuated matroids} 
\label{sec:completeClasses}

In the following, we consider certain classes of valuated matroids that arise naturally in combinatorial optimization.

\begin{enumerate}[label=(\roman*)]
\item The class of \emph{transversally valuated matroids} are those valuated matroids arising from trivially valuated free matroids by induction via a bipartite graph.
\item The class of \emph{valuated gammoids} are contractions of those valuated matroids arising from trivially valuated free matroids by induction via a bipartite graph. 
\item The class of \emph{R-induced valuated matroids} are those valuated matroids arising from trivially valuated matroids by induction via a bipartite graph. 
\item The class of \emph{N-induced valuated matroids} are those valuated matroids arising from trivially valuated matroids by induction via a network.
\item The class of \emph{R-minor valuated matroids} are those valuated matroids arising as contractions of R-induced valuated matroids. 
\end{enumerate}

\begin{figure}
\centering
\includegraphics[width=0.7\textwidth]{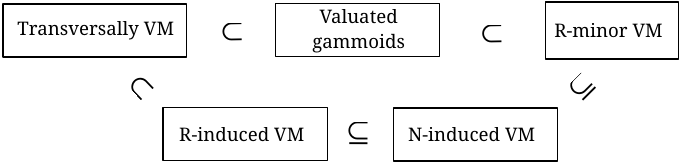}
\caption{The inclusion relationship between classes of valuated matroids.}
\label{fig:class+containment}
\end{figure}

Transversally valuated matroids are essentially the layers of assignment valuations, see also Example~\ref{ex:transversal+VM}. 
They were extensively studied in~\cite{FinkOlarte:2021}, which also considered the class of valuated strict gammoids, a subclass of valuated gammoids, 
from the perspective of tropical geometry. 

The inclusion relationship between these classes is shown in Figure~\ref{fig:class+containment}.
These are laid out in the following example and lemmas.

\begin{example}\label{ex:snowflake}
Consider the valuated matroid on six elements of rank two that takes the value $-\infty$ on $\{12,34,56\}$, and $0$ on all other pairs of elements.
This valuated matroid, referred to as the ``Snowflake'', has been studied in tropical geometry: in particular it is not a transversally valuated matroid as shown in \cite[Example 3.10]{FinkOlarte:2021}.
However, it is both a valuated gammoid and an R-induced valuated matroid, as given by the representations in Figure~\ref{fig:snowflake}.
\end{example}
\begin{figure}
\centering
\includegraphics[width=0.7\textwidth]{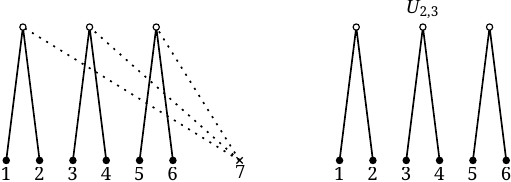}
\caption{Two representations of the Snowflake, defined in Example~\ref{ex:snowflake}.
The left is a valuated gammoid representation, where the element $7$ is contracted.
The right is an R-induced representation with induced matroid $U_{2,3}$.
All edges are weighted zero.}
\label{fig:snowflake}
\end{figure}

\begin{lemma}
  The class of valuated gammoids forms a strict subclass of R-minor valuated matroids.
\end{lemma}
\begin{proof}
  Containment is given by Theorem~\ref{thm:complete+class}.
  By \cite[Lemma 1]{OstrovskyPaesLeme:2015}, valuated gammoids are \emph{(strongly) base orderable}.
  However, any trivially valuated matroid that is not base orderable is an R-induced valuated matroid, giving strict containment.
\end{proof}

\begin{lemma}
  The class of R-induced valuated matroids forms a subclass of N-induced valuated matroids and a subclass of R-minor valuated matroids.
  Furthermore, N-induced valuated matroids form a subclass of R-minor valuated matroids. 
\end{lemma}
\begin{proof}
  The inclusion of R-induced within N-induced and R-minor are immediate from definition.
  Furthermore, Lemma~\ref{lem:network-bipartite+contraction} shows how to represent an N-induced valuated matroid as an R-minor valuated matroid.
\end{proof}

The strictness of the inclusion between N-induced valuated matroids and R-minor valuated matroids remains unresolved. 
From an algorithmic point of view, it would be desirable for N-induced valuated matroids to exhibit concise representations in the spirit of the small representation of gammoids in~\cite{KratschWahlstroem:2020}; see~\cite[Section~39.4a]{schrijver2003combinatorial} for more on transversal matroids and their contractions, the gammoids. 
\begin{conjecture}
  Let $N=(T, A)$ be a directed network with a weight function $\co\in \R^A$. 
  Let $V, U\subseteq T$ be two non-empty subsets of nodes of $N$.
  Let $g$ be a valuated matroid on $U$ of rank $d$.

  Then there is a directed network $N' = (T',A')$ with $U, V \subseteq T'$, and arc weights $\co' \in \R^{A'}$ 
  such that $\networkinducednew{N}{g}{\co}{V} = \networkinducednew{N'}{g}{\co'}{V}$ and such that $|T'|$ is polynomial in $|V|$ and $|U|$.
\end{conjecture}

As we show in Appendix~\ref{sec:size-bound-R-rep}, R-induced valuated matroids have a polynomial size representation.
However, the information-theoretic argument given does not extend to N-induced and R-minor valuated matroids as it cannot control the size of the contracted set.
This suggests that several of the inclusions in Figure~\ref{fig:class+containment} should indeed be strict.

\subsection{Complete classes}

\begin{definition}[Complete class]
  Let $\mathcal{V}$ be a subset of the set of valuated matroids. We call $\mathcal{V}$ a \emph{complete class} if it is closed under taking restriction, duals, direct sum and principal extension. 
\end{definition}

The original definition of a complete class of matroids included many other operations including contraction, truncation and induction.
We extend~\cite[Theorem 6.1]{BoninSavitsky:2016} stating that these four operations suffice to define a complete class for valuated matroids.

\begin{theorem} \label{thm:complete+class}
  A complete class of valuated matroids is closed under taking contraction, truncation, induction by bipartite graphs, induction by directed graph and valuated union.
  
  Furthermore, valuated gammoids forms the smallest complete class.
  Hence, they are contained in all complete classes. 
\end{theorem}
\begin{proof}
  The first claim follows from Lemma~\ref{lem:contraction-dual+deletion}, Lemma~\ref{lem:truncation-extension+contraction}, Lemma~\ref{lem:induction-principle+extension}, Lemma~\ref{lem:union-induction+sum} and Lemma~\ref{lem:network-bipartite+contraction}. 

  A non-empty complete class must contain the free matroid on one element.
  By taking iterated direct sum, this yields all free matroids. 
  Then closure under induction by bipartite graphs and minors yields valuated gammoids.
\end{proof}

\subsection{R-minor valuated matroids} \label{sec:R-induced}
The classes of valuated matroids discussed in the beginning of this section arising from induction through a network may only be induced by trivially valuated matroids.
As discussed in Example~\ref{ex:trivially-valuated-matroids}, a trivially valuated matroid $g$ can be identified with its underlying matroid $\M$, where $g(X)$ takes the value zero on bases of $\M$ and $-\infty$ otherwise.
Working with this underlying matroid shall be more convenient much of the time, therefore we extend the notation of Definition~\ref{def:induction-by-networks} to define $\networkinducednew{N}{\M}{\co}{V} := \networkinducednew{N}{g}{\co}{V}$.

Let $f$ be an R-minor valuated matroid on $V$.
By definition, there exists an R-induced valuated matroid $\tilde{f}$ on $V \cup W$ such that $f = \tilde{f}/W$.
By definition, there exists some bipartite graph $G = (V \cup W, U; E)$ with edge weights $\co \in \R^E$ and matroid $\M = (U, r)$ such that $\tilde{f} = \tempinducednew{G}{\M}{\co}$; we say $\tilde{f}$ has an \emph{R-induced representation} $(G, \M, \co)$.
As $f = \tempinducednew{G}{\M}{\co}/W$, we extend this notation to say that $f$ has an \emph{R-minor representation} $(G,\M, \co, W)$, where $W$ is the set to be contracted.

In the following, we show that R-minor valuated matroids are closed under deletion, principal extension, duality and direct sum, making them a complete class.
In the following we shall assume $f$ is an R-minor matroid with representation $(G, \M, \co, W)$ as above.

\begin{lemma}\label{lem:R-deletion}
  For a subset $X \subseteq V$, let $G \setminus X$ be the graph obtained from $G$ by deleting the nodes $X$ and all edges adjacent.
  The deletion $f \setminus X$ is represented by $(G\setminus X, \M, \co, W)$.
\end{lemma}
\begin{proof}
  This follows by the definition of deletion. 
\end{proof}

Let $w \in (\Trop)^V$ and consider $f^w$.
Let $V', W'$ denote copies of $V,W$, and define a network $N = (T,A)$ on the node set $T = (V'\cup W' \cup \{p\}) \cup (V \cup W) \cup U$, where $p$ is a new node.
We write $\widehat{V} = V'\cup W' \cup \{p\}$ to ease notation.
The arc set $A$ weighted by $\co' \in \R^A$ consists of arcs $(v',v)$ with weight $0$, where $v' \in V'\cup W'$ denotes the copy of $v \in V\cup W$.
We also add arcs $(p,v)$ with weight $w_{v}$ for all $v \in V$, and arcs $(v,u)$ for all edges $E$ of $G$ with weight inherited by $\co$.
The constructed network $N$ is displayed in Figure~\ref{fig:R-principal-extension}.
This network can intuitively be thought of as the ``concatenation'' of $G$ with the graph from Remark~\ref{rem:principle+extension}.
\begin{figure}
\centering
\includegraphics[height=4cm]{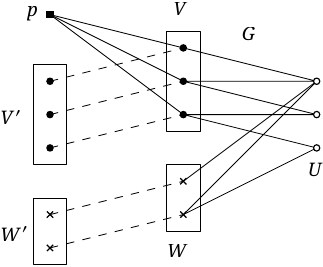}
\caption{The network $N$ constructed from a graph $G$ inducing the principal extension of an R-minor valuated matroid, as described before and within Lemma~\ref{lem:R-principal-extension}.}
\label{fig:R-principal-extension}
\end{figure}
\begin{lemma}\label{lem:R-principal-extension}
  The principal extension $f^w$ arises as the contraction of $\networkinducednew{N}{\M}{\co}{\widehat{V}}$ by $W'$.
  In particular, it can be represented as an R-minor valuated matroid. 
\end{lemma}
\begin{proof}
  Consider a subset $X \subseteq V \cup \{p\}$, the principal extension $f^w$ is defined as
  \[
  f^w(X) = (\tilde{f}/W)^w(X) = 
  \begin{cases}
  \max_{v \in V \setminus Y}\left(\tilde{f}(Y \cup v \cup W) + w_v\right) & \, X = Y \cup \{p\} \\
  \tilde{f}(X \cup W) & \, p \notin X \\
  \end{cases} \, .
  \]
  We claim that $\networkinducednew{N}{\M}{\co}{\widehat{V}}(X' \cup W') = f^w(X)$ for $X' \subseteq V' \cup \{p\}$.
  
  If $p \notin X'$, then the value of $\networkinducednew{N}{\M}{\co}{\widehat{V}}(X' \cup W')$ is simply the maximal independent matching in $G$ to $X\cup W$ with no contribution from the zero edges, i.e. $\networkinducednew{N}{\M}{\co}{\widehat{V}}(X') = \tilde{f}(X\cup W)$.
  If $X' = Y' \cup \{p\}$, then the value of $\networkinducednew{N}{\M}{\co}{\widehat{V}}(X'\cup W')$ is the maximal independent matching in $G$ to $Y \cup v \cup W$ for some $v \in V \setminus Y$, plus $w_v$ picked up from the arc $(p,v)$, i.e.
  \[
  \networkinducednew{N}{\M}{\co}{\widehat{V}}(X'\cup W') = \max_{v \in V \setminus Y}\left(\tilde{f}(Y \cup v \cup W) + w_v\right) \, ,
  \]
  which is precisely the value of $f^w$.
  Therefore $f^w = \networkinducednew{N}{\M}{\co}{\widehat{V}}/W'$.
  Applying Lemma~\ref{lem:network-bipartite+contraction}, we can represent $\networkinducednew{N}{\M}{\co}{\widehat{V}}$ as an R-minor valuated matroid, and therefore also $f^w$. 
\end{proof}

Consider the dual valuated matroid $f^*$, we claim it can be represented in the following way.
Let $U', V', W'$ be copies of $U, V, W$ respectively.
Let $G' = (U \cup V \cup W, U' \cup V' \cup W', E')$ whose edge set $E'$ consists of edges
\[
E' = \SetOf{(v,v')}{v \in U \cup V \cup W} \cup \SetOf{(u, v')}{(v,u) \in E} \, .
\]
The edge weights are given by $\co' \in  \R^{E'}$ where $\co'(v,v') = 0$ and $\co'(u,v') = \co(v,u)$.
This graph is displayed in Figure~\ref{figure:dualRado}.
We also use in our representation the matroid $\M' = \M^* \oplus \free{V' \cup W'}$, the direct sum of the dual matroid $\M^* = (U', r^*)$ and the free matroid on $V' \cup W'$.

\begin{lemma}\label{lem:R-dual}
  The dual $f^*$ is an R-minor valuated matroid.
\end{lemma}
\begin{proof}
Let $f = \tilde{f}/W$, then its dual is $f^* = (\tilde{f}/W)^* = (\tilde{f})^* \setminus W$ by Lemma~\ref{lem:contraction-dual+deletion}.
As R-minor valuated matroids are closed under deletion by Lemma~\ref{lem:R-deletion}, we are done if we can show $(\tilde{f})^*$ is an R-minor valuated matroid.
We claim that $(\tilde{f})^*$ is represented by $(G',\M', \co', U)$.

\begin{figure}
  \centering
  \includegraphics[width = 0.9\textwidth]{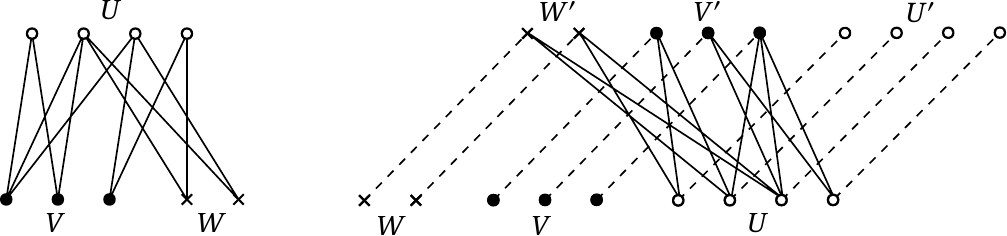}
  \caption{Construction of R-minor representation for $f^*$. }
  \label{figure:dualRado}
\end{figure}

Fix some $X \subseteq V$, we shall compute $\tempinducednew{G'}{\M'}{\co'}(X\cup U)$.
First observe that $v \in X$ can only be matched to $v' \in X'$ with weight zero, and that there are no matroid constraints on these edges.
Therefore the rest of the matching is an independent matching from $U$ to $(U' \cup V' \cup W') \setminus X'$.
For any independent matching, $Y \subseteq U$ matches to $(V' \cup W') \setminus X'$ if and only if $U' \setminus Y'$ is independent in $\M*$, which by matroid duality only occurs when $Y$ is independent in $\M$.
Therefore all independent matchings are of the form
\[
\SetOf{(u,v')}{(v,u) \in \mu} \cup \SetOf{(v,v')}{v \in X \cup (U \setminus Y)}
\]
where $\mu$ is an independent matching in $G$ from $(V \cup W) \setminus X$ to $Y \subseteq U$.
As the weights of these edges are either $0$ or inherited from $G$, we have
\[
\tempinducednew{G'}{\M'}{\co'}(X \cup U) = \tilde{f}((V \cup W)\setminus X) = (\tilde{f})^*(X) \, ,
\]
implying that $(\tilde{f})^* = \tempinducednew{G'}{\M'}{\co'}/U$ as claimed.
As $U, W$ are disjoint, contracting and/or deleting them commute and so $f^*$ has the representation $(G' \setminus W, \M', \co', U)$; the same representation as $(\tilde{f})^*$, but with $W$ deleted from $G'$.
\end{proof}

\begin{lemma}\label{lem:R-sum}
  Let $f_1$ and $f_2$ be two R-minor valuated matroids represented by $(G_1,\M_1, \co_1, W_1)$ and $(G_2,\M_2, \co_2, W_2)$.
  Then $f_1 \oplus f_2$ is represented by $(G',\M_1 \oplus \M_2, \co', W_1 \cup W_2)$, where $G'$ and its weight function $\co'$ arises by taking the union of the weighted graphs $G_1$ and $G_2$. 
\end{lemma}
\begin{proof}
  This follows from the definitions. 
\end{proof}

\begin{theorem}
  The set of R-minor valuated matroids forms a complete class of valuated matroids. 
\end{theorem}
\begin{proof}
  This follows directly from Lemmas~\ref{lem:R-deletion},~\ref{lem:R-principal-extension},~\ref{lem:R-dual} and \ref{lem:R-sum}.
\end{proof}

\section{Rado representation of matroids}
\label{section:structureOfMatroids}
Our proof of Theorem~\ref{thm:non-r-minor-main} 
  relies heavily on the properties of (unvaluated) matroids arising from the valuated matroids in $\cF_n$.
More precisely, the proof relies on the properties of the R-induced and R-minor representations of the underlying matroids.
In this section, we specialize R-representations to matroids (without valuation) and give several lemmas that will be used later in Section~\ref{sec:R-minor+not+cover}.
For a representation of a matroid,
we do not care about the weights of the edges in the bipartite graph but only the existence of the independent matchings.
Then the representation boils down to well-known results in matroid theory. 
This allows us to deduce useful structural statements and uncrossing properties.

\begin{definition}[Rado representation]\label{def:RadoRepresentation}
Let $G=(V, U; E)$ be a bipartite graph and $\M=(U, r_{\M})$ be a matroid.
We define a matroid $\cN$ on $V$ as follows. 
A set $X\subseteq V$ is independent in $\cN$ if there exists $S\subseteq U$ such that there is a perfect matching in the subgraph induced by $(X, S)$ and $S$ is independent in $\M$. 
We say that $(G, \M)$ is a \emph{Rado representation} of $\cN$.
\end{definition}

The following theorem verifies that this construction indeed defines a matroid, and characterizes its rank function. 

\begin{theorem}[Rado's theorem~\cite{oxley,Rado1942}]\label{thm:RadosTheorem}
Let $\cN$ be as in Definition~\ref{def:RadoRepresentation}.
Then $\cN$ is a matroid.
Moreover, a set $X\subseteq V$ is independent in $\cN$ if and only if $r_{\M}(\Gamma(Y)) \ge |Y|$ for all $Y\subseteq X$. If a set $X\subseteq V$ is a circuit in $\cN$, then $r_{\M}(\Gamma(X))=|X|-1$.
\end{theorem}

A more general representation can be obtained as minors of the above.
\begin{definition}[Rado-minor representation]\label{def:RadoMinorRepresentation}
Let $G=(V\cup W, U; E)$ be a bipartite graph and $\M=(U, r_{\M})$ be a matroid.
We define a matroid $\cN$ on $V$ as follows. 
A set $X\subseteq V$ is independent in $\cN$ if there exists $S\subseteq U$ such that there is a perfect matching in the subgraph induced by $(X\cup W, S)$ and $S$ is independent in $\M$. 
We say that $(G, \M, W)$ is a \emph{Rado-minor representation} of~$\cN$.
\end{definition}

\begin{proposition}\label{prop:RadoMinorRepresentation}
Let $\cN$ be as in Definition~\ref{def:RadoMinorRepresentation}.
Then $\cN$ is a matroid. 
Moreover, $X\subseteq V$ is independent in $\cN$ if and only if for all $Z\subseteq X\cup W$ it holds that $r_{\M}(\Gamma (Z)) \ge |Z|$.
If a set $X\subseteq V$ is a circuit in $\cN$, then there exists $Z' \subseteq X\cup W$ with $X = Z' \cap V$ such that $r_{\M}(\Gamma(Z'))=|Z'|-1$.
\end{proposition}
\begin{proof}
  Consider $G, W$ and $\M$ as in Definition~\ref{def:RadoMinorRepresentation}. 
  Then, let ${\cN}'$ be the matroid on $V\cup W$ with Rado representation $(G, \M)$. 
  By definition of contraction, $\cN$ can be obtained by contracting $W$ in~${\cN}'$.
  The first two parts of the proposition follow immediately from Theorem~\ref{thm:RadosTheorem}.
  
  For the final statement, note that if $X$ is a circuit of $\cN$, there exists a circuit $Z'$ of $\cN'$ such that $X = Z' \cap V$~\cite[Proposition 3.1.10]{oxley}.
  Applying Theorem~\ref{thm:RadosTheorem} gives the final part of the proposition.
\end{proof}

Any matroid that has no independent sets other than the empty set is said to be an \emph{empty matroid}. 
We next introduce some basic matroidal notions, and present their properties in the context of Rado representations.
Recall the notions of matroid union and direct sum motivating Definition~\ref{def:valuatedSumUnion}.
We restate them here for clarity, along with some additional related notions for decomposing matroids.

\begin{definition}[Direct sum, Disconnected]
\label{def:matroidSum}
Let $U_1, \dots, U_k$ be disjoint sets. 
For $i \in [k]$ let $\B_i$ be the bases of matroid $\M_i$ on $U_i$. 
The \emph{direct sum} $\M_1 \oplus \dots \oplus \M_k$ is a matroid $\M$ on $U = \dot{\cup}_{i = 1}^k U_i$ with bases $\B = \{\dot{\cup}_{i=1}^{k} B_i : B_i \in \B_i\}$.
We say that a matroid $\M$ is \emph{disconnected} if it is a \emph{direct sum} of at least two non-empty matroids. 
A matroid is \emph{connected} if it is not \emph{disconnected}.
\end{definition}

\begin{definition}[Matroid union, Fully reducible]
For $i \in [k]$, let $\I_i$ be the set of independent sets of the matroid $\M_i$ on $U$. 
The \emph{matroid union} $\M_1 \vee \dots \vee \M_k$ is a matroid $\M$ on $U$ with
independent sets given by 
$\I = \{\cup_{i=1}^{k} I_i : I_i \in \I_i\}$.
If additionally $r(\M)=\sum_{i=1}^k r(\M_i)$, then $\M$ is the \emph{full-rank matroid union} of $\M_1, \dots ,\M_k$~(see~\cite{ChoeOxleySokalWagner:2004}). 

We say that a matroid $\M$ is \emph{reducible} if it is a \emph{matroid union} of at least two non-empty matroids. 
Further, $\M$ is \emph{fully reducible} if it is a \emph{full-rank matroid union} of at least two non-empty matroids.
\end{definition}

We will use the rank formula of matroid union, see e.g. \cite[Theorem 13.3.1]{frankbook}.
\begin{theorem}[Edmonds and Fulkerson, 1965]\label{thm:matroid-union}
Consider the matroid union
$\M=\M_1 \vee \dots \vee \M_k$ for matroids $\M_1,\ldots,\M_k$, $k\ge 2$ on the ground set $U$, and let $r_i$ denote the rank function of the $i$-th matroid. Then for any $X\subseteq U$, the rank $r(X)$ in $\M$ equals
\[
r(X)=\min\left\{\sum_{i=1}^k r_i(Z)+|X\setminus Z|:\, Z\subseteq X\right\}\, .
\]
Consequently, if $X$ is a circuit in $\M$ then $\sum_{i=1}^k r_i(X)=|X|-1$.
\end{theorem}

\begin{lemma}\label{lemma:disconnectedGivesIrreducible}
  Let $\cN$ be a matroid with a Rado representation $(G,\M)$, where $G=(V,U; E)$ and $\M=(U, r)$.
  Assume that $\M=\M_1 \oplus \dots \oplus \M_k$, with $k \geq 2$, for matroids $\M_i =(U_i, r_i)$, and  $\Gamma(V)\cap U_i\neq\emptyset$ for every component of $\M$. Then, $\cN$ is reducible.
 Furthermore, if $r_{\cN}(V)=r_{\M}(U)$, then $\cN$ is fully reducible.  
\end{lemma}
\begin{proof}
Let ${\cN}_i$   be the matroid with Rado representation $(G_i, \M_i)$,
 where $G_i=(V,U_i; E_i)$ and $E_i$ is the set of edges between $V$ and $U_i$.
 Then each independent set in $\cN$ arises from a matching that is composed of matchings on the $G_i$.
 Hence, the definition of matroid union yields ${\cN} = {\cN}_1 \vee\ldots \vee {\cN}_k$. 
 By the assumption, $\Gamma(V)\cap U_i\neq\emptyset$ for each $i\in [k]$, hence each ${\cN}_i$ is a non-empty matroid.
 Using $k \geq 2$, then $\cN$ is reducible. 
 The second part follows from
 \[
 r_{\cN}(V) \leq r_{{\cN}_1}(V) + \dots + r_{{\cN}_k}(V) \leq r_{{\M}_1}(U_1) + \dots + r_{{\M}_k}(U_k) = r_{\M}(U) \enspace . \qedhere
 \]
\end{proof}

\subsection{Uncrossing properties for Rado-minor representation}\label{sec:uncrossing}

We now present some technical statements for Rado-minor representations that will be used in the proof of Theorem~\ref{thm:non-r-minor-main}.
Consider a matroid $\cN$ on ground set $V$ with Rado-minor representation $(G, \M, W)$ where  $G=(V\cup W, U; E)$ and $\M=(U,r)$.

For a subset $X$ of the ground set $V$ of $\cN$, 
we say that $Z\subseteq V\cup W$ is an $X$-set if $Z\cap V=X$.
For $Z\subseteq V\cup W$, let
\[
\rho(Z) := r(\Gamma(Z)) - |Z| \, .
\]
For an $X$-set $Z$, we give lower bounds on $\rho(Z)$ depending on the independence of $X$ in $\cN$. 
Throughout, we will use $X, Y$ for subsets of $V$;
and $Z, I, J$ for subsets of $V\cup W$, i.e., 
$X$-sets for some $X\subseteq V$ are denoted with letters $Z, I, J$.

\begin{lemma}\label{lem:rho-values}
The function $\rho : 2^{V\cup W}\to \Z$ defined above is submodular. 
Let $X\subseteq V$ and consider any $X$-set $Z$.
\begin{enumerate}
\myitem{(ind)}\label{i:rho-three} If $X$ is independent in $\cN$, then  $\rho(Z)\ge 0$. 
\myitem{(cir)}\label{i:rho-four} If $X$ is a circuit in $\cN$, then $\rho(Z) \ge -1$. 
Moreover, in this case there is an $X$-set $Z$ such that $\rho(Z) = -1$. 
\myitem{(span)}\label{i:rho-more} If $X$  contains a basis of $\cN$,
         then $\rho(Z)\ge r({\cN})-|X|$.
\end{enumerate}
\end{lemma}
\begin{proof}
The function $\rho$ is the difference of a submodular function $r(\Gamma(.))$ 
and a modular function~$|.|$, and thus it is submodular.
(For the submodularity of $r(\Gamma(.))$, see~\cite[Lemma 11.2.13]{oxley}.) 

\ref{i:rho-three} follows immediately from Proposition~\ref{prop:RadoMinorRepresentation}. 
Let us show~\ref{i:rho-four}.
Using again Proposition~\ref{prop:RadoMinorRepresentation}, we have $\rho(Z\setminus \{i\}) \ge 0$ for an arbitrary $i\in Z\cap V = X$.  
By adding an element to $Z\setminus \{i\}$, the $\rho$-value decreases by at most $1$
(the rank function is monotonic so cannot decrease, while the size increases by one).
It follows that $\rho(Z) \ge -1$.
Moreover, Proposition~\ref{prop:RadoMinorRepresentation} implies there exists an $X$-set where equality is attained.

For~\ref{i:rho-more}, let $B \subseteq X= Z \cap V$ be a basis. Then, using the monotonicity of $r(\Gamma(.))$ we have
\begin{eqnarray*}
    \rho(Z) = r(\Gamma(Z)) - |Z|  &=& r(\Gamma(Z)) - |B \cup (Z \cap W)| - |X \setminus B| \\
     &\geq&  r(\Gamma(B \cup (Z \cap W) )) - |B \cup (Z \cap W)| +|B|-|X| \\
&=&\rho(B \cup (Z \cap W)) +r({\cN}) - |X|\\
&\ge& r({\cN}) - |X|\,.
\end{eqnarray*}
Where the last inequality follows by applying \ref{i:rho-three} for the basis $B$ and $B$-set $B\cup (Z \cap W)$.
\end{proof}
Recall that for a matroid $\M = (U,r)$ and a subset $X \subseteq U$, the closure $\cl_\M[X]$ of $X$ is the maximal set containing $X$ whose rank equals the rank of $X$, i.e.
\[
\cl_{\M}[X] := \left\{x \in U : r(X \cup\{x\})= r(X) \right\} \, .
\] 
If $\M$ is clear from the context, we simply denote as $\cl[X]$. A set $X$ with $\cl[X]=X$ is called a \emph{closed set} or \emph{flat}.
\begin{lemma}\label{lem:rho-prop}
\label{rho-prop3} If $\rho(I)+\rho(J)=\rho(I\cup J)+\rho(I\cap J)$, then $\cl[\Gamma(I)]\cap \cl[\Gamma(J)]=\cl[\Gamma(I\cap J)]$.
\end{lemma}
\begin{proof}
As $\rho(I)+\rho(J)=\rho(I\cup J)+\rho(I\cap J)$, we have 
$r(\,\Gamma(I) \,)+r(\,\Gamma(J) \,)=r( \,\Gamma(I\cup J) \,)+r( \, \Gamma(I\cap J) \,)$.
Then trivially, 
\begin{equation}\label{eq:lemma6iii}
r(\,\cl[\Gamma(I)] \,)+r(\,\cl[\Gamma(J)] \,)=r(\,\cl[\Gamma(I\cup J)] \,)+r(\, \cl[\Gamma(I\cap J)] \,)\,. 
\end{equation}
On the other hand, we have
\begin{eqnarray*}
r(\,\cl[\Gamma(I)] \,)+r( \, \cl[\Gamma(J)]\,) &\ge& r(\, \cl[\Gamma(I)]\cup \cl[\Gamma(J)] \,)
        +r( \,\cl[\Gamma(I)]\cap \cl[\Gamma(J)]\,) \\
&\ge& r(\,\cl[\Gamma(I) \cup \Gamma(J)]\, )+r(\, \cl[\Gamma(I)\cap \Gamma(J)]\,) \\
&\ge& r(\, \cl[\Gamma(I \cup J)]\, )+r(\, \cl[\Gamma(I \cap J)] \,) \,. 
\end{eqnarray*}
\begin{sloppypar}
\noindent The first inequality follows by submodularity of $r$. 
The second inequality follows from the properties of the closure operator $\cl[\Gamma(I) \cap \Gamma(J)] \subseteq \cl[\Gamma(I)]\cap \cl[\Gamma(J)]$ and since $r(\,\cl[\Gamma(I)]\cup \cl[\Gamma(J)] \,) = r(\,\cl[\Gamma(I)\cup \Gamma(J)]\,)$ 
(here we used $\cl[\, \cl[\Gamma(I)]\cup \cl[\Gamma(J)]\, ] = \cl[\Gamma(I)\cup \Gamma(J)]$ ).
The third inequality follows from the properties of the neighbourhood function $\Gamma(I)\cup \Gamma(J) = \Gamma(I\cup J)$ and $\Gamma(I\cap J) \subseteq \Gamma(I) \cap \Gamma(J)$.
\end{sloppypar}

Thus, by~\eqref{eq:lemma6iii}, we have $r(\,\cl[\Gamma(I)]\cap \cl[\Gamma(J)]\,) = r(\,\cl[\Gamma(I\cap J)]\,)$.
Hence, $\cl[\Gamma(I \cap J)]$ is a closed set that is a subset of the closed set $\cl[\Gamma(I)] \cap \cl[\Gamma(J)]$, and both
$\cl[\Gamma(I \cap J)]$ and $\cl[\Gamma(I)] \cap \cl[\Gamma(J)]$ have the same rank. 
Thus, $\cl[\Gamma(I)]\cap \cl[\Gamma(J)]=\cl[\Gamma(I\cap J)]$.
\end{proof}

We will also use the following simple consequence of the definition of $\rho$.
\begin{lemma}\label{lem:rho-0-fully-matched}
Let $Z\subseteq V\cup W$ such that $\rho(Z)=0$.
Let
$\mu$ be any matching in $G$ such that $\partial_U(\mu)$ is independent in $\M$, and  $Z\subseteq \partial_{V\cup W}(\mu)$. 
Then, for every $(i,j)\in \mu$, $i\in Z$ if and only if $j\in \cl_\M[\Gamma(Z)]$.
\end{lemma}
\begin{proof}
It is immediate that if $i\in Z$ for $(i,j)\in \mu$, then $j\in \Gamma(Z)\subseteq \cl_\M[\Gamma(Z)]$. For the other direction, note that 
 there are $|Z|$ edges in $\mu$ between $Z$ and $\Gamma(Z)$. Since $0=\rho(Z)=r(\Gamma(Z))-|Z|$, these endpoints already have full rank $r(\Gamma(Z))$. 
Using that $\partial_\mu(U)$ is independent in $\M$,  there may not be any edges $(i,j)\in \mu$ with $i\in (V\cup W)\setminus Z$ and $j\in\cl_\M[\Gamma(Z)]$, completing the proof.
\end{proof}
Throughout we shall refer to the following uncrossing lemmas liberally. The first statement is the standard uncrossing argument for tight sets prevalent in combinatorial optimization. However, in our case one needs to require that $X\cup Y$ is also independent.

\begin{lemma}[Uncrossing I]
\label{lem:uncrossing}
    For $X, Y\subseteq V$ let $I, J \subseteq V\cup W$ be any $X$-set and any $Y$-set respectively, 
    and assume $\rho(I) = \rho(J) = 0$.
    If $X\cup Y$ is independent in $\cN$ then, 
    \[\rho(I\cap J), \rho(I\cup J) = 0 \,. \]
    In particular, if $X=Y$ for an independent set $X$ in $\cN$, and $\rho(I) = 0$ for some $X$-set $I$, 
    then there exists a unique largest maximal $X$-set $I$ with $\rho(I) = 0$.
\end{lemma}
\begin{proof}
    By submodularity, we have
    $
    0 = \rho(I) + \rho(J) \ge \rho(I\cap J) + \rho(I\cup J)\,.
    $
    By definition, $I\cap J$ is an $(X\cap Y)$-set and $I\cup J$ is an $(X\cup Y)$-set.
    Since both $X\cap Y$ and $X\cup Y$ are independent, 
    we have $\rho(I\cap J), \rho(I\cup J) \ge 0$ by Lemma~\ref{lem:rho-prop}.
    The first part follows.

    By the first part, the family of sets $I$ that are $X$-sets with $\rho(I) = 0$ is closed under intersection and union.
    If this family is non-empty then there exists a unique largest $X$-set $I$ with $\rho(I) = 0$.
\end{proof}

In other words, the above lemma states that the set of $X$-sets $I$ where $X$ is independent in $\cN$ and with $\rho$-value $0$ is a lattice over $V\cup W$ with respect to union and intersection (as are minimizers of a submodular function). 

Using $\rho(\emptyset) = 0$, we apply the uncrossing lemma to $X = Y = \emptyset$.
This yields the following corollary.
\begin{corollary}\label{i:W-0} 
There exists a unique largest set $Q \subseteq W$ such that $\rho(Q)=0$.
\end{corollary}

\begin{lemma}[Uncrossing II]\label{lem:uncrossing2}
    Let $X, Y \subseteq V$ be two different circuits in the matroid $\cN$ such that $|X \cup Y| = r(\cN) +2$ and $X \cup Y$ contains a basis.
    Consider an $X$-set $I$ and a $Y$-set $J$ with $\rho(I) = \rho(J) = -1$.
    Then, we have $\rho(I\cap J)=0$ and $\rho(I\cup J)=-2$.
\end{lemma}
\begin{proof}
Since $I\cap J$ is an $(X\cap Y)$-set and $X\cap Y$ is an independent set, we have $\rho(I\cap J) \ge 0$. 
Since $I\cup J$ is an $(X\cup Y)$-set and $X\cup Y$ contains a basis, applying \ref{i:rho-more} gives us $\rho(I\cup J) \ge -2$.
By submodularity we get
$-2=\rho(I)+\rho(J)\ge \rho(I\cap J)+\rho(I\cup J)\ge 0-2\, .$
Hence, the equalities $\rho(I\cap J)=0$ and $\rho(I\cup J)=-2$ hold.
\end{proof}

\subsection{Lov\'asz extension and the matroid of maximum weight bases}
We close this section by recalling some results on Lov\'asz extension and maximum weight bases of a matroid that will be useful when dealing with R-minor valuated matroids.
\begin{definition}[Lov\'asz extension]
Let $\M=(U,r)$ be a matroid. The \emph{Lov\'asz extension} $\hat r : \R^U \to \R$ of the rank function $r$ is defined for $\tau\in\R^U$ as the maximum $\tau$-weight of a basis of $\M$, i.e.
\[
\hat r(\tau) := \max \SetOf{\sum_{i \in B} \tau_i}{B \in \B(M)} \, .
\]
\end{definition}
For a given $\tau\in\R^U$, the value $\hat r (\tau)$ can be calculated by the following well-known characterization, 
see e.g., \cite[Theorem 5.5.5]{frankbook}.
\begin{lemma}
Let $\M=(U,r)$ be a matroid. 
For $\tau\in\R^U$, the \emph{Lov\'asz extension} $\hat r(\tau)$ equals
\[
\hat r(\tau)=r(U)\tau_{u_m}+\sum_{i=1}^{m-1} r(U_i)(\tau_{u_i}-\tau_{u_{i+1}})\, ,
\]
where we reordered $U=\{u_1,u_2,\ldots,u_m\}$ such that $\tau_{u_1}\ge\tau_{u_2}\ge\ldots\ge\tau_{u_m}$, and $U_i=\{u_1,\ldots,u_i\}$ for all $i\in [m]$.
\end{lemma}
In this context, we say that $S\subseteq U$ is a \emph{level set} of $\tau$ if $S=\emptyset$, $S=U$, or $S=U_i$ for some $i\in [m]$ with $\tau_{u_i}>\tau_{u_{i+1}}$.
Thus, the level sets of $\tau$ form a chain.
Using these level sets we can nicely capture all maximum weight bases in a matroid.
The following lemma follows from the greedy algorithm for finding maximum weight bases in a matroid.

\begin{lemma}[{\cite[Proposition 2]{ArdilaKlivans:2006}}]\label{lem:maxweight}
 For a matroid $\M=(U,r)$ and $\tau\in\R^U$, let
 $\emptyset=S_0\subsetneq S_1\subsetneq S_2\subsetneq \ldots S_t\subsetneq S_{t+1}=U$ denote the level sets of $\tau$. 
 Let us define the matroid
 \[
 \M_{\tau} := \bigoplus_{\ell = 1}^{t+1} \left( \restr{\M}{S_{\ell}} \right)/S_{\ell-1}\, .
 \]
 This is the matroid formed by the maximum $\tau$-weight bases of $\M$. That is, a basis $B$ in $\M$ maximizes $\sum_{i\in B}\tau_i$ if and only if $B$ is a basis in $\M_{\tau}$.
\end{lemma}

\section{Linear programming representation of R-minor functions}
\label{section:LPs}
R-induced valuated matroids are defined via independent matchings. 
Thus, the function value of a set can be naturally captured by a linear program.
Similarly, the set of all maximizers of an R-induced valuated matroid corresponds to the integral solutions of a linear program. 
Below, we obtain a description of all such integral solutions using the dual linear program and complementary slackness.

Throughout this section, unless stated otherwise, $f$ is an R-minor valuated matroid with representation $(G,\M,\co,W)$ given by a bipartite graph $G = (V\cup W, U; E)$, edge weights $\co \in \R^E$ and a matroid $\M=(U, r)$.
\begin{lemma}\label{lemma:LP1}
 For $X\subseteq V$, $f(X)$ is the objective value of the linear program 
\begin{equation}\label{prog:lp}
\begin{aligned}
   &\quad \max \quad \textstyle\sum_{(i,j)\in \E}    \co_{ij} x_{ij}   \\
   &\begin{aligned}
   \text{s.t.: } &&  \textstyle\sum_{j \in U} x_{ij}  &= \1_{i \in X \cup W}   && \forall i \in V\cup W \\
   &&  \textstyle\sum_{i \in V\cup W, j \in S}   x_{ij} &\le r(S)      && \forall S \subset U \\
   &&  \textstyle\sum_{i \in V\cup W, j \in U}   x_{ij} & = r(U)      && \\
   &&                x_{ij} &\ge 0        && \forall i \in V\cup W, \forall j \in U \,.
   \end{aligned} 
\end{aligned}
\end{equation}
Here, $\1_{i\in Z}$ is the indicator function of the set $Z$, taking value $1$ if $i \in Z$ and $0$ otherwise. 
\end{lemma}
\begin{proof} The formulation clearly gives a relaxation of the integer program defining the value of $f(X)$.
Using the total-dual integrality of polymatroid intersection, see~\cite[Theorem 46.1 and Corollary~41.12b]{schrijver2003combinatorial}, the existence of an integer optimal solution $x\in \Z^E$ is guaranteed; see the proof of Lemma~\ref{lemma:LP2} for more details. 
By the first set of constraints and since $\sum_{i\in V} x_{ij} \le r(\{j\}) \le 1$ for all $j\in U$, it is clear that $x=\chi_{\mu}$ for a matching $\mu$. Moreover, it holds $\partial_{V\cup W}(\mu) = X \cup W$ and $\partial_{U}(\mu)$ is a basis in $\M$. The lemma follows.
\end{proof}

We next characterize the set of maximizers of an R-minor valuated matroid. 

\begin{lemma}[Dual LP]\label{lemma:LP2}
Let $\B$ be the set of maximizers of $f$. 
Then $\B$ corresponds to the set of integral optimal solutions of 
\begin{equation}\label{prog:Fp}
\begin{aligned}
   &\quad \max \quad \textstyle\sum_{(i,j)\in \E}    \co_{ij} x_{ij}   \\
   &\begin{aligned}
   \text{s.t.: } &&  \textstyle\sum_{j \in U} x_{ij}  &\le 1   && \forall i \in V \\
   &&  \textstyle\sum_{j \in U} x_{ij}  &= 1   && \forall i \in W \\
   &&  \textstyle\sum_{i \in V\cup W, j \in S}   x_{ij} &\le r(S)      && \forall S \subset U \\
   &&  \textstyle\sum_{i \in V\cup W, j \in U}   x_{ij} & = r(U)      && \\
   &&                x_{ij} &\ge 0        && \forall i \in V\cup W, \forall j \in U \,.
   \end{aligned} 
\end{aligned}
\end{equation}
The dual of~\eqref{prog:Fp} is then 
\begin{equation}\label{prog:dualFp}
\begin{aligned}
   & \min \quad  \pi (V) + \pi(W)  + \hat r (\tau)  \\
   &\begin{aligned}
   \text{s.t.: } &&  \pi_i + \tau_j  &\ge \co_{ij}  && \forall (i,j) \in E \\
   &&         \pi_i&\ge 0  && \forall i\in V \\
   &&         \pi_i&-\text{free}  && \forall i\in W \\
   &&                \tau &-\text{free.}  && 
   \end{aligned} 
\end{aligned}
\end{equation}
Above, $\hat r$ is the Lov\' asz extension of the matroid rank function $r$. 
Let $(\pi, \tau) \in \R^{V\cup W \cup U}$ be an optimal dual solution. 
Let $E_0= \{(i,j) \in E : \pi_i + \tau_j  =  \co_{ij}\}$ be the set of tight edges, 
and $G_0 = (V\cup W, U; E_0)$ the tight subgraph.
Let $\emptyset=S_0\subsetneq S_1\subsetneq S_2\subsetneq \ldots S_t\subsetneq S_{t+1}=U$ be the level sets of $\tau$ in $U$, 
and denote with $\M_\tau$ the matroid of maximum weight bases. 
Let ${\cN}$ be the matroid on $V\cup W$ with bases $\{B\cup W:\, B\in \B\}$.
Then, $(G_0, \M_\tau)$ is a Rado representation of ${\cN}$.
We have $\pi_i = 0$ for all $i \in V$ for which there is a maximizer set $X\in\B$ with $i \not \in X$.

Further, the optimal solution $(\pi, \tau)$  can be chosen with  the following additional properties:
\begin{itemize}
  \item Every level set $S_\ell$, $\ell\in [t+1]$ is a flat in $\M$.
  \item For every $\ell\in [t+1]$, $(S_\ell\setminus S_{\ell-1})\cap \Gamma_{E_0}(V)\neq\emptyset$.
\end{itemize}
\end{lemma}
\begin{proof}
Observe that the problem is a special case of matroid intersection. We can define two matroids on the edge set $E$: a partition matroid enforcing that only one edge can be selected incident to every node in $V \cup W$, and a second matroid enforcing that the set of endpoints in $U$ must be independent in $\M$; this can be obtained from $\M$ by replacing every node $u\in U$ by parallel copies corresponding to the edges incident to $u$. By the integrality of polymatroid intersection~\cite[Theorem 46.1 and Corollary~41.12b]{schrijver2003combinatorial}, 
the set $\argmax \{ f(X) : X\subseteq V\}$ corresponds to the set of integral solutions of~\eqref{prog:Fp}.

The dual LP formulation can be easily derived from Frank's weight splitting theorem \cite[Theorem 13.2.4]{frankbook}, interpreted in this bipartite setting. The Rado representation of ${\cN}$ and the condition on the $\pi_i=0$ values follow by complementary slackness.

Let us now show that the additional properties can be ensured. 
Consider the smallest  level set $S_\ell$ that is not a flat. Thus, $S_\ell=\{i\in U:\, \tau_i\ge \lambda\}$ for some $\lambda\in \R$. Let us increase $\tau_j$ to $\lambda$ for every $j\in \cl[S_\ell]\setminus S_\ell$. By definition of the Lov\'asz extension, this does not change the value $\hat r(\tau)$; and since we only increase $\tau$, the solution remains feasible. After the change, $\cl[S_\ell]$ replaces $S_\ell$ as a level set. Thus, after at most $|U|$ such changes, we can guarantee that all level sets are flats. 

We show that this also implies the final property, i.e., that for every $i\in [t+1]$, there exists a tight edge $(i,j)\in E_0$ with $j\in S_\ell\setminus S_{\ell-1}$. Indeed, if no such edge exists, then we can decrease $\tau_{k}$ by some positive $\varepsilon>0$ for every $k\in S_\ell\setminus S_{\ell-1}$ such that $(\pi,\tau)$ remains feasible, and $S_\ell$ remains a level set, i.e. $\tau_k>\tau_{k'}$ for any $k\in S_{\ell}$, $k'\in S_{\ell+1}$. This decreases $\hat r(\tau)$ by $\varepsilon \left( r(S_\ell)-r(S_{\ell-1}) \right)>0$, a contradiction to optimality.
\end{proof}
Note that as an immediate corollary, the set of maximizers $\B$ is a matroid with Rado-minor representation $(G_0, \M_\tau, W)$.

\begin{lemma}\label{lem:pi-reducible}
Let $f$ be an R-\emph{induced} valuated matroid represented by $(G,\M, c)$ and $\B$ be the set of maximizers of $f$.
Consider a dual optimal solution $(\pi,\tau)$ as in Lemma~\ref{lemma:LP2}. 
If $\tau_i\neq\tau_j$ for some $i,j\in U$, then the matroid on $V$ defined by the bases $\B$ is  fully reducible.
\end{lemma}
\begin{proof}
By Lemma~\ref{lemma:LP2}, $(V,U;E_0)$ and $\M_\tau$ gives a Rado representation of the matroid with bases $\B$ 
(since $W=\emptyset$). 
For the flats $S_\ell$, $\M_\tau$ is the direct sum of the matroids
$\left(\restr{\M}{S_{\ell}} \right)/S_{\ell-1}$ (Lemma~\ref{lem:maxweight}).
Since all level sets $S_\ell$ are flats, each matroid  $\left(\restr{\M}{S_{\ell}} \right)/S_{\ell-1}$ is non-empty. If there are more than two terms, then Lemma~\ref{lemma:disconnectedGivesIrreducible} implies that $\B$ corresponds to a fully reducible matroid. Otherwise, the only flats can be $S_0=\emptyset$ and $S_1=U$; consequently, $\tau_i$ is the same for all $i\in U$.
\end{proof}

\section{R-minor functions do not cover valuated matroids}
\label{sec:R-minor+not+cover}
In this section we prove that no function in $\cF_n$ arises as an R-minor valuated matroid. 
Together with Appendix~\ref{section:functionH} this proves Theorem~\ref{thm:non-r-minor-main}.
The proof will be by contradiction, considering a carefully chosen minimal counterexample. We start by giving some definitions and the selection criteria for the minimal counterexample. We then outline a roadmap to the proof in Section~\ref{sec:roadmap-new}.

Recall that $\cF_n$ (Definition~\ref{def:class}) is a family of valuated matroids defined over the ground set $V = [2n]$,
using \emph{pairs} $P_i = \{2i - 1, 2i\}$ for $i\in [n]$. 
We let $\cH$ be the set of $P_i \cup P_j$ such that at least one of $i,j$ is even and we let $\Xs=P_1\cup P_2=\{1,2,3,4\} \in \cH$.
A function $h:{V \choose 4}\to \R\cup\{-\infty\}$ is in $\cF_n$ if and only if the following hold:
\begin{itemize}
  \item $h(X)=0$ if $X\in {V \choose 4}\setminus  {\cH}$,
  \item $h(X)<0$ if $X\in {\cH}$, and
  \item $h(\Xs)$ is the unique largest nonzero value of the function.
\end{itemize}

First, we introduce some notation and choose an appropriate minimal counterexample. 
Let us fix a value $n\ge 16$.
For a contradiction, let us assume there exists a valuated matroid $h\in\cF_n$ that is $R$-minor arising via a bipartite graph $G=(V\cup W,U;E)$, a matroid $\M=(U,r)$, and weights $\co\in \R^E$. 
Define
\[
\B_0\coloneqq \argmax(h) = {V\choose 4}\setminus \cH\, ,\quad\ \B_1\coloneqq\dom(h)\, .
\]
By Lemma~\ref{lemma:sparsePavingB} both $\B_0$ and $\B_1$ are (sparse) paving matroids. From the definition of $\cF_n$, we have $\B_0\cup\{\Xs\}\subseteq\B_1$.

\paragraph{Maximum weight matchings}
For every  $X\in \B_1$, (arbitrarily) fix a maximum weight independent matching $\mu^X$ with $\partial_{V\cup W}(\mu^X)=X\cup W$; let $\mathcal{L}$ be the set of all these matchings.
Thus, $c(\mu^X)=0$ if $X\in\B_0$ and $c(\mu^X)<0$ if $X\in \B_1\setminus \B_0$.
Define 
\[
E^*\coloneqq\{(i,j):\, (i,j)\in \mu^X \text{ for some }X\in\B_0\}
\]
as the union of all maximal independent matchings in $G$ of maximum weight.

\paragraph{Selection criteria for $h$}
As we assume that the set of valuated matroids in $\cF_n$ with an R-minor representation $(G, \M, c, W)$ is not empty, we can select an extremal one $h$ among them 
according to the following criteria:
\begin{enumerate}[label = (S\arabic*)]
  \item The function $h$ has minimal effective domain, that is, $|\B_1|$ is minimal.\label{crit:minsup}
  \item Subject to this, $|W|$ is minimal.\label{crit:minW}
  \item Subject to this, $|E\setminus E^*|$ is minimal.\label{crit:nontight}
\end{enumerate}
Note that \ref{crit:minsup} only depends on $h$, whereas 
\ref{crit:minW} and \ref{crit:nontight} also depend on the representation.

We will refer to this choice as the \emph{minimal counterexample}.
This choice is well-defined, since all criteria 
minimize over non-negative integers. 
For \ref{crit:minsup}, note that the extreme case is  $\B_1=\B_0\cup\{\Xs\}$; a key step in the proof is  to show that this must always be the case.
 
\begin{remark} To show that not all valuated matroids arise as R-minor valuated matroids, it would suffice to show Theorem~\ref{thm:non-r-minor-main} in a weaker form, only for functions $h\in \cF_n$ such that $h(X)=-\infty$ for $X\in \cH\setminus\{\Xs\}$. This would already postulate $\B_1=\B_0\cup\{\Xs\}$, enabling a slightly simpler proof. However, we need Theorem~\ref{thm:non-r-minor-main}  for the entire class $\cF_n$, because  in order to refute the MBV conjecture in Theorem~\ref{thm:mbv+counterexample}, we require a function in $\cF_n$ with finite values that is not R-minor.
\end{remark}

\paragraph{Dual solutions}
We will also select an optimal dual solution $(\pi,\tau)$ to \eqref{prog:dualFp} in Lemma~\ref{lemma:LP2}. Let us introduce some notation; the choice of the particular solution will be specified in Lemma~\ref{lem:E-E-0}.

Let $E_0=\{(i,j)\in E:\, \pi_i+\tau_j=\co_{ij}\}$ denote the set of tight edges. By complementarity, $E^*\subseteq E_0$ must hold for any optimal dual $(\pi,\tau)$. Note that  $\pi_i=0$ for every $i\in V$ by Lemma~\ref{lemma:LP2}, since for every $i\in V$ there is an optimal primal solution to \eqref{prog:Fp} matching a set $X\in {V\choose 4}\setminus \cH$ with $i\notin X$.

Recall that $\M_\tau$ denotes the matroid of the maximum $\tau$-weight bases as in Lemma~\ref{lem:maxweight}; we let $r_\tau$ denote its rank function. 
The bipartite graph $G=(V\cup W,U;E)$ and matroid $\M=(U,r)$ and $W$ give a Rado-minor representation of $\B_1$.
As $\B_0$ is the set of maximizers of $h$, it follows from Lemma~\ref{lemma:LP2} that 
$G_0=(V\cup W,U;E_0)$ and $\M_\tau=(U,r_\tau)$ and $W$ give a Rado-minor representation of $\B_0$.

For $Z\subseteq V\cup W$, we let $\Gamma(Z)$ and $\Gamma_0(Z)$ denote the set of neighbours of $Z$ in $U$ in the edge sets $E$ and $E_0$, respectively.
Furthermore, for $Z\subseteq V\cup W$ we define 
\begin{eqnarray*}
  \rho_0(Z) &\coloneqq& r_{\tau}(\Gamma_{0}(Z)) - |Z| \,,\\
  \rho_1(Z) &\coloneqq& r(\Gamma (Z)) - |Z| \,.
\end{eqnarray*}
Note that $\rho_1(Z)\ge \rho_0(Z)$ for every $Z\subseteq V\cup W$.
Finally, let $Q_0$ denote the unique largest subset of $W$ with $\rho_0(Q_0)=0$ as guaranteed by Corollary~\ref{i:W-0}.
Analogously, one could also define $Q_1$ as the unique largest subset of $W$ with $\rho_1(Q_1)=0$. However, using  selection criterion \ref{crit:minW}, it turns out that $Q_1=\emptyset$. The following lemma will be proved in Section~\ref{sec:two-cases}.
\begin{lemma}\label{lemma:noQ1}
  In the minimal representation, we have $\rho_1(Z) > 0$ for each non-empty $Z\subseteq W$.
\end{lemma}

\subsection{Roadmap to the proof}\label{sec:roadmap-new}
We derive a contradiction by a thorough combinatorial analysis of the structure and Rado representations of the matroids $\B_0$ and $\B_1$.
Our first key lemma shows that the primal and dual optimal solutions must be in either of two restricted configurations.
\begin{lemma}\label{lem:E-E-0}
The minimal counterexample can be selected to satisfy one of the following properties:
\begin{enumerate}[label = (C\Roman*)]
\item We can choose a dual optimal solution $(\pi, \tau)$ such that 
  $E=  E_0\cup \{(i',j')\}$ for an edge $(i',j')$ where $i' \in \Xs\cup W$, $\M_\tau=\M$, and $\B_1=\B_0\cup\{\Xs\}$.\label{case:samemat}
   \item $E=E_0 = E^*$ and $\M_\tau\neq\M$ for any dual optimal $(\pi, \tau)$.\label{case:alltight}
\end{enumerate}
\end{lemma}

Intuitively, the above lemma states that the difference between Rado-minor representations of $\B_0$ and $\B_1$ is either in the edge set only, or in the matroid on $U$ only.
In case \ref{case:samemat}, all bases in $\M$ have the same $\tau$-weight, and there is a single non-tight edge.
Further, $h(\Xs)$ is the only finite value outside $\B_0$.
In contrast, in case~\ref{case:alltight}, all edges are tight, but we need to work with two different matroids on $U$.

The proof in Section~\ref{sec:two-cases} exploits the selection criteria. First, using \ref{crit:nontight}, we argue that $E=E^*\cup \mu^{X^*}$, i.e., all edges are used in one of the optimal matchings or in the maximum weight independent matching on $X^*$. If $E=E^*$, we get a contradiction to case \ref{case:alltight} immediately. If $E\neq E^\star$, we can get \ref{case:samemat} by constructing another dual solution.

\paragraph{The base case $W=\emptyset$}
We first exclude the case that $W=\emptyset$,
that is, the case when $h$ is R-induced.
We tackle this case in Section~\ref{sec:W-0}. 
First, in case \ref{case:alltight},
 $\M_\tau\neq\M$ implies, using Lemma~\ref{lemma:disconnectedGivesIrreducible}, that $\B_0$ is \emph{fully reducible}, that is, it can be written as a full-rank matroid union of smaller matroids. An elementary argument in Lemma~\ref{lem:not-reducible} shows that  $\B_0$ is not fully reducible for $n\ge 10$, leading to a contradiction. The proof exploits the low rank of the matroid (4), and  the combinatorics of the pairs $P_i$ in the construction.

 Hence, \ref{case:samemat} must be the case.
We note that the set $\Xs=P_1\cup P_2$ does not have an independent matching in $E_0$ but has one in $E=E_0\cup \{(i',j')\}$.
The edge $(i',j')$ is incident to $\Xs$; say, $i'\in P_1$.
With an uncrossing argument using the submodularity of the rank of the neighbourhood function, we show that $(i',j')$ should create an independent matching also for another set $X=P_1\cup P_k\notin\B_0$. 
Since $\M=\M_\tau$ and this is the single non-tight edge, it follows that $0>h(Z)\ge h(\Xs)$, a contradiction that $h(\Xs)$ is the unique largest negative function value.

\paragraph{Robust matroids}
If $W\neq\emptyset$, we get a contradiction by showing the existence of a counterexample with smaller $W$.
Towards this goal, in Section~\ref{section:robustMatroids}, we define a common abstraction of $\B_0$ and $\B_1$.
We call them \emph{robust} matroids (Definition~\ref{def:robust}), and analyze their Rado representations. 
These are a class of sparse paving matroids of rank 4 with elements arranged in pairs $P_i$. 
All circuits are formed by the union of two pairs \ref{i:H-B}.

We call these matroids 
`robust' because the structure is robust against the `perturbation' from the set of bases $\B_0$ to $\B_1 = \B_0 \cup X^*$.
This is due to the fact that the non-bases of $\B_0$ have been carefully selected, ensuring that they are neither too sparse nor too uniform.
We define a robust matroid $\B$ by a similar selection process, ensuring that the set of non-bases $\cH$ have the correct density.
The crucial ingredients to this selection process are the properties \ref{i:H-graph2a} and \ref{i:H-graph2b}. 
They allow us to control the excess functions $\rho_0$ and $\rho_1$ (see Lemma \ref{lem:all-Z}).

We consider the Rado representation of such a robust matroid $\B$ and analyze the excess function $\rho(X)=r(\Gamma(X))-|X|$; statements derived in Section~\ref{section:robustMatroids}  will be valid for $\rho_1$ in the representation $G=(V\cup W,U;E)$ with $\M=(U,r)$ as well as for $\rho_0$ in $G_0=(V\cup W,U;E_0)$ with $\M_\tau=(U,r_\tau)$ (Lemma~\ref{lem:B-0-robust}).

The difficulty in the proof is to control the freedom in the set $W$.
While we are ultimately interested in the matroid on $V$, we also have to deal with its `preimage' before contraction in the bigger matroid in $V \cup W$.
As $\B$ is robust, the structure of the non-bases $\cH$ is sufficiently rigid that their `preimages' in $V \cup W$ also exhibit a similar structure.
As such, the analysis reveals that the structure of the pairs $P_i$ `forces itself' on the full representation.
This can be quantified with the excess function $\rho$.
  Let $Q\subseteq W$ be the unique largest set with $\rho(Q)=0$ (Corollary~\ref{i:W-0}).
In Lemma~\ref{lem:W-Z}, we show that for each pair $P_i$ there exists a unique largest `extension set' $Z_i\subseteq V\cup W$ with $\rho(Z_i)=0$ and  $Z_i\cap V=P_i$. 
 Moreover, for every $i\neq j$, we have $Z_i\cap Z_j=Q$, as well as 
$\rho(Z_i\cup Z_j)=0$ if $P_i\cup P_j\in\B$ and $\rho(Z_i\cup Z_j)=-1$ otherwise (Lemma~\ref{lemma:ZiZj}). The last property asserts that $Z_i\cup Z_j$ certifies $P_i\cup P_j\notin \B$ whenever this set is a circuit.

Let $Z^0_i$ and $Z^1_i$ denote the respective sets in the representations of the two matroids $\B_0$ and $\B_1$. For the rest, we argue that these two representations must be near-identical, while still corresponding to two different matroids.

In Section~\ref{sec:robust-to-rep}, we first show that $\supp(h)=\B_1=\B_0\cup\{\Xs\}$ always holds (Lemma~\ref{lem:B1-B0}). This is already known in \ref{case:samemat}; however, it is significantly more difficult to show in \ref{case:alltight}. We then argue that the $Z_i^1$ sets partition the ground set $V\cup W$ (Lemma~\ref{lem:unionZ}). Further, for each $i\in [n]$, $Z_i^0=Z_i^1\cup Q_0$ must hold; that is, the two extension sets may only differ in $Q_0$, the largest subset of $W$ with $\rho_0(Q_0)=0$. 

\paragraph{Completing the proof}
It remains to complete the analysis  for the two cases \ref{case:samemat}  and \ref{case:alltight}. In case~\ref{case:samemat} (Section~\ref{sec:samemat}), the set $Q_0$ plays a key role.
First, assume $Q_0=\emptyset$, i.e. $Z_i:=Z_i^0=Z_i^1$. Then, for each set $Z_i$, the neighbourhoods using edges in $E_0$ and the entire set $E$ should have the same rank; this basically leaves no room for the edge $(i',j')\in E\setminus E_0$.
 Hence, $Q_0\neq \emptyset$, and we must have  $Q_0\subseteq Z_q^1$ for a unique $q\in \{3,4,\ldots,n\}$. The final contradiction is reached by a submodular uncrossing argument.

Finally, in case~\ref{case:alltight} (Section~\ref{sec:alltight}), we can show $Q_0=\emptyset$ (Lemma~\ref{lem:no-Q0-CII}), and therefore $Z_i:=Z_i^0=Z_i^1$ for all $i\in [n]$. While the edge sets $E$ and $E_0$ are the same, the functions $\rho_1$ and $\rho_0$ may be still different as they correspond to  different matroids, $\M$ and $\M_\tau$.
We derive a contradiction to \ref{crit:minW} by showing that for some $Z_i$ with $|Z_i|>2$,  deleting an element of $W\cap Z_i$ and contracting a node in the matroid $\M$ leads to a smaller representation of $h$.

\subsection{Proofs of Lemma~\ref{lemma:noQ1} and Lemma~\ref{lem:E-E-0}}\label{sec:two-cases}
We now present two proofs deferred from the previous section. 

\begin{proof}[Proof of Lemma~\ref{lemma:noQ1}]
By~\ref{i:rho-three}, we must have $\rho_1(Z) \geq 0$ for all $Z\subseteq W$.
For a contradiction, let $Z\subseteq W$ be a non-empty set with $\rho_1(Z) = 0$.
Let $T = \cl[\Gamma(Z)]$, then $r(T) = |Z|$.
By Lemma~\ref{lem:rho-0-fully-matched}, 
for every edge $(i,j)$ in some maximal independent matching $\mu$, we have $i \in Z$ if and only if $j \in T$.
That is, the nodes in $Z$ can only be matched to nodes of $T$ and vice versa in maximal independent matchings.
Moreover, the weight of the edges covering $Z$ must be the same value $\delta$ for any maximal independent matching.
This follows since for any two maximal independent matchings $\mu,\mu'$, we can replace the set of edges covering $Z$ in $\mu$ by the set of edges covering $Z$ in $\mu'$ and obtain another independent matching covering $Z\cup X$. Let $\M' \coloneqq \M /T$ denote the matroid obtained from $\M$ by contracting $T$, and let $U'\coloneqq U\setminus T$.
Then, we obtain a smaller $R$-minor representation of $h$ by restricting to $W'\coloneqq W\setminus Z$, and using $\M'$ on $U'$.
Moreover, we define the new weight function on the edges as $\co'(i,j) \coloneqq \co(i,j) + \delta/r(\M')$ for each edge $(i,j)$ 
with $i\in V\cup W'$ and $j\in U'$ to obtain the same $h(X)$ values. 
This contradicts criterion \ref{crit:minW} whenever $Z \neq \emptyset$. 
\end{proof}

\begin{proof}[Proof of Lemma~\ref{lem:E-E-0}]
Let $\mu^{X^*}\in{\mathcal{L}}$ denote a maximum weight independent matching covering $\Xs{\cup}W$. 
First, we show that $E=E^* \cup \mu^{X^*}$. 
Indeed, removing an edge in $E\setminus (E^*\cup \mu^{X^*})$ does not affect $h(X)$ for $X\in \B_0\cup\{\Xs\}$ 
as all matchings $\mu^X$ for $X\in \B_0$ lie in $E^*$. 
For any other set, $h(X)$ may decrease (possibly to $-\infty$); 
but this would yield another function in $\cF_n$ that is the same or better on criterion \ref{crit:minsup},  
the same on \ref{crit:minW}, and strictly better on \ref{crit:nontight}.
Hence, $E=E^*\cup \mu^{X^*}$.

First, assume $E = E^*$. Then, $\mu^{X^*}\subseteq E^* \subseteq E_0$. Thus, $\partial_U(\mu^{X^*})$ cannot be independent in $\M_\tau$, as otherwise $h(\Xs)=0$ would follow by complementary slackness. 
Hence, $\M\neq\M_\tau$, giving case~\ref{case:alltight}.

Next, assume that $E \setminus E^* = \mu^{X^*}\setminus E^* \neq \emptyset$.
Let $(i',j')$ be an arbitrary edge in $\mu^{X^*}\setminus E^*$, i.e., $i' \in X^* \cup W$.
We start increasing $c$ to $c'$ for $\varepsilon \ge 0$ as follows
$$
c'_{ij} \coloneqq 
\begin{cases} 
c_{ij} + \varepsilon  \, \text{ for } (i,j) = (i', j') \\
c_{ij} \qquad \text{ otherwise.} 
\end{cases}
$$
Pick the largest $\varepsilon \ge 0$ such that the maximum weight of an independent matching in $G$, $\M$, $c$ remains $0$, i.e.,
such that the optimum value of the LP~\eqref{prog:Fp} does not change.

\begin{techclaim} 
 $\varepsilon = - h(X^*)$.
\end{techclaim}
\begin{subproof}
Suppose that $\varepsilon < -h(X^*)$.
By definition of $E^*$, 
we have stopped increasing $\varepsilon$ as the edge $(i',j')$ has now entered $E^*$
and increasing the value further would increase the optimal value via a set $X\in \B_0$. 
However, we still have $h(X^*)<0$, which contradicts~\ref{crit:nontight}. 
\end{subproof}
Next, we note that $\B_0 \cup \{X^*\}$ is the set of maximizers of LP~\eqref{prog:Fp} under the increased weights $c'$.
Indeed, by the choice of $\varepsilon$ all previous maximizers $\B_0$ remain maximizers and now $\mu^{X^*}$ achieves the same value thereby becoming a maximizer as well. Moreover, for $X\in \cH\setminus\{\Xs\}$, we have 
$c'(\mu^X)\le c(\mu^X)+\varepsilon<c(\mu^{\Xs})+\varepsilon=0$.

Let us pick an optimal  dual solution $(\pi, \tau)$ to~\eqref{prog:dualFp} under $c'$.
Recall that $E = E^*\cup \mu^{X^*}$ and therefore all edges $E$ are tight with respect to $c'$.
Since $c'\ge c$ and the optimum value is the same for the two cost functions, it follows that  $(\pi, \tau)$ is also optimal to~\eqref{prog:dualFp} with the original weights $c$.

Since $c$ and $c'$ differ only on $(i', j')$, 
all edges $E\setminus \{(i', j')\}$ are tight under $(\pi, \tau)$ for $c$; thus, $E_0 = E\setminus \{(i', j')\}$.

As $\partial_{U}(\mu^{X^*})$ is a maximum $\tau$-weight basis in $\M$, it follows that we can replace $\M$ by $\M_{\tau}$. 
This is because all $\mu^X \in {\mathcal{L}}$ for $X\in \B_0 \cup \{X^*\}$ remain independent matchings. 
The function value $h(X)$ might decrease for $X\not \in \B_0 \cup \{X^*\}$, but this may only lead to improvement in \ref{crit:minsup}, or otherwise we get another solution that is equally good on the selection criteria.

It is left to show $\B_1 = \B_0 \cup \{X^*\}$. Take any $X\in\B_1$. 
As we replaced $\M$ by $\M_\tau$ and every basis in $\M_\tau$ has maximum $\tau$-weight, the value of $c(\mu)$ is the optimum minus the sum of the slack values on the edges as given in the dual LP of Lemma~\ref{lemma:LP2}, that is,  
$h(X)=c(\mu^X)=-\sum_{(i,j)\in \mu} (\pi_i+\tau_j-c_{ij})$.
Since $(i',j')$ is the only edge with positive slack, this means that $h(X)=0$ if $(i',j')\notin \mu^X$ and $h(X)=h(\Xs)$ if $(i',j')\in \mu^X$. Since $\Xs$ is the unique set with the largest negative function value, this implies $\B_1 = \B_0 \cup \{X^*\}$.
This completes the analysis of the case $E\setminus E^*\neq\emptyset$, showing \ref{case:samemat} holds.
\end{proof}

\subsection{\texorpdfstring{$h$ is not R-induced  ($W=\emptyset$)}{h
    is not R-induced  (W =
    empty set)}}\label{sec:W-0}
\label{section:Rinduced}

We start by showing that $W=\emptyset$ is not possible; in other words, $h$ cannot have an $R$-induced representation.
We start with a structural claim on $\mathcal{B}_0$. 

\begin{lemma}\label{lem:not-reducible}
The matroid on $[2n]$ defined by bases $\mathcal{B}_0$ is not fully reducible for $n\ge 10$.
\end{lemma}
\begin{proof}
Let $r_0$ denote the rank function of the matroid with bases $\mathcal{B}_0$.
For a contradiction, assume this matroid is obtained as the union of two non-empty matroids $\M_1$ and $\M_2$ on $V=[2n]$ with rank functions $r_1$ and $r_2$, such that $r_1(V)+r_2(V)=4$. Recall that 
$r_0(X)=\min\{r_1(Z)+r_2(Z)+|X\setminus Z|:\, Z\subseteq X\}$ by Theorem~\ref{thm:matroid-union}. W.l.o.g. $r_1(V)\le r_2(V)$. We distinguish two cases.

\paragraph{Case I: $r_1(V)=1$, $r_2(V)=3$.} Let $T\coloneqq\{v\in V:\, r_1(\{v\})=0\}$ denote the set of loops in $\M_1$. 
We observe that $T$ may intersect at most three different pairs $P_i$. 
To see this, suppose $T$ intersects four distinct pairs $P_i, P_j, P_k, P_\ell$ and let $X = \{v_i, v_j, v_k, v_\ell\} \subseteq T$ be a subset of $T$ consisting of one element from each pair respectively.
Then $X\in \B_0$ but $r_0(X) \leq 3$ as each element of $X$ is a loop in $\M_1$, giving a contradiction.

Let us select four pairs $P_i$, $P_j$, $P_k$, $P_\ell$ that do not intersect $T$ such that $i$ and $j$ are odd and $k$ and $\ell$ are even; such selection is possible for $n\ge 10$ as the most restrictive scenario is when all three pairs intersecting $T$ have even index, requiring five pairs of even index in total. 
Since $P_i\cup P_j\in\B_0$, we must have $r_2(P_i\cup P_j)=3$; w.l.o.g. assume $r_2(P_i)=2$. 

Consider $P_i\cup P_k$ and recall that it forms a circuit in $\B_0$. 
By Theorem~\ref{thm:matroid-union}, $r_1(P_i\cup P_k)+r_2(P_i\cup P_k)=3$, implying $r_2(P_i\cup P_k)=2$. Similarly, $r_2(P_i\cup P_\ell)=2$. By submodularity, we have $r_2(P_i\cup P_k\cup P_\ell)=2$, and thus $r_0(P_i\cup P_k\cup P_\ell)=3$, a contradiction as the union of any three pairs contains a basis.

\paragraph{Case II: $r_1(V)=r_2(V)=2$.} Note that there can be at most one pair $P_t$ such that $r_1(P_t)=0$, and at most one pair $P_{t'}$ with $r_2(P_{t'})=0$.
Otherwise, if there existed $P_a, P_b$ such that $r_1(P_a)=r_1(P_b)= 0$, then $r_1(P_a \cup P_b) = 0$, contradicting that the union of any two pairs has rank at least 3 in $\B_0$.

Let us select $P_i$, $P_j$, $P_k$, $P_\ell$ such that $i$ is even, and $j$, $k$, and $\ell$ are odd, and all these pairs have rank $\ge 1$ in both matroids; again such sets can be selected for $n\ge 10$ as we require at most four pairs of odd index and two pairs of even index. 
Since $P_i\cup P_j$ is a circuit in $\B_0$, Theorem~\ref{thm:matroid-union} yields $r_1(P_i\cup P_j)+r_2(P_i\cup P_j)=3$. Similarly, $r_1(P_i\cup P_k)+r_2(P_i\cup P_k)=3$ and $r_1(P_i\cup P_\ell)+r_2(P_i\cup P_\ell)=3$. W.l.o.g. $r_1(P_i\cup P_j)=r_1(P_i\cup P_k)=1$. By the assumption $r_1(P_i)\ge 1$, submodularity gives $r_1(P_i\cup P_j\cup P_k)=1$. This again contradicts the fact that $r_0(P_i\cup P_j\cup P_k)=4$.
\end{proof}

\begin{lemma}\label{lem:M-M-tau}
If $W=\emptyset$, then we must have $\pi\equiv 0$ and $\tau\equiv0$ for the optimal dual $(\pi,\tau)$ in  \eqref{prog:dualFp}.
\end{lemma}
\begin{proof}
By definition of $h\in\cF_n$,  the optimum value of the LP \eqref{prog:Fp} is 0. Recall that 
 $\pi_i=0$ for all $i\in V$.
From Lemma~\ref{lem:pi-reducible}, it follows that $\tau_i$ has the same value for all $i\in U$; let $\alpha$ be this common value. Then, the objective value of the dual program \eqref{prog:dualFp} is $0=\alpha \cdot r({\M})$. Consequently, $\alpha=0$, and therefore $\tau=0$.
\end{proof}

Therefore $\M_\tau=\M$, implying case~\ref{case:samemat} of
Lemma~\ref{lem:E-E-0}:
$E=E_0\cup \{(i^*,j^*)\}$ and $\B_1=\B_0\cup\{\Xs\}$. The rest of the analysis is covered by the argument in Section~\ref{sec:samemat} for \ref{case:samemat}. We include a simpler direct proof that also illustrates some key ideas of the more complex subsequent arguments.

 Let $\ell\in \{1,2\}$ such that  $i^*\in P_\ell$.
We note that the cases $\ell = 1$ and $\ell = 2$ are not symmetric, because of different parity.

\begin{claim}\label{cl:P-1-rk}
We have $r(\Gamma(P_\ell))=3$.
\end{claim}
\begin{proof}
We first show that  $j^*\not\in \cl[\Gamma_0(X^*)]$. Indeed, since $X^*$ is a circuit in $\B_0$, $r(\Gamma_0(X^*))=3$ by Theorem~\ref{thm:RadosTheorem}. If $j^*\in \cl[\Gamma_0(X^*)]$, then $r(\Gamma(X^*))=3$ since $\Gamma(X^*)=\Gamma_0(X^*)\cup\{j^*\}$; thus, $X^*\notin \B_1$, a contradiction.

As $P_\ell \subset X^*$, it follows that $j^*\not \in \cl[\Gamma_0(P_\ell)]$. 
Since $P_\ell$ is a subset of a basis in $B_0$ we have $r(\cl[\Gamma_0(P_\ell))]) \ge 2$. Thus, $r(\Gamma(P_\ell)) = r(\Gamma_0(P_\ell) \cup \{j^*\}) = r(\cl[\Gamma_0(P_\ell))] \cup \{j^*\})\ge 3$.

Let us show $r(\Gamma(P_\ell))\le 3$.
Recall that $P_\ell \cup P_4  \in \cH \setminus \{X^*\}$, i.e., $P_{\ell}\cup P_4\not \in \B_1=\B_0 \cup \{X^*\}$. Moreover, $P_{\ell}\cup P_4$ is a circuit in matroid $\B_1$. 
Consequently, 
$r(\Gamma(P_\ell))\le r(\Gamma(P_{\ell}\cup P_4))= 3$ by Theorem~\ref{thm:RadosTheorem}.
\end{proof}

\begin{proposition} \label{prop:h-not-R-induced}
$h$ is not an R-induced valuated matroid.
\end{proposition}
\begin{proof}
Let $X, Y \in \cH\setminus\{\Xs\}$ be two  sets whose intersection is $P_\ell$. If $\ell=1$, we can select $X=P_1\cup P_4 = \{1,2,7,8\}$ and $Y=P_1\cup P_6 = \{1,2,11,12\}$, and if $\ell=2$, we can select $X,Y\in \cH\setminus\{\Xs\}$ intersecting in $P_2$, such as $X=P_2\cup P_3 = \{3,4,5,6\}$ and $Y= P_2\cup P_4 = \{3,4,7,8\}$. 

 Since $h(X),h(Y)=-\infty$ by $\B_1=\B_0\cup\{\Xs\}$, there is no independent matching in $E$ covering $X$ or $Y$, and $X$ and $Y$ are circuits in $\B_1$. By Theorem~\ref{thm:RadosTheorem}, we have
 $r(\Gamma(X))=r(\Gamma(Y))=3$. By Claim~\ref{cl:P-1-rk}, it follows that $\Gamma(X),\Gamma(Y)\subseteq \cl[\Gamma(P_\ell)]$. This further implies that $r(\Gamma(X\cup Y))\le r(\Gamma(P_\ell))=3$, a contradiction since $X\cup Y$ contains a set in $\B$. (For $\ell=1$, one such set is $\{1,2,7,11\}$, and for $\ell=2$, we can select $\{3,4,5,7\}$.)
\end{proof}

\subsection{Robust matroids and their Rado-minor representations}
\label{section:robustMatroids}
In this section we study some additional properties of Rado-minor representations of the matroid $\B_0$. We formulate the properties more generally, so that we can also use them whenever $\B_1=\B_0\cup\{\Xs\}$. 
This always holds in case~\ref{case:samemat}, and we will later show that it must also be true in case~\ref{case:alltight}.

\begin{definition}[Robust matroid]
\label{def:robust}
Let $V = [2n]$, and let $P_i = \{2i - 1, 2i\}$ for $i\in [n]$; these are called \emph{pairs}. 
We define a matroid by its set of bases $\mathcal{B} \subseteq \binom{V}{4}$ and let $\mathcal{H} \coloneqq  \binom{V}{4} \setminus \mathcal{B}$.
We say that $\B$ forms the bases of a \emph{robust} matroid if
\begin{enumerate}[label = (D\arabic*)]
\item\label{i:H-B} Every circuit in $\cH$ is the union of two pairs $P_i\cup P_j$,
\item\label{i:H-graph2} Consider a graph $([n], H)$ where $\{i,j\} \in H$ if and only if $P_i\cup P_j \in \cH$. 
  Then, we can partition $[n]$ into two sets $S$ and $K$ such that $|S|\ge 3$, $K$ is a clique in $H$
  with $|K|\ge 3$, and every node in $S$ is adjacent to every node in $K$.
  (A schematic view of $H$ is given in Figure~\ref{fig:robust}.)
  Moreover, for each $i\in S$
  \begin{enumerate}[label = (D2\alph*)]
    \item\label{i:H-graph2a} there is $j\in S$ such that $i$ is non-adjacent to $j$ in $H$, and
    \item\label{i:H-graph2b} $S\setminus \{i\}$ is not a clique.
  \end{enumerate} 
\end{enumerate}
\end{definition}
Note that this defines a sparse paving matroid of rank 4. 
\begin{figure}
\centering
\includegraphics[width=0.4\textwidth]{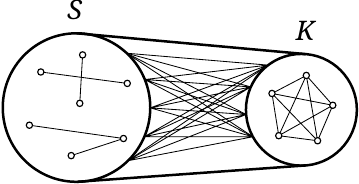}
\caption{The graph $H$ of a robust matroid defined in~\ref{i:H-graph2}.}
\label{fig:robust}
\end{figure}

\begin{lemma}\label{lem:B-0-robust}
Both $\B_0$ and $\B_0\cup\{\Xs\}$ are robust matroids for $n\ge 8$.
\end{lemma}
\begin{proof}
The first property is immediate. 
For \ref{i:H-graph2}, in $\B_0$ (respectively $\B_0\cup\{\Xs\}$), 
it suffices to choose $K$ as the set of even indices (respectively the set of even indices different from $2$). 
In both cases, $S = [n]\setminus K$.
Note that $S$ is a stable set for $\B_0$. 
For $\B_0\cup\{\Xs\}$,  $H[S]$ is disjoint union of a star with center in node $2$ and isolated node $1$.
\end{proof}

Let $\B$ be a robust matroid on $V$. 
Consider a Rado-minor representation $(G,\M)$ with bipartite graph $G=(V\cup W,U;E)$ and $\mathcal{M}=(U,r)$. 
Recall that for $Z\subseteq V\cup W$, we define
$\rho(Z) \coloneqq  r(\Gamma(Z)) - |Z|$, and that there is a unique maximal subset $Q$ of $W$ such that $\rho(Q) = 0$ by Corollary~\ref{i:W-0}. 
We now derive strong structural properties for such a representation of the matroid $\B$ by making heavy use of the results in Section~\ref{sec:uncrossing}. 
In particular, we note that if $\{i,j\}, \{i,k\} \in H$ then the circuits $P_i \cup P_j, P_i \cup P_k \in \cH$ satisfy the conditions of Lemma~\ref{lem:uncrossing2}: their union has cardinality $6$, i.e., the rank of the matroid plus 2, and contains a basis, e.g., $P_i \cup \{2j, 2k\}$.

\begin{lemma}\label{lem:W-Z}\label{i:Z-i} 
For each pair $P_k$, there exists a unique largest $P_k$-set $Z_k$ with $\rho(Z_k)=0$; and $Q
\subseteq Z_k$.
\end{lemma}
\begin{proof}
By \ref{i:H-graph2},
there exist different indices $i, j \in [n]\setminus \{k\}$ 
such that $P_i\cup P_j$ and $P_i\cup P_k$ are circuits in $\cH$. 
By \ref{i:rho-four}, there exists a $(P_i\cup P_k)$-set $I$ and a $(P_j\cup P_k)$-set $J$ with $\rho(I)=\rho(J)=-1$. 
By the second uncrossing Lemma~\ref{lem:uncrossing2}, the intersection $I\cap J$ is a $P_k$-set with $\rho(I\cap J)=0$. 
This shows the existence of a $P_k$-set with $\rho$-value $0$.
The existence of a unique largest such set
follows by the first uncrossing Lemma~\ref{lem:uncrossing} by choosing $X = Y = P_k$ there. 

To see that $Q \subseteq Z_k$, we apply the first uncrossing Lemma~\ref{lem:uncrossing} for $X = \emptyset, I= Q$ and $Y = P_k,  J = Z_k$. 
Namely, $Q \cap Z_k$ is an $\emptyset$-set and $Q\cup Z_k$ is a $P_k$-set. 
Thus, $\rho(Q \cup Z_k) = 0$ and $Q \subseteq Z_k$.
\end{proof}

Let us interpret the above lemma. 
It states that for any pair $P_k$ there exists a unique largest set~$Z_k$ containing exactly $P_k$ in $V$ with $\rho(Z_k) = 0$.
Having $\rho(Z_k) = 0$ means that any independent matching $\mu$ in the Rado-minor representation with 
$\partial_{V\cup W}(\mu) = P_k \cup W$, 
must match the nodes in $Z_k$ to $\cl[\Gamma(Z_k)]$ and no other node is matched to a node in $\cl[\Gamma(Z_i)]$ (Lemma~\ref{lem:rho-0-fully-matched}).

Next we describe how the sets $Z_k$, given by Lemma~\ref{lem:W-Z}, interact with each other.
 
\begin{lemma}\label{lemma:ZiZj}
For any $i,j\in [n]$, $i\neq j$, we have 
\begin{itemize}
  \item If $P_i\cup P_j\in \mathcal{B}$ then  $\rho(Z_i \cup Z_j) = 0$;
  \item if $P_i\cup P_j\in\cH$ then $\rho(Z_i\cup Z_j)=-1$.
  \item For all $i, j\in [n]$, $i\neq j$ we have $Z_i\cap Z_j=Q$ and $\cl[\Gamma(Z_i)] \cap \cl[\Gamma(Z_j)] = \cl[\Gamma(Q)]$.
\end{itemize}
\end{lemma}
\begin{proof}
First, we show the lemma for pairs $P_i$ and $P_j$ such that $P_i\cup P_j$ is a basis in $\B$. 
We have that $Z_i\cap Z_j$ is an $\emptyset$-set and $Z_i\cup Z_j$ is a $(P_i\cup P_j)$-set. 
By the first uncrossing Lemma~\ref{lem:uncrossing}, as $P_i\cup P_j$ is an independent set, we have $\rho(Z_i\cap Z_j) = \rho(Z_i\cup Z_j)=0$. 
By the maximality of $Q$ and since $Q \subseteq Z_i, Z_j$, we have $Z_i\cap Z_j=Q$. 
Finally, Lemma~\ref{lem:rho-prop} implies  $\cl[\Gamma(Z_i)] \cap \cl[\Gamma(Z_j)] = \cl[\Gamma(Q)]$.
This proves the lemma for $i, j\in [n]$ with $P_i\cup P_j \in \B$. 

\smallskip
For the rest of the proof, consider pairs $P_i$ and $P_j$ such that $P_i\cup P_j$ is a circuit in $\cH$. 
We show that $\rho(Z_i \cup Z_j) = -1$.
By~\ref{i:rho-four}, there is a $(P_i\cup P_j)$-set $A$ with $\rho(A)=-1$. 
Let $k\in K\setminus \{i, j\}$ be such that $P_i\cup P_k$ and $P_j \cup P_k$ are circuits $\cH$;
such $k$ is guaranteed by~\ref{i:H-graph2}.
Again by~\ref{i:rho-four}, there exists a $(P_i\cup P_k)$-set $I$ and a $(P_j\cup P_k)$-set $J$ such that $\rho(I)=\rho(J)=-1$. Further, $|P_i\cup P_j\cup P_k|=6$.
By the second uncrossing Lemma~\ref{lem:uncrossing2}, we have $\rho(I\cup J) = -2$.

{Let us show that $A\cap I\subseteq Z_i$ and $A\cap J\subseteq Z_j$. By symmetry, it suffices to show the first claim. Uncrossing gives
\[
-2=\rho(A)+\rho(I)\ge \rho(A\cap I)+\rho(A\cup I)\ge 0-2\, ,
\]
where the second inequality follows by \ref{i:rho-three} and \ref{i:rho-more}, since $A\cap I$ is a $P_i$-set and $A\cup I$ is a $P_i\cup P_j\cup P_k$-set. Consequently, $\rho(A\cap I)=0$, and therefore $A\cap I\subseteq Z_i$ by maximality of $Z_i$.

Next, we uncross $A$ and $I\cup J$, using
$\rho(A) = -1$ and $\rho(I\cup J) = -2$:
\[
-3=\rho(A) + \rho(I\cup J)\ge \rho(A \cap (I\cup J))+\rho(A\cup I\cup J)\ge -1-2\, ,
\]
 by~\ref{i:rho-four} and~\ref{i:rho-more}, since $C\coloneqq A\cap (I\cup J)$ is a $(P_i\cup P_j)$-set and
$A\cup I\cup J$ is a $(P_i\cup P_j\cup P_k)$-set. 
Thus, $\rho(C)=-1$. 
We can rewrite $C=(A\cap I)\cup (A\cap J)$. 
As noted above, $A\cap I\subseteq Z_i$, $A\cap J\subseteq Z_j$, and consequently, $C\subseteq Z_i\cup Z_j$.

Uncrossing $C$ with $Z_i$ gives
\[
-1=\rho(C)+\rho(Z_i)\ge \rho(C\cap Z_i)+\rho(C\cup Z_i)\ge -1\, ,
\]
that is, $\rho(C\cup Z_i)=-1$. Similarly, $\rho(C\cup Z_j)=-1$. Finally, we uncross $C\cup Z_i$ and $C\cup Z_j$:
\[
-2=\rho(C\cup Z_i)+\rho(C\cup Z_j)\ge \rho(C\cup Z_i\cup Z_j)+\rho((C\cup Z_i)\cap (C\cup Z_j))\ge -2\, .
\]
Recalling that $C\subseteq Z_i\cup Z_j$, it follows that  $\rho(Z_i\cup Z_j)=-1$.}

\smallskip
Next, we show that $Z_i\cap Z_j=Q$.
For a contradiction, assume there exists $w\in (Z_i\cap Z_j)\setminus Q \subseteq W$.
Consider $k\in K\setminus \{i, j\}$ as before, i.e., $k\in K\setminus \{i, j\}$ such that $\{i, j, k\}$ is a triangle in the graph $H$.
By the second uncrossing Lemma~\ref{lem:uncrossing2} for $I=Z_i\cup Z_k$ and $J=Z_j\cup Z_k$, we see that $\rho(I\cap J)=0$.
Since $Z_k\subseteq I\cap J$ and $Z_k$ is the largest $P_k$-set with $\rho(Z_k)=0$, it follows that $I\cap J=Z_k$. 
Consequently, $Z_i\cap Z_j\subseteq Z_k$ and $w\in Z_k$ for all $k\in K$.

Let $k, k' \in K$, and consider any $\ell \in S$. 
These three indices again form a triangle in the graph $([n], H)$.
By the same argument as in the previous paragraph, we conclude $w\in Z_\ell$ for all $\ell \in S$.
Hence, $w\in Z_\ell$ for all $\ell\in [n]$.
This is a contradiction as we have already shown that $Z_a \cap Z_b = Q$ whenever $P_a\cup P_b$ is a basis in $\B$.

Finally, we show that $\cl[\Gamma(Z_i)] \cap \cl[\Gamma(Z_j)] = \cl[\Gamma(Q)]$. Similarly to the previous argument,
we assume for the contradiction that there exists $u\in (\cl[\Gamma(Z_i)] \cap \cl[\Gamma(Z_j)])\setminus \cl[\Gamma(Q)]$.
Again, by the second uncrossing Lemma~\ref{lem:uncrossing2} for $I=Z_i\cup Z_k$ and $J=Z_j\cup Z_k$, we have $\rho(I\cap J)=0$ and $I\cap J=Z_k$. 
Moreover, it holds $\rho(I)+\rho(J) = \rho(I\cap J) + \rho(I\cup J)$. 
Lemma~\ref{lem:rho-prop} implies that $\cl[\Gamma(I)] \cap \cl[\Gamma(J)]=\cl[\Gamma(Z_k)]$; 
consequently, $\cl[\Gamma(Z_i)]\cap \cl[\Gamma(Z_j)]\subseteq \cl[\Gamma(Z_k)]$.
As before, this implies that $u\in \cl[\Gamma(Z_{\ell})]$ for all $\ell \in [n]$.
This is a contradiction as we have already shown that $\cl[\Gamma(Z_a)] \cap \cl[\Gamma(Z_b)] = \cl[\Gamma(Q)]$ 
whenever $P_a\cup P_b$ is a basis in $\B$.
\end{proof}

\begin{lemma}\label{lem:all-Z}
We have $\rho(\cup_{i=1}^n Z_i)=4-2n$ and $\rho(\cup_{i\in[n]\setminus\{j\}} Z_i)=2-2n$ for every $j\in [n]$.
\end{lemma}
\begin{proof}
We rely on the following two claims.
\begin{techclaim}\label{claim:threeIndices}
Consider three different indices $i, j, k \in [n]$ such that at least two out of $\{i,j\}, \{i,k\},$ and  $\{j, k\}$ are edges in $H$.
Then, $\rho(Z_i\cup Z_j\cup Z_{k})=-2$. 
\end{techclaim}
\begin{subproof}
Consider the pairs $P_i$, $P_j$, and $P_{k}$ with indices as in the claim.
Without loss of generality assume that $\{i,k\}, \{j,k\}$ are edges in $H$.
Thus, $P_i \cup P_{k}$ and $P_j \cup P_{k}$ are circuits in $\cH$.
Then, we have $\rho(Z_i\cup Z_k)=\rho(Z_j\cup Z_k)=-1$ 
by the second part of Lemma~\ref{lemma:ZiZj}. Let us uncross these two sets.
By submodularity and Lemma~\ref{lem:rho-values}, we have 
\begin{equation*}
  -2 = \rho(Z_i\cup Z_{k}) + \rho(Z_j\cup Z_{k}) \ge 
  \rho(Z_{k}) + \rho(Z_i\cup Z_j\cup Z_{k})\ge 0 - 2\,.
\end{equation*}
Hence, $\rho(Z_i\cup Z_j\cup Z_{k})=-2$.
\end{subproof}

\begin{techclaim}\label{claim:SK3}
Let $L\subseteq [n]$ such that $|L\cap K| \ge 3$ and $L\cap S$ contains two non-adjacent indices $i$ and $j$. (Recall that $K$ and $S$ are the sets given by~\ref{i:H-graph2}.)
Then, $\rho(\cup_{i\in L} Z_i) = 4-2|L|$. 
\end{techclaim}
\begin{subproof}
As $\{i,j\}\not \in H$ then $P_i\cup P_j\in \B$ and thus $\rho(Z_i\cup Z_j)=0$ by the first part of Lemma~\ref{lemma:ZiZj}.
Consider any index $k \in K$.
By Claim~\ref{claim:threeIndices}, $\rho(Z_{k}\cup Z_i\cup Z_j) = -2$. 
Therefore, adding $Z_{k}$ to $Z_i\cup Z_j$ decreases the $\rho$ value by 2.
In other words, for any $k \in K$ we have 
\begin{equation}\label{equation:submodularityThreeZ}
  \Delta_{\rho} (Z_{k} | Z_i \cup Z_j) \coloneqq 
  \rho(Z_{k}\cup Z_i\cup Z_{j}) 
  - \rho(Z_i \cup Z_j) =  -2 - 0 = -2 \,.
\end{equation}

As the intersection of any pair $Z_{a} \cap Z_b = Q$ by Lemma~\ref{lemma:ZiZj}, submodularity implies that adding $\ell$ different sets $Z_{k}$ with $k \in K$ to $Z_i\cup Z_j$  
decreases $\rho$ by at least $2\ell$. We proceed to prove a similar statement for sets $Z_k$ with $k\in S$.

Consider three different indices $a, b, c \in K\cap L$.
Let $Y=Z_a\cup Z_b\cup Z_c$. 
We then have, $\rho(Y \cup Z_i \cup Z_j)\le 0 - 2\cdot3= 4-2\cdot 5$. 
Since $a,b,c\in K$, by Claim~\ref{claim:threeIndices}, we have $\rho(Y)=-2$. 
By Lemma~\ref{lem:rho-values}~\ref{i:rho-more}, we also have $\rho(Y \cup Z_i \cup Z_j) \ge 4 - 2\cdot 5$ 
and consequently  $\rho(Y \cup Z_i \cup Z_j) = 4 - 2\cdot 5$. 
(Which proves the claim if $L = \{a,b,c,i,j\}$.) 

\medskip
Rearranging the above we conclude that whenever $\{i,j\} \not \in H$, we have 
\begin{equation}\label{equation:submodularityFiveZ}
  \Delta_{\rho} (Z_i \cup Z_j | Y ) = -4 \,.
\end{equation}
In other words, adding $Z_i \cup Z_j$ to $Y$ leads to a decrease of $4$ in the $\rho$ value. 
By Lemma~\ref{lem:rho-values}~\ref{i:rho-more} we also have 
$\rho(Y \cup Z_i) \ge 4 - 2\cdot 4 = -4$, 
and $\rho(Y \cup Z_i \cup Z_j) \ge 4 - 2\cdot 5 = -6$.
Combining it with the previous paragraph, 
we have $\Delta_{\rho} (Z_i | Y ) \ge -2$, and $\Delta_{\rho} (Z_j | Y \cup Z_i) \ge -2$.
Using~\eqref{equation:submodularityFiveZ} and submodularity we conclude that the inequalities hold with equality. 
That is, we have
\begin{equation}\label{equation:submodularityOddZ}
  \Delta_{\rho} (Z_i | Y ) = -2 \,
\end{equation}
for every $i$ such that $\{i,j\} \not \in H$ for some $j\in S$, i.e., by~\ref{i:H-graph2a}, for every $i \in S$.
By submodularity, adding $\ell$ different sets $Z_i$ with $i\in S$ to $Y$ decreases the $\rho$ value by at least $2\cdot \ell$.

Thus, for our set $L$, by submodularity and combining~\eqref{equation:submodularityThreeZ} and~\eqref{equation:submodularityOddZ}
we have $\rho(\cup_{i\in L}Z_i) \le 4 - 2\cdot |L|$. The equality holds by Lemma~\ref{lem:rho-values}~\ref{i:rho-more}.
\end{subproof}
The lemma follows by~\ref{i:H-graph2b} and the last claim for $L = [n]$ and $L  = [n]\setminus \{j\}$.
\end{proof}


\subsection{Bounding the support of \texorpdfstring{$h$}{h}}\label{sec:robust-to-rep}
Since $\B_0$ is always a robust matroid, we can use the results in the previous section for $\B=\B_0$.
Let $Z^0_i$ denote the  $P_i$-set $Z_i$ for $\B=\B_0$, and $Q_0$ the unique largest subset of $W$ with $\rho_0(Q_0)=0$.

Our first goal is to show Lemma~\ref{lem:B1-B0} below, namely, that in both case \ref{case:samemat} and \ref{case:alltight}, we have that $\dom(h)=\B_1= \B_0 \cup \{X^*\}$. Thus, we get the smallest possible size according to the main selection criterion \ref{crit:minsup}. This will enable us to also use the robust matroid analysis on $\B=\B_1$.
The proof will rely on the following `reduction' of the matroid $\M$.

\paragraph{Reducing $\M$} 
We replace $\M$ on $U$ by the following matroid $\overline \M$: a set $T\in {U\choose |W|+4}$ is a basis in $\overline \M$ if and only if there is a matching in $E$ between $T$ and a basis in
\[
\overline \B\coloneqq \{X \cup W:\, X\in \B_0\cup\{\Xs\}\}\, .
\] 
These sets $T$ form the bases of a matroid by Rado's theorem. 
Since $h(X)$ is finite for all $X\in\B_0\cup\{\Xs\}$, this will be a `submatroid' (weak map) of $\M$, i.e., all bases of $\overline \M$ are bases in $\M$. 
Let $\bar h(X)$ be the function corresponding to the modified representation $(G, \overline \M, c, W)$. 
Clearly, $\bar h(X)=h(X)$ for every $X\in\B_0\cup\{\Xs\}$ and  $\bar h(X)\le h(X)$ otherwise.
As $\bar h$ has the same or better criteria~\ref{crit:minsup}--\ref{crit:nontight} than $h$,
 for the rest of the proof we shall assume $h = \bar h$ and $\M = \overline \M$.

Using this construction, we first show that $Q_0=\emptyset$ in \ref{case:alltight}. However, $Q_0\neq \emptyset$ may still be possible in case \ref{case:samemat}.
\begin{lemma}\label{lem:no-Q0-CII} 
In case \ref{case:alltight}, i.e., if $E=E^*$, then  $Q_0=\emptyset$ must hold. Thus, $\rho_1(S)\ge\rho_0(S)\ge 1$ for all $S\subseteq W, S\neq \emptyset$ in this case.
\end{lemma}
\begin{proof}
Denote with $T_0 = \Gamma(Q_0)$.
By definition of $\rho_0$, $r_{\tau} (T_0) = |Q_0|$. 
We claim that also $r(T_0) = |Q_0|$. The next claim will be needed for this proof.
\begin{techclaim}\label{claim:noEdgeSomewhere}
    There is no edge $(i,j)\in E$ with $i \in (V\cup W)\setminus Q_0$ and $j \in T_0$.
\end{techclaim}
\begin{subproof}
  Suppose there is such an edge. 
  By definition of $E^*$ ($= E$), 
  there exists an independent matching $\mu$ containing $(i,j)$ with weight $0$ such that the endpoints $\partial_U(\mu)$ are independent in $\M_\tau$.
  Trivially, this matching also covers $Q_0$ as $Q_0 \subseteq W$. This  leads to a contradiction with Lemma~\ref{lem:rho-0-fully-matched}.
\end{subproof}

Suppose that $r(T_0) > |Q_0|$.
Then there is a basis $S$ of $\M$ such that $|S\cap T_0| > |Q_0|$.
As $\overline \M = \M$ there is an independent matching 
that matches $S\cap T_0$ to a subset of size $>|Q_0|$ in $V\cup W$.
This is impossible as the neighbourhood of $T_0$ in $V$ is $Q_0$ by Claim~\ref{claim:noEdgeSomewhere}.
Hence, $r(T_0) = |Q_0|$.
This contradicts Lemma~\ref{lemma:noQ1}.
\end{proof}

\begin{lemma}\label{lem:B1-B0}
 $\B_1= \B_0 \cup \{X^*\}$ must hold.
\end{lemma}
\begin{proof}
There is nothing to prove in \ref{case:samemat}, so let us assume we are in case~\ref{case:alltight}; thus, $E=E^*$. According to the previous lemma, we also have $Q_0=\emptyset$.
Let $Z^*=\cup_{i=1}^n Z^0_i$; in particular $V\subseteq Z^*$. 

\begin{techclaim}\label{cl:W-Zst}
There are no edges between $W\setminus Z^*$ and $\Gamma(Z^*)$. 
\end{techclaim}
\begin{subproof}
Let $F$ denote the edge set in the claim.
Let $T^*\coloneqq \Gamma(Z^*)=\Gamma_0(Z^*)$.
By Lemma~\ref{lem:all-Z} and Lemma~\ref{lem:no-Q0-CII}, we have that $\rho_0(Z^*)=4-2n$.
As $\rho_0 (Z^*) = r_{\tau}(\Gamma_0(Z^*)) - |Z^*| = r_{\tau}(T^*) - |Z^*| $ we have $r_\tau(T^*)=4+|Z^*\cap W|$. 
Consequently, an independent matching $\mu$ of weight $0$ cannot use any of the edges in $F$, 
since $|\partial_{Z^*}(\mu^X)| = 4 + |Z^* \cap W|$ and thus $\partial_{Z^*}(\mu^X)$ must be matched to a maximal independent set in $T^*$. 
Hence, $E^*\cap F = \emptyset$. Then $F=\emptyset$ as $E=E^*$.
\end{subproof}

Consider any $X\in \B_1\setminus(\B_0\cup\{\Xs\})$. 
We have $X=P_i\cup P_j$ for some $i,j\in [n]$, $\{i,j\}\neq\{1,2\}$ by the definition of $\cF_n$. 
Thus, there is a matching in $E$ between $X\cup W$ and a basis $R$ of $\mathcal{M}$. Since  $\M = \overline \M$, there is in turn another matching $\mu$ between $R$ and a set $Y\in \B_0\cup\{\Xs\}$.

Let $S\coloneqq Z^0_i\cup Z^0_j$ and $T\coloneqq \Gamma(S)$. Thus, $|R\cap T|=|S|$. Therefore, the matching $\mu$ must match $R\cap T$ to an independent subset $Y'\subseteq Y$ with $|Y'|=|S|$. Note that $Y'\neq S$, since $S$ is not a subset of any basis in $\B_0\cup\{\Xs\}$.  Therefore, $\mu$ must contain an edge between $T$
and $(V\cup W)\setminus S$. The following claim shows that this is not possible, giving rise to a contradiction.

\begin{techclaim}
$\Gamma(T)=S$.
\end{techclaim}
\begin{subproof}
In case~\ref{case:alltight},
Lemma~\ref{lem:no-Q0-CII} shows $Q_0 = \emptyset$. Thus, Lemma~\ref{lemma:ZiZj} implies that $T=\Gamma(Z^0_i)\cup \Gamma(Z^0_j)$ is disjoint from
 $\Gamma(Z^0_k)$ for all $k\notin \{i,j\}$. Thus, $\Gamma(T)\cap Z^*=S$. Further, $\Gamma(T)\cap ((V\cup W)\setminus Z^*)=\emptyset$ according to Claim~\ref{cl:W-Zst}. The claim follows.
\end{subproof}
\end{proof}

In light of the above Lemma, we can apply the techniques in Section~\ref{section:robustMatroids} 
to the robust matroid $\B_1=\B_0\cup \{X^*\}$. 
Let $Z^1_i$ denote the corresponding sets in Lemma~\ref{lemma:ZiZj}, and recall that $\rho_1(Z)=r(\Gamma(Z))-|Z|$. By Lemma~\ref{lemma:noQ1}, the largest subset $Q_1$ of $W$ with $\rho_1(Q_1)=0$ is $Q_1=\emptyset$.
Applying Lemma~\ref{lemma:ZiZj} then immediately gives:
\begin{lemma}\label{lem:Z1-disjoint}
For any $i,j \in [n], i\neq j$, we have $Z^1_i \cap Z^1_j = \emptyset$ and $\cl_{\M}[\Gamma(Z^1_i)] \cap \cl_{\M}[\Gamma(Z^1_j)] = \emptyset$.
\end{lemma}
\begin{lemma}\label{lem:unionZ} 
$\cup_{i=1}^n Z^1_i=V\cup W$.
\end{lemma}
\begin{proof}
We use the following claim stating that, in the minimal counterexample, for any $V$-set $Z$ with sufficiently large $\rho_1$-value it holds $V\cup W = Z$.
\begin{techclaim}\label{claim:W'=W}
Let $Z = V\cup W'$ for $W' \subseteq W$ such that $\rho_1(Z) = 4 - |V|$.
Then,  we must have $W' = W$ or equivalently, $Z = V\cup W$.
\end{techclaim}
\begin{subproof}
For a contradiction assume that $W' \neq W$.
Let $T=\cl[\Gamma(V\cup W')]$. 
By definition of $\rho_1$, having $\rho_1(V\cup W')= 4 - |V|$ means  $r(T) = |V\cup W'| + 4  - |V| = |W'| + 4$.
Thus, for any $X\in\B_1$ ($|X| = 4$) the corresponding matching $\mu^X\in {\mathcal{L}}$ matches exactly $r(T)$ nodes in $T$ to the nodes in 
$X\cup W'$. 
In other words, any matching $\mu^X \in {\mathcal{L}}$ matches nodes $W\setminus W'$ to $|W\setminus W'|$ nodes in $U\setminus T$.

Similarly to the proof of Lemma~\ref{lemma:noQ1}, it follows that in any $\mu^X\in {\mathcal{L}}$, the cost of the edges covering $W\setminus W'$ is the same.
Hence, we can get a smaller representation by restricting $W$ to $W'$ and $U$ to $U'$. This contradicts the selection criterion \ref{crit:minW}.
\end{subproof}
Lemma~\ref{lem:all-Z} for $\B_1$ gives $\rho_1(\cup_{i=1}^n Z^1_i)=4-2n$. Also noting that $V\subseteq \cup_{i=1}^n Z^1_i$,
the statement follows by Claim~\ref{claim:W'=W}. 
\end{proof}

\begin{lemma}\label{lem:Q0inZk}
We have $Z^0_i = Z^1_i \cup Q_0$ for $i\in [n]$.
In particular, in case \ref{case:alltight} we have $Z^0_i=Z^1_i$. 
\end{lemma}
\begin{proof}
Let us first show $Z^1_i\cup Q_0\subseteq Z^0_i$. By Lemma~\ref{i:Z-i}, $Q_0\subseteq Z^0_i$. Let us show $Z^1_i\subseteq Z^0_i$.
We have $\rho_0(Z^1_i)\ge 0$ by~\ref{i:rho-three} since $Z^1_i$ is a $P_i$-set, and also
$\rho_0(Z^1_i)\le \rho_1(Z^1_i)=0$. Thus, $\rho_0(Z^1_i)=0$. By the maximality of $Z^0_i$ (Lemma~\ref{i:Z-i}), 
it follows that $Z^1_i\subseteq Z^0_i$. 

We next show that equality holds. 
For the sake of contradiction, assume that we have $w\in Z^0_i\setminus (Z^1_i \cup Q_0)$ for some $i\in [n]$.
Lemma~\ref{lem:unionZ} shows that $\cup_{i=1}^n Z^1_i=V\cup W$, 
and hence we must have $w\in (Z^0_i\cap Z^1_j)\setminus Q_0$ for some $j\neq i$. Thus, we get a contradiction by
\[
w\in Z^0_i \cap Z^1_j\subseteq Z^0_i\cap Z^0_j =Q_0\, ,\]
where the last inequality follows by the third part of Lemma~\ref{lemma:ZiZj}.
\end{proof}

\subsection{The case \ref{case:samemat}}\label{sec:samemat}
We are ready to show that case~\ref{case:samemat} cannot occur. 
In this case, we have $\M_\tau=\M$, $E=E_0\cup\{(i^*,j^*)\}$, and  $\B_1=\B_0\cup\{\Xs\}$.

\begin{lemma}\label{lem:Z-q}
  Either $Q_0 = \emptyset$ or there exists a unique $q\in [n]$ such that $Q_0 \subseteq Z^1_q$.
\end{lemma}
\begin{proof}
Suppose $Q_0\cap Z^1_q \neq \emptyset$ for some $q\in [n]$.
By definition, $\rho_1(Z^1_q) = 0$. By Lemma~\ref{lemma:noQ1}, we have $\rho_1(Q_0) \ge 1$ and $\rho_1(Z^1_q \cap Q_0)\ge 1$. By~\ref{i:rho-three}, $\rho_1(Z^1_q \cup Q_0) \ge 0$ holds since $Z^1_q \cup Q_0$ is a $P_q$-set.
By submodularity,
\[
0+1 = \rho_1(Z^1_q) + \rho_1(Q_0) \ge \rho_1(Z^1_q \cap Q_0) + \rho_1(Z^1_q \cup Q_0) \ge 1  + 0\, ,
\]
implying $\rho_1(Z^1_q \cup Q_0) = 0$.
By maximality of $Z^1_q$, we have $Q_0 \subseteq Z^1_q$. The uniqueness of such a $q$ follows since the sets $Z^1_i$ are pairwise disjoint (Lemma~\ref{lem:Z1-disjoint}).
\end{proof}

\begin{lemma}\label{lemma:-10}
  We have $\rho_0(Z^0_1 \cup Z^0_2) = -1$ and  $\rho_0(Z^1_1 \cup Z^1_2) = 0$. Consequently, $Q_0\neq \emptyset$  and $q\notin \{1,2\}$ for $q$ as in Lemma~\ref{lem:Z-q}.
\end{lemma}
\begin{proof}
  Recall that $\rho_0(Z^0_1 \cup Z^0_2) = -1$ by Lemma~\ref{lemma:ZiZj} as $X^*\in\cH$.
  We claim that $\rho_0(Z^1_1 \cup Z^1_2) = 0$. 
  First, note  that $\rho_0(Z^1_1 \cup Z^1_2)\le 0$ since $\rho_0(Z^1_1 \cup Z^1_2)\leq\rho_1(Z^1_1 \cup Z^1_2)=0$.
  For a contradiction, assume that $\rho_0(Z^1_1 \cup Z^1_2) < 0$.  

Thus, $\rho_0(Z^1_1 \cup Z^1_2)<\rho_1(Z^1_1 \cup Z^1_2)=0$, implying 
$r(\Gamma(Z^1_1 \cup Z^1_2))>r_\tau(\Gamma_0(Z^1_1 \cup Z^1_2))$. Noting that  $\M = \M_{\tau}$ and therefore $r=r_{\tau}$, this means that $\Gamma(Z^1_1 \cup Z^1_2)\supsetneq\cl[\Gamma_0(Z^1_1 \cup Z^1_2)]$. Thus, the single edge 
$(i^*,j^*)\in E\setminus E_0$ is incident to $Z^1_1 \cup Z^1_2$.
  Let $\ell\in \{1,2\}$ such that $i^*\in Z^1_\ell$. 
  Now, we must have  $0\le \rho_0(Z^1_\ell)<\rho_1(Z^1_\ell)=0$, a contradiction.

  The last statements follow since if $Q_0=\emptyset$ or $q\in\{1,2\}$, then $Z^0_1\cup Z^0_2=Z^1_1\cup Z^1_2$ by Lemma~\ref{lem:Q0inZk}.
\end{proof}
For the rest of the proof, let us fix $q\in[n]$ such that $Q_0 \subseteq Z^1_q$, and
  let $Y \coloneqq  \cup_{i\in [n]\setminus \{q\}} Z^1_i$.
\begin{lemma}\label{lem:rho-0-Y}$\rho_0(Y) = 2-2n$. 
\end{lemma}
\begin{proof}
  By the second part of Lemma~\ref{lem:all-Z} for $\rho_1$, we have $\rho_1(Y) = 2-2n$.
  We show that the same holds for $\rho_0$.

  As all $Z^1_i$ are disjoint (Lemma~\ref{lem:Z1-disjoint}), we have  $Q_0 \cap Z^1_i =\emptyset$ for all $i \in[n]\setminus \{q\}$. 
  Then by Lemma~\ref{lem:Q0inZk} we have that $Z^1_i = Z^0_i \setminus Q_0$. Thus,
\begin{align}
  \rho_0(Y) &= \rho_0(\cup_{i\in [n]\setminus \{q\}} Z^1_i) \notag \\
      &= \rho_0(\cup_{i\in [n]\setminus \{q\}} (Z^0_i \setminus Q_0)) \notag\\
      &= \rho_0\left((\cup_{i\in [n]\setminus \{q\}} Z^0_i) \setminus Q_0\right) + \rho_0(Q_0) \tag{$\rho_0(Q_0) = 0$}\\
      &\ge \rho_0(\cup_{i\in [n]\setminus \{q\}} Z^0_i) + \rho_0(\emptyset) \tag{submodularity}\\
      & = 2- 2n \tag{Lemma~\ref{lem:all-Z} for $\rho_0$}\,.
\end{align}
Since $\rho_0(Y)\le \rho_1(Y)$, we conclude $\rho_0(Y) = 2-2n$.
\end{proof}

Let us now derive the final contradiction for~\ref{case:samemat}. 
As $Q_0\subseteq Z^1_q\subseteq (V\cup W)\setminus Y$, by submodularity we get
\[
\rho_0(Z^1_1 \cup Z^1_2 \cup Q_0) + \rho_0(Y)\ge 
\rho_0(Y \cup Q_0) + \rho_0(Z^1_1 \cup Z^1_2)  \,.
\]
Then, by Lemma~\ref{lemma:-10}, and using $Z^1_1\cup Z^1_2\cup Q_0=Z^0_1\cup Z^0_2$, we further have 
\begin{align*}
\rho_0(Y \cup Q_0) - \rho_0(Y) &\le \rho_0(Z^1_1 \cup Z^1_2 \cup Q_0) - \rho_0(Z^1_1 \cup Z^1_2) \\
&= \rho_0(Z^0_1 \cup Z^0_2) - \rho_0(Z^1_1 \cup Z^1_2)= -1\,.
\end{align*}
From Lemma~\ref{lem:rho-0-Y}, $\rho_0(Y\cup Q_0) \le 1-2n$.
On the other hand, $\displaystyle \rho_0(Y\cup Q_0) = \rho_0(\cup_{i\in [n]\setminus \{q\}} Z^1_i \cup Q_0) = 
\rho_0(\cup_{i\in [n]\setminus \{q\}} Z^0_i) = 2-2n$. A contradiction.


\subsection{The case~\ref{case:alltight}}
\label{sec:alltight}
In the remaining case~\ref{case:alltight}, we have $E=E_0=E^*$ but $\M_\tau\neq \M$. In Section~\ref{sec:robust-to-rep}, we have already showed some strong properties for this case: $Q_0=\emptyset$ (Lemma~\ref{lem:no-Q0-CII}),  $\B_1=\B_0\cup\{\Xs\}$ (Lemma~\ref{lem:B1-B0}), and $Z^0_i=Z^1_i$ for all $i\in [n]$ (Lemma~\ref{lem:Q0inZk}). In light of this, we can simplify the notation to $Z_i=Z_i^0=Z_i^1$. Recall also that the sets $Z_i$ are disjoint by Lemma~\ref{lem:Z1-disjoint}.

Let $D_i\coloneqq \cl[\Gamma_{E}(Z_i)]$; see Figure~\ref{fig:figureZi}.
By  Lemma~\ref{lemma:ZiZj}, there are no edges with one end point in $Z_i$ and the other in $D_j$ whenever $i\neq j$.

Let us additionally modify the bipartite graph in the representations:
we may assume that $E=E_0=E^*$ is a complete bipartite graph between $Z_k$ and $D_k$ for any $k\in [n]$.
For this, we need to check that after adding such new edges between $Z_k$ and $D_k$, $(G,\M,c,W)$ still represents a function in $\cF_n$ and is equally good on the selection criteria \ref{crit:minsup}, \ref{crit:minW}, and \ref{crit:nontight}. Since we add these new edges to be tight, the criteria \ref{crit:minW} and \ref{crit:nontight} remain the same. To show that we still represent a function in $\cF_n$ and that $\B_1$ does not increase, we use that $\rho_1(Z_i\cup Z_j)=-1$
by Lemma~\ref{lemma:ZiZj} for any $X=P_i\cup P_j\in\cH$, and this will hold after adding new edges between the $Z_k$ and $D_k$ sets. Similarly, $\rho_0(Z_1\cup Z_2)=-1$ will be maintained, and thus we still have $h(X^*)<0$, even though the value of $h(X^*)$ might increase.

\begin{figure}
\centering
\includegraphics[width=0.8\textwidth]{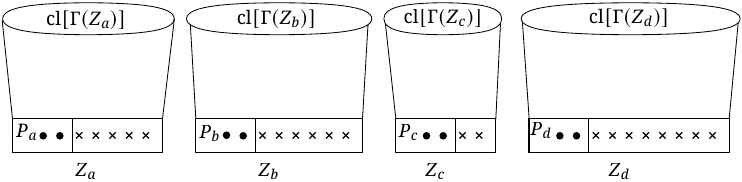} 
\caption{Case~\ref{case:alltight}: Schematic example of matroid $\B_1$ with its Rado-minor representation $(G, \M, W)$. Here, the neighbourhood is taken in the edge set $E$, and the closure in the matroid $\M$. 
The black dots represent set $V$ and the white dots represent $W$.
Similarly, a Rado-minor representation holds for $\B_0$ once we replace $\M$ (and closure) by $\M_{\tau}$.}
\label{fig:figureZi}
\end{figure}

Introducing these new edges allows us to describe the representations of $\B_0$ and $\B_1$ in purely set-theoretic and matroidal terms.
The following definition captures this representation.
\begin{definition} For a set $X\in \binom{V}{4}$, we say that a set $S\subseteq U$, $|S|=|W|+4$ \emph{conforms} $X$ if 
$|S\cap D_i|=|X\cap P_i|+|Z_i|-2$ for all $i\in [n]$. 
\end{definition}
 The requirements on our matroids $\M$ and $\M_\tau$ can be stated as follows:
\begin{itemize}
  \item For any $X\in \binom{V}{4}$, there exists a basis $S$ in $\M$ conforming $X$ if and only if $X\in \B_1$. 
  \item For any $X\in \binom{V}{4}$, there exists a basis $S$ in $\M_\tau$ conforming $X$  if and only if $X\in \B_0$.
\end{itemize}

The next lemma concludes the proof of Theorem~\ref{thm:non-r-minor-main}, by showing that $W=\emptyset$ is a minimal representation. Thus, the existence of an R-minor representation would imply the existence of an R-induced representation, which we have already shown cannot exist.

Recall from Lemma~\ref{lem:all-Z} (applied to both $\rho_0$ and $\rho_1$) that 
$\rho_0(\cup_{i=1}^n Z_i)=\rho_1(\cup_{i=1}^n Z_i)=4-2n$ and
and $\rho_0(\cup_{i\in[n]\setminus\{j\}} Z_i)=\rho_1(\cup_{i\in[n]\setminus\{j\}} Z_i)=2-2n$ for every $j\in [n]$.

\begin{lemma}
In a minimal representation we must have $W=\emptyset$.
\end{lemma}
\begin{proof}
For a contradiction, assume $W\neq\emptyset$; pick $i\in [n]$ such that $|Z_i|>2$. 
Now, every basis in $\M$ (and thus in $\M_\tau$) must intersect $D_i$ in at least $|W\cap Z_i|=|Z_i|-2>0$ elements. Here, we use that $\M=\overline \M$ as in Section~\ref{sec:robust-to-rep}: thus, every basis in $\M$ must have a matching to a set $W\cup X$, $X\in\B_0\cup\{X^*\}=\B_1$. 
Since $|Z_i|>2$, every basis in $\M_\tau$ has a non-empty intersection with $D_i$, implying  that $D_i\not\subseteq  \cl_{\M_\tau}[U\setminus D_i]$; let us pick $u\in D_i\setminus \cl_{\M_\tau}[U\setminus D_i]$.
We claim that a smaller representation can be obtained by contracting $u$ in the matroid $\M$ and deleting an arbitrary node from $W\cap Z_i$.

 To see this, it suffices to prove that for every $X\in \B_0$, there exists a basis $S$ in $\M_\tau$ conforming $X$ with $u\in S$, 
 and there exists a basis $S^*$ in $\M$ conforming $X^*$ with $u\in S^*$. 
 Then, the requirements listed above remain true in the smaller instance. 
 Note that we do not require that $S^*$ has the largest possible $\tau$-weight; 
 as long as we can guarantee the existence of a basis in $\M$ but not in $\M_\tau$ that conforms $X^*$, we get a function in $\cF_n$ that is the same on \ref{crit:minsup}, but better on \ref{crit:minW} (with possibly different negative value $h(\Xs)$).

 Consider any $X\in \B_0$ and a basis $S$ in $\M_\tau$ conforming $X$. We are done if $u\in S$. If  $u\notin S$, then 
 let $C \subseteq S\cup \{u\}$ be the fundamental circuit of $u$ with respect to $S$.
 Then, $(C\setminus u)\cap D_i\neq\emptyset$: 
 otherwise, $C\setminus u\subseteq U\setminus D_i$ would yield $u\in \cl_{\M_\tau}[U\setminus D_i]$, a contradiction to the choice of $u$. Hence, we can exchange $u$ with an element of $S\cap D_i$ and thereby obtain another basis $S'$ conforming $X$ with $u\in S'$.

 The same argument applies for the basis $S^*$ in $\M$ conforming $X^*$, noting that $\cl_{\M}[U\setminus D_i]\subseteq \cl_{\M_\tau}[U\setminus D_i]$.
\end{proof}

\section{Valuated generalized matroids}
\label{sec:valuated+generalized+matroids}
In this section, we build on Theorem~\ref{thm:non-r-minor-main} to refute the matroid based valuation conjecture.
To do this, we extend the class of R-minor valuated matroids to $\Rnat$-minor valuated generalized matroids, and show this contains matroid based valuations as a subclass.
Furthermore, we extend our main counterexample to a monotone valuated generalized matroid that is not $\Rnat$-minor and therefore not a matroid based valuation, refuting the MBV conjecture.

Recall from~\eqref{eq:Mnat-concave} and~\eqref{eq:M-concave} the properties of valuated generalized matroids.
In Appendix~\ref{sec:vgm-to-vm}, we demonstrate a construction which allows one to consider valuated generalized matroids as special cases of valuated matroids on a larger ground set. 
On the other hand, we already saw valuated matroids as a special class of valuated generalized matroids.

An important class are the trivially valuated generalized matroids, those taking only values $0$ and $-\infty$. 
This includes the characteristic functions of the family of independent sets of a matroid. Indeed, if $g(\emptyset)>-\infty$ for a valuated generalized matroid, then $\dom(g)$ is the family of independent sets of a matroid~\cite[Corollary~1.4]{MurotaShioura:2018}. 
\smallskip

Given a valuated generalized matroid $g$ on $V$, we define its \emph{$k$-th layer} $\layer{k}{g}$ to be the restriction of $g$ to $\binom{V}{k}$.
It is immediate from the definition that $\layer{k}{g}$ is a valuated matroid.
We defined several constructions for valuated matroids in Section~\ref{sec:preliminaries}.
It turns out that these operations extend essentially layer-wise to valuated generalized matroids. 

\begin{definition}
  Let $f \colon 2^V \to \Trop$ be a valuated generalized matroid and $Y\subset V$ some subset of $V$.
  The operations \emph{deletion (restriction), contraction, dualization, truncation, principal extension} are defined by the respective operations on the layers from Definition~\ref{def:M+operations}.
\end{definition}

Note that direct sum and valuated matroid union do not extend layerwise to valuated generalized matroids.
Intuitively, this is because the $k$-th layer of the union must take information from multiple layers of the constituent valuated generalized matroids, all $i$-th and $j$-th layers such that $k = i+j$.
The analogue of direct sum and valuated matroid union for valuated generalized matroids is the following operation.

\begin{definition}
  Let $f, g \colon 2^V \to \Trop$. 
  The \emph{merge} of $f$ and $g$ is the function $f*g \colon 2^V \to \Trop$ defined as
  \[
  (f*g)(X) = \max\SetOf{f(Y) + g (X \setminus Y)}{Y \subseteq X}, \quad \forall X\subseteq V .
  \]
\end{definition}

We now extend induction by network to valuated generalized matroids.

\begin{definition} \label{def:induction-by-networks-Mnat}
  Let $N=(T, A)$ be a directed network with a weight function $\co\in \R^A$. Let $V, U\subseteq T$ be two non-empty subsets of nodes of $N$.
  Let $g$ be a valuated generalized matroid on $U$.
  Then the \emph{induction of $g$ by $N$} is the function $\tempinducednew{N}{g}{\co} \colon 2^V \to \Trop$ such that
  \[
  \layer{k}{\tempinducednew{N}{g}{\co}} = \tempinducednew{N}{\layer{k}{g}}{\co} , 
  \]
  where $\tempinducednew{N}{g}{\co}(\emptyset) = g(\emptyset)$.

In the special case that the directed network is bipartite with the edges directed from $V$ to $U$, we can also consider this as an undirected weighted bipartite graph and call the corresponding operation \emph{induction by bipartite graphs}. 
\end{definition}
Note that the formula defining $\tempinducednew{N}{g}{\co}(X)$ remains identical to the one in Definition~\ref{def:induction-by-networks}, namely,
\[
\tempinducednew{N}{g}{\co}(X) = \max\SetOf{\sum_{a \in \mathcal{P}} \co(a) + g(Y)}{\begin{aligned}&\text{node-disjoint paths } \mathcal{P} \text{ in } N \colon\\ &\partial_V({\cP}) = X \wedge \partial_U({\cP}) = Y\end{aligned}} .
\]
Analogous to Theorem~\ref{thm:induction-network} this is just a special case of transformation by networks.

\begin{theorem}[{\cite[Theorem~9.27]{Murotabook}}]\label{thm:induction-network-Mnat}
  Let $N, g, \co$ as in Definition~\ref{def:induction-by-networks-Mnat}. Then if $\tempinducednew{N}{g}{c} \not\equiv -\infty$ the induced function is a valuated generalized matroid. 
\end{theorem}

With this operation, we get an analogue of Theorem~\ref{thm:VM+closed+operations}. 

\begin{theorem}
    The class of valuated generalized matroids is closed under the operations deletion, contraction, dualization, truncation, principal extension, merge. 
\end{theorem}
\begin{proof}
  Deletion, dualization and merge are covered by \cite[Theorem 6.15]{Murotabook}; the latter is integer infimal convolution restricted to the interval $[0,1]$, parts (8) and (5) respectively.
  Lemma~\ref{lem:contraction-dual+deletion} implies layerwise closure under contraction and therefore globally closed contraction.
  Remark~\ref{rem:principle+extension} shows principal extension are special cases of induction by networks, which valuated generalized matroids are closed under via Theorem~\ref{thm:induction-network-Mnat}.
  Finally, Lemma~\ref{lem:truncation-extension+contraction} implies layerwise closure under truncation and therefore globally closed under truncation.
\end{proof}

As with induction of valuated matroids, we shall often be most interested in the induction of trivially valuated generalized matroids.
A trivially valuated generalized matroid $g$ can be identified with its underlying domain $\I$, where $g(I) = 0$ if $I \in \I$ and $-\infty$ otherwise.
As stated previously, if $\emptyset \in \I$ then $\I$ forms the set of independent sets of a matroid; however this does not have to be the case, $\I$ only has to satisfy the independent set exchange axiom (the unvaluated equivalent of~\eqref{eq:Mnat-concave}).
We call such an $\I$ a \emph{generalized matroid}.
As working with $\I$ directly will be convenient in some situations, we extend the notation of Definition~\ref{def:induction-by-networks-Mnat} to define $\tempinducednew{N}{\I}{c} := \tempinducednew{N}{g}{c}$.

The following example demonstrates how weighted rank functions arise by this operation.

\begin{example}\label{ex:mbv}
  Let $\I$ be the independent sets of a matroid $\M$ on ground set $V$.
A \emph{weighted rank function} $r^{w} : 2^V \rightarrow \R_{\geq 0}$ with weight $w \in \R_{\geq 0}^n$ is
\begin{equation*}
r^w(X) = \max\SetOf{\sum_{i\in I} w_i}{I \subseteq X \, , \, I \in \I} \enspace .
\end{equation*}
Note that if $w$ is the vector of all ones, then $r^w$ is precisely the rank function of $\M$.

We now show how the function $r^w$ arises via induction by network from a trivially valuated generalized matroid.
Let $V'$ and $V''$ be copies of $V$ and let $\overline{\I}$ be the independent sets of the matroid $\overline{\M} = \M \oplus \free{V''}$ on $V' \cup V''$.
Furthermore, we define the bipartite graph $G = (V, V' \cup V'';E)$ where $E$ consists of the edges $(v,v')$ and $(v,v'')$ connected each node in $V$ its copies in $V'$ and $V''$.
We attach weights $\co \in \R^E$ where the edge $(v,v')$ gets the weight $w_v$ and the edge $(v,v'')$ gets the weight $0$.
This bipartite graph is depicted in Figure~\ref{fig:weighted+rank}.

Let $I \subseteq X$ be the max weight independent set contained in $X$.
The value of $\tempinducednew{G}{\overline{\I}}{\co}(X)$ is obtained by connecting elements of $I$ to $I' \subseteq V'$  via edges of weight $w_i$, and then connecting elements of $X\setminus I$ to their copy in $V''$ by edges of weight zero.
In this way $r^w =\tempinducednew{G}{\overline{\I}}{\co} $ arises from a trivially valuated generalized matroid by induction via a bipartite graph. 
\end{example}

\begin{figure}
\centering
\includegraphics[width=0.4\textwidth]{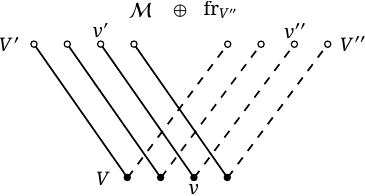}
\caption{The graph $G = (V, V'\cup V'';E)$ realising the weighted matroid rank function from Example~\ref{ex:mbv}.
Edges of weight $w_v$ are solid while edges of weight zero are dashed.}
\label{fig:weighted+rank}
\end{figure}

It was shown in~\cite{BalkanskiLeme:2020} that valuated generalized matroids are not covered by the cone of matroid rank functions; note that not even all non-negative combinations of matroid rank functions are valuated generalized matroids.
In particular, not every valuated generalized matroid can be represented as a weighted matroid rank function~\cite[Theorem~4]{Shioura:2012}.

However, it was conjectured that allowing two operations, merge and endowment, would suffice to construct all.
Here, the endowment by $T \subseteq V$ of a function $f \colon 2^V \to \Trop$ is the function $\Delta_T(f) \colon 2^{V \setminus T} \to \Trop$ with $\Delta_T(f)(X) = f(X \cup T) - f(T)$ for all $X \subseteq V \setminus T$.
Note that endowment is equivalent to the contraction of $f$ by $T$, along with a normalization to ensure $f(\emptyset) = 0$.

With this, the class of matroid based valuations are those functions arising from the class of weighted matroid rank functions by arbitrary application of merge and endowment. 

\begin{conjecture}[MBV conjecture~\cite{OstrovskyPaesLeme:2015}]\label{conj:mbv}
The class of matroid based valuations is equal to the class of monotone valuated generalized matroids taking value zero on the empty set and not attaining the value $-\infty$. 
\end{conjecture}

We study a subclass of valuated generalized matroids which is an extension of the class of R-minor valuated matroids.
This allows us to use the results from Section~\ref{sec:R-minor+not+cover}. 

\begin{definition}
  The class of \emph{$\Rnat$-induced functions} are valuated generalized matroids arising from trivially valuated generalized matroids via induction by bipartite graph. 
  
  The class of \emph{$\Rnat$-minor functions} are valuated generalized matroids arising from contractions of $\Rnat$-induced functions.
\end{definition}

Throughout the proofs in this section, we use the same notation as introduced in Section~\ref{sec:R-induced}.
Let $f$ be an $\Rnat$-minor function on $V$; by definition, there exists an $\Rnat$-induced function $\tilde{f}$ on $V \cup W$ such that $f = \tilde{f}/W$.
By definition, there exists some bipartite graph $G = (V \cup W, U; E)$ with edge weights $\co \in \R^E$ and generalized matroid $\I$ on $U$ such that $\tilde{f} = \tempinducednew{G}{\I}{\co}$; we say $\tilde{f}$ has an \emph{$\Rnat$-induced representation} $(G, \I, \co)$.
As $f = \tempinducednew{G}{\I}{\co}/W$, we extend this notation to say that $f$ has an \emph{$\Rnat$-minor representation} $(G,\I, \co, W)$, where $W$ is the set to be contracted.

\begin{lemma}\label{lem:Rnat-endowment}
  The class of $\Rnat$-minor functions is closed under endowment. 
\end{lemma}
\begin{proof}
  Given $f$ as above, we show we can represent $\Delta_T(f)$ as an $\Rnat$-minor function for some $\emptyset\neq T \subseteq V$.
  As $f$ is a contraction of $\tilde{f}$ by $W$, the endowment by $T$ can be written as
  \[
  \Delta_T(f) = f(X \cup T) - f(T) = \tilde{f}(X \cup T \cup W) - \tilde{f}(T \cup W) = \Delta_{T \cup W}(\tilde{f}) \, .
  \]
  Let $\delta = \tilde{f}(T\cup W)/|T\cup W|$ and consider a new edge weight function $\co'(e)$ that takes the value $\co(e) - \delta$ on all edges adjacent to $T \cup W$, and $\co(e)$ otherwise.
  Then the induction of $\I$ through the graph $G$ with altered weight function $\co'$ is
  \[
  (\tempinducednew{G}{\I}{\co'})(Z) = \tilde{f}(Z) - \delta \cdot |Z \cap (T \cup W)| \, .
  \]
  Taking the contraction of $\tempinducednew{G}{\I}{\co'}$ by $T \cup W$ yields
  \[
  (\tempinducednew{G}{\I}{\co'}/(T\cup W))(X) = \tilde{f}(X \cup T \cup W) - \delta \cdot | T\cup W| = \Delta_T(f)(X) \, . \qedhere
  \]
  \end{proof}

\begin{lemma}\label{lem:Rnat-merge}
  The class of $\Rnat$-minor functions is closed under merge. 
\end{lemma}
\begin{proof}
  Let $f_1, f_2$ be $\Rnat$-minor functions on a common ground set $V$ with representation $(G_i, \I_i, \co_i, W_i)$ where $G_i = (V \cup W_i, U_i, E_i)$ for $i = 1,2$ such that $f_1=\tilde f_1/W$ and $f_2=\tilde f_2/W$.
  In particular, we can choose the contracted sets to be disjoint i.e., $W_1 \cap W_2 = \emptyset$.
  This last assertion is particularly important as it allows merge and contraction to commute.
  By extending $\tilde{f}_1$ and $\tilde{f}_2$ to the ground set $V \cup W_1 \cup W_2$, taking the value $-\infty$ where previously undefined, we see that for any $X \subseteq V$,
    \begin{align*}
  (f_1*f_2)(X) &= (\tilde{f}_1/W_1 * \tilde{f}_2/W_2)(X) \\
  &= \max\SetOf{\tilde{f}_1(Y \cup W_1) +\tilde{f}_2((X \setminus Y) \cup W_2)}{Y \subseteq X} \\
  &= \max\SetOf{\tilde{f}_1(Z) + \tilde{f}_2((X\cup W_1 \cup W_2) \setminus Z)}{Z \subseteq X\cup W_1 \cup W_2} \\
  &= (\tilde{f}_1 * \tilde{f}_2)(X \cup W_1 \cup W_2) \\
  &= ((\tilde{f}_1 * \tilde{f}_2)/(W_1\cup W_2))(X)\, .
  \end{align*}
  Therefore if we can represent $(\tilde{f}_1 * \tilde{f}_2)$ via induction by bipartite graph, contracting $W_1\cup W_2$ completes the proof.
  
  Let $G'$ be a graph obtained by ``gluing'' $G_1$ and $G_2$ along their common ground set.
  Explicitly, $G' = (V \cup W_1 \cup W_2, U_1 \cup U_2; E_1 \cup E_2)$ whose weight function $\co'$ inherits the same weights from $\co_1$ and $\co_2$.
  The graph is given in Figure~\ref{fig:merge}.
  We consider the trivially valuated generalized matroid $\I' = \I_1 \oplus \I_2$.
  Then the value of $\tempinducednew{G'}{\I'}{\co'}(Z)$ is the maximum over all matchings from $Y \subset Z$ to $U_1$ and matchings $Z \setminus Y$ to $U_2$, ranging over subsets $Y \subset Z$, precisely realizing $(\tilde{f}_1 * \tilde{f}_2)$ as an $\Rnat$-induced function.
\end{proof}
\begin{figure}
\centering
\includegraphics[width=0.45\textwidth]{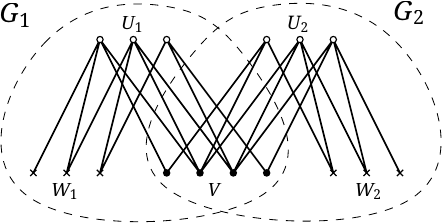}
\caption{The graph $G'$ constructed in the proof of Lemma~\ref{lem:Rnat-merge}, obtained by gluing $G_1$ and $G_2$ at their common node set $V$.}
\label{fig:merge}
\end{figure}

Example~\ref{ex:mbv} showed that weighted matroid rank functions are special cases of $\Rnat$-induced functions.
Combining this with Lemmas~\ref{lem:Rnat-endowment} and~\ref{lem:Rnat-merge}, we see that matroid based valuations are a subclass of $\Rnat$-minor functions.

\begin{corollary}\label{cor:mbv-rnat-subclass}
  Matroid based valuations form a subclass of $\Rnat$-minor functions with the properties that they are monotone, real-valued and have value $0$ on the empty set.  
\end{corollary}

\subsection{A valuated generalized matroid extending a robust matroid}

Let $h$ be an arbitrary function in the class $\cF_n$ in Definition~\ref{def:class} which takes only values in $(-1,0]$.

We define a function $h^{\natural}\colon 2^V \to \R$ by
\begin{equation*}
  h^{\natural}(X) \coloneqq \begin{cases}
    |X| & \text{ for } |X| \leq 3 \\ 
    4 + h(X) & \text{ for } |X| = 4\\
    4 & \text{ for } |X| \geq 5 \\
  \end{cases}
\end{equation*}
Thus, $h^{\natural}$ arises as a `perturbation' of the rank function of the uniform matroid on $V$ of rank $4$.

\begin{lemma} \label{lem:hv+vgm}
  The function $h^{\natural}$ is a valuated generalized matroid. 
\end{lemma}
\begin{proof}
  We first show $h^{\natural}$ satisfies~\eqref{eq:M-concave}, where $|X| = |Y| = k$.
  When $k \neq 4$, all sets of that cardinality $k$ have the same value and so $h^{\natural}$ satisfies~\eqref{eq:M-concave}.
  The case when $k = 4$ follows from Lemma~\ref{lem:F-valmat} and all sets being shifted by the same value.
  
  We next show $h^{\natural}$ satisfies~\eqref{eq:Mnat-concave}, where without loss of generality $|X| < |Y|$.
  \begin{itemize}
  \item If $|X| \geq 5$, then all sets involved in \eqref{eq:Mnat-concave} take the value 4, and therefore~\eqref{eq:Mnat-concave} is trivially satisfied.
  \item If $|X| = 4$, then $h^{\natural}(X) + h^{\natural}(Y) \leq 8$.
  If we can pick $i \in Y\setminus X$ such that $Y- i \notin \cH$, then $h^{\natural}(X + i) + h^{\natural}(Y - i) = 8$ and thus \eqref{eq:Mnat-concave}  holds.
  If no such $i$ exists, then $|Y| = 5$.
  Furthermore, there cannot be two distinct elements $i, j \in Y\setminus X$, else $Y - i, Y - j \in \cH$ intersect in three elements, which no two sets in $\cH$ do.
  Therefore $Y = X + i$, and so~\eqref{eq:Mnat-concave} holds with equality.
  \item If $|Y| = 4$, then $h^{\natural}(X) + h^{\natural}(Y) \leq |X| + 4$.
  If we can pick $i \in Y\setminus X$ such that $X + i \notin \cH$, then $h^{\natural}(X + i) + h^{\natural}(Y - i) = |X| + 4$ and this case holds.
  By a similar argument as above, if no such $i$ exists then $Y = X + i$, and so~\eqref{eq:Mnat-concave} holds with equality.
  \item If $|Y| \leq 3$, then all sets take the value of their cardinality, and therefore trivially satisfy~\eqref{eq:Mnat-concave}.\qedhere
  \end{itemize}
\end{proof}

\begin{lemma} \label{lem:non-rnat-minor}
  For $n \geq 16$, the function $h^{\natural}$ is not an $\Rnat$-minor function. 
\end{lemma}
\begin{proof}
  Suppose $h^{\natural}$ is $\Rnat$-minor, therefore it has representation $(G, \I,\co,W)$ for some graph $G = (V \cup W, U; E)$.
  From this, we derive an R-minor representation for $h$.
  
  First note that 
  \begin{align*}
  h(X) &= \layer{4}{h^{\natural}}(X) - 4 \\
  &= \layer{|W|+4}{\tempinducednew{G}{\I}{\co}}(X \cup W) - 4 \\
  &= \tempinducednew{G}{\layer{|W|+4}{\I}}{\co}(X \cup W) - 4 \, .
  \end{align*}
  By introducing the altered weight function $\co'(e) = \co(e) - 4/(|W|+4)$, we get
  \[
  \tempinducednew{G}{\layer{|W|+4}{\I}}{\co'}(X \cup W) = \tempinducednew{G}{\layer{|W|+4}{\I}}{\co}(X \cup W)) - \frac{4|X\cup W|}{|W|+ 4} = h(X) \, .
  \] 
  Therefore, $h$ has the R-minor representation $(G,\layer{|W|+4}{\I},\co',W)$, contradicting Theorem~\ref{thm:non-r-minor-main}.
\end{proof}

\begin{theorem} \label{thm:mbv+counterexample}
  The class of $\Rnat$-minor functions is not equal to the class of valuated generalized matroids.
  In particular, Conjecture~\ref{conj:mbv} is false. 
\end{theorem}
\begin{proof}
  The first claim follows immediately from Lemmas~\ref{lem:hv+vgm} and~\ref{lem:non-rnat-minor}.
  For the second claim, we observe that $h^{\natural}$ is a monotone and only takes finite values.
  However, by Corollary~\ref{cor:mbv-rnat-subclass} it is not a matroid based valuation, providing a counterexample to Conjecture~\ref{conj:mbv}.
\end{proof}

\section{Lorentzian polynomials}
\label{sec:lorentzian+polynomials}
In this section, we recall basic concepts of Lorentzian polynomials and their connection to valuated matroids, and more generally M-concave functions, via tropicalization.
We strengthen this connection by reframing operations on Lorentzian polynomials as natural operations on valuated matroids.
In particular, we show the multiplicative action of non-negative matrices on Lorentzian polynomials translates to induction by networks for valuated matroids.
As an application of our main counterexample, we demonstrate the limitation of this operation. 
We show this operation does not suffice to generate the space of Lorentzian polynomials over Puiseux series from generating functions of matroids, and that other real closed fields require arbitrarily large matroids.

\subsection{Background}

We recall the basic properties of M-concave functions; see~\cite{Murotabook} for further details. 
A function $f \colon \Z^n \to \Trop$ is \emph{M-concave} if and only if
  \begin{equation} \label{eq:M-concave-general}
    \begin{aligned}
      &\forall x,y \in \Z^n \text{ and all } i \in \supp^+(x-y): \\
      &f(x) + f(y) \leq \max_{j \in \supp^-(x-y)} \{f(x - e_i + e_j) + f(y + e_i - e_j)\} ,
    \end{aligned}
\end{equation}
where $\supp^+(z) = \{i \in [n] \colon z_i  > 0\}$ and $\supp^-(z) = \{i \in V \colon z_i  < 0\}$ for $z \in \Z^n$, and $e_{\ell}$ is the $\ell$-th unit vector. 
This extends~\eqref{eq:M-concave} from points in $\{0,1\}^n$ to $\Z^n$. 
An M-concave function has $\sum_{i=1}^n z_i = d$ for some fixed $d \in \Z$ for all $z \in \dom(f)$; we call $d$ the \emph{rank} of $f$. 
Observe that an M-concave function with $\dom(f) \subseteq \{0,1\}^n$ is a valuated matroid.

A set $B \subset \Z^n$ is \emph{M-convex} if its characteristic function, taking value $0$ on elements of $B$ and $-\infty$ otherwise, is an M-concave function.

Let $\K$ be an arbitrary ordered field. 
Furthermore, let $\Delta^d_n$ be the set of lattice points $\{x \in \Z_{\geq 0}^n \colon \sum_{i=1}^{n}x_i = d\}$. 
Given a multivariate polynomial $p(w) = \sum_{\alpha \in \Delta^d_n} c_{\alpha} w^{\alpha} \in \K[w_1,\dots,w_n]$, its \emph{support} $\supp(p)$ is the set $\{\alpha \in \Delta^d_n \colon c_{\alpha} \neq 0\}$.

Several characterizations of Lorentzian polynomials were given in~\cite{BraendenHuh:2020}; we follow their exposition.
Let $\Mco{d}{n}(\K)$ denote the homogeneous polynomials over $\K$ of degree $d$ on $n$ variables with non-negative coefficients whose support is an M-convex set. 
The set of Lorentzian polynomials over $\K$ of degree $d$ on $n$ variables is denoted by $\Lor{d}{n}(\K)$ and is defined recursively.

\begin{definition}[{\cite[Definition~3.18]{BraendenHuh:2020}}]
  $\Lor{0}{n}(\K) = \Mco{0}{n}(\K)$ and $\Lor{1}{n}(\K) = \Mco{1}{n}(\K)$, 
  \[
  \Lor{2}{n}(\K) = \{p \in \Mco{2}{n}(\K) \colon \text{Hessian of $p$ has at most one eigenvalue in }\K_{>0} \}. 
  \]
  For $d \geq 3$
  \[
  \Lor{d}{n}(\K) = \{p \in \Mco{d}{n}(\K) \colon \partial^{\alpha} p \in \Lor{2}{n}(\K) \text{ for all } \alpha \in \Delta^{d-2}_n \} \, ,
  \]
  where $\partial^{\alpha} = \partial_1^{\alpha_1}\cdots\partial_n^{\alpha_n}$ denotes the composition of $\alpha_i$-th partial derivative with respect to $w_i$.
\end{definition}

Br\"{a}nd\'{e}n and Huh give several other characterizations of Lorentzian polynomials when $\K = \R$, 
see~\cite[Definitions~2.1 \&~2.6]{BraendenHuh:2020}.
These definitions require taking limits, while the Hessian of a polynomial can be defined independently of a limit process, hence we only need to require $\K$ to be ordered.

We also note that while this definition holds for arbitrary ordered fields, many key results concerning Lorentzian polynomials were only proved over the real numbers.
We can extend these results to the larger class of \emph{real closed fields} via Tarski's principle; 
this states that first-order sentences of ordered fields hold over a real closed field $\K$ if and only if they hold over $\R$.
We therefore will restrict to working with real closed fields from now, and construct explicit fields from Section~\ref{sec:tropicalization} onwards.
For further model theoretic details, see~\cite[Section 3.3]{Marker:2002}.

A polynomial is \emph{multi-affine} if it has degree at most one in each variable.
For a general polynomial $p$, its \emph{multi-affine part} $\MAP(p)$ is the polynomial obtained by taking only terms with degree at most one in each variable.
These polynomials are of particular interest to us as the support of a multi-affine polynomial in $\Mco{d}{n}$ forms the bases of a matroid, as~\eqref{eq:M-concave-general} becomes the basis exchange axiom.
This will be a key connection to results from previous sections.

\smallskip

Lorentzian polynomials are closed under several basic operations, see~\cite{ChoeOxleySokalWagner:2004,BraendenHuh:2020}

\begin{proposition} \label{propdef:operations+polynomials}
  Let $\K$ be a real closed field, and let $p \in \Lor{d}{n}(\K) , q \in \Lor{e}{m}(\K)$ and $A \in \K_{\geq 0}^{n \times k}$.
  Then the following polynomials are also Lorentzian:
  \begin{enumerate}[label=(\roman*)]
  \item the \emph{deletion} $p \setminus i \in \Lor{d-1}{n}(\K)$ obtained from $p(u)$ by setting $u_i = 0$, 
  \item the \emph{contraction} $p / i := \partial_i p \in \Lor{d-1}{n}(\K)$, 
  \item the \emph{multi-affine part} $\MAP(p) \in \Lor{d}{n}(\K)$, 
  \item the \emph{matrix action} $(A \actson p)(w) := p(Aw) \in \Lor{d}{k}(\K)$ where $w = (w_1,\dots, w_k)$. 
  \end{enumerate}
\end{proposition}

\begin{proof}
Deletion and contraction follow essentially from the definition of Lorentzian polynomials.
The multi-affine part and matrix action are shown in~\cite[Corollary~3.5, Theorem~2.10]{BraendenHuh:2020} for $\K = \R$ respectively.
As both are first-order sentence, Tarski's principle implies they hold over arbitrary real closed fields.
\end{proof}

\subsection{Tropicalization} \label{sec:tropicalization}
In this section, we focus on Lorentzian polynomials over $\K = \pseries{\R}{t}$, the field of \emph{(generalized) Puiseux series}, see~\cite{Markwig:2010} for further details.
The field $\pseries{\R}{t}$ consists of formal series of the form
\[
c(t) = \sum_{k \in A} a_k t^k \ , \ a_k \in \R
\]
where $A \subset \R$ has no accumulation point and a well defined maximal element.
The leading term of a Puiseux series is the term with largest exponent.
We say a Puiseux series is \emph{positive} if its leading term has a positive coefficient, and denote the semiring of positive Puiseux series (with zero) by $\pseries{\R}{t}_{\geq 0}$.
We can extend this to make $\pseries{\R}{t}$ an ordered field by defining $c > d$ if and only if $c - d$ is a positive Puiseux series.
Crucially, $\pseries{\R}{t}$ is also real closed and therefore we can invoke Tarski's principle.

This ordered field is equipped with a non-archimedean valuation $\deg$ (an extension of the degree map) which maps all non-zero elements to their leading exponent and zero to $-\infty$.
The valuation $\deg$ extends entry-wise to vectors and matrices. 
It is enough to think of Puiseux series as polynomials in $t$ with arbitrary exponents and coefficients in $\R$.

\begin{observation} \label{obs:semiring-morphism}
  For $x,y \in \pseries{\R}{t}_{\geq 0}$ the map $\deg$ is a semiring homomorphism, this means $\deg(x+y) = \max(\deg(x),\deg(y))$ and $\deg(x \cdot y) = \deg(x) + \deg(y)$.
  Note that this does not hold for general Puiseux series, as the sum of a positive and negative series may cause the leading terms to cancel.
\end{observation} 

Recall that by definition, Lorentzian polynomials have non-negative coefficients.
As $\deg$ is a semiring homomorphism on these coefficients, this motivates the study of Lorentzian polynomials under the degree map.

\begin{definition}
  For a polynomial $p(w) = \sum_{\alpha \in \Delta^d_n} c_{\alpha}(t) w^{\alpha} \in \pseries{\R}{t}[w]$ over Puiseux series, its \emph{tropicalization} is the function $\trop(p) \colon \Delta^d_n \to \Trop$ with $\trop(p)(\alpha) = \deg(c_{\alpha}(t))$. 
\end{definition}

\begin{theorem}[{\cite[Theorem 3.20]{BraendenHuh:2020}}]\label{thm:M-concave-Lorentzian}
  For $f \colon \Delta^d_n \to \Trop$, the following are equivalent:
  \begin{enumerate}[label=(\roman*)]
  \item the function $f$ is M-concave, 
  \item there is a Lorentzian polynomial $p \in \pseries{\R}{t}[w_1,\dots,w_n]$ with $\trop(p) = f$. 
  \end{enumerate}
\end{theorem}

\begin{remark}
  Lorentzian polynomials are usually associated with M-convex functions, which are the negatives of M-concave functions. 
  However, this is merely a matter of how we choose the tropicalization as highest or lowest term, or actually its negative. 
  It translates to the choice of convention between $\max$ and $\min$ and one can easily switch between them via the relation $\max(x,y) = \min(-x,-y)$. 
\end{remark}

As a corollary of Theorem~\ref{thm:M-concave-Lorentzian}, if $p$ is a multi-affine Lorentzian polynomial then $\trop(p)$ is a valuated matroid.
This relation is strengthened in the following propositions, where many of the constructions in Section~\ref{sec:preliminaries} and the constructions in Proposition~\ref{propdef:operations+polynomials} are shown to commute. 

\begin{proposition} \label{prop:operations+tropical}
  Let $p$ be a multi-affine Lorentzian polynomial over $\pseries{\R}{t}$.
  \begin{enumerate}[label=(\roman*)]
  \item $\trop(p \setminus i) = \trop(p) \setminus i$,
  \item $\trop(p / i) = \trop(p) / i$.
  \end{enumerate}
\end{proposition}
\begin{proof}
Let $p = \sum_{\alpha \in \Delta^d_n} c_{\alpha}(t) w^{\alpha}$ be multi-affine, we can view $\alpha$ as a subset of $[n]$.

For {\em (i)}, note that $\alpha \in \supp(p \setminus i)$ if and only if $i \notin \alpha$, therefore
\[
(\trop(p\setminus i))(\alpha) = \begin{cases} \deg(c_\alpha(t)) & i \notin \alpha \\ -\infty &  i \in \alpha \end{cases} \, ,
\]
precisely the value of $(\trop(p)\setminus i)(\alpha)$.

For {\em (ii)}, note that $\beta \in \supp(p / i)$ if and only if $\beta \cup i \in \supp(p)$, therefore
\[
(\trop(p/ i))(\beta) = \deg(c_{\beta \cup i}(t)) = (\trop(p)/i)(\beta) \, .
\]
\end{proof}

\begin{proposition}
Let $p$ be a Lorentzian polynomial over $\pseries{\R}{t}$.
The tropicalization of its multi-affine part $\trop(\MAP(p))$ is the restriction of $\trop(p)$ to $\{0,1\}^n$.
\end{proposition}
\begin{proof}
Note that the claim trivially holds if $p$ is multi-affine. 
If $p$ is not multi-affine, then its multi-affine part $\MAP(p)$ has zero as the coefficient for any terms containing squares.
Under the degree map, we get
\[
\trop(\MAP(p))(\alpha) = \begin{cases} \deg(c_\alpha(t)) & \alpha \in \{0,1\}^n \\ -\infty &  \alpha \in \{0,1\}^n \end{cases} \, .
\]
which is precisely $\trop(p)(\alpha)$ restricted to $\{0,1\}^n$.
\end{proof}

We give a more general version of induction by bipartite graph than introduced in Definition~\ref{def:induction-by-networks} and Lemma~\ref{lem:induction-principle+extension}, allowing for M-concave functions and more general subgraphs than matchings.
Note this is still a special case of transformation by networks derived from~\cite[Theorem~9.27]{Murotabook}. 

\begin{proposition} \label{propdef:transformation-by-networks}
  Let $G=(V,U; E)$ be a bipartite graph with weight function $\co \in \R^E$.
  Let $g$ be an M-concave function on $\Z_{\geq 0}^U$ of rank $d$.
  Then the \emph{transformation of $g$ by $G$} is the function $\temptrafonew{G}{g}{\co} \colon \Z_{\geq 0}^V \to \Trop$ with
  \[
  \temptrafonew{G}{g}{\co}(x) = \max\SetOf{\sum_{e \in \mu} \co(e) + g(y)}{\begin{aligned}&\mu \text{ subgraph of } G \text{ with }\\ &\delta_V(\mu) = x \text{ and } \delta_U(\mu) = y\end{aligned}} , 
  \]
  where $\delta_V(\mu)$ and $\delta_U(\mu)$ are the degree vectors of $\mu$ on $V$ and $U$. 
  Furthermore, $\temptrafonew{G}{g}{\co}(x)$ is an M-concave function.
\end{proposition}
For consistency of notation with Lorentzian polynomials, we will use the node sets $V = [n]$ and $U = [k]$.

\begin{theorem}\label{thm:Lorentzian-transformation}
  Let $q \in \Lor{d}{n}(\pseries{\R}{t})$ and let $A \in \pseries{\R}{t}_{\geq 0}^{n \times k}$.
  Let $G = ([n], [k]; E)$ be the bipartite graph with weight function $\deg(A) \in \R^E$ that weights $(i,j)$ by $\deg(a_{ij})$.
  Then $\trop(A \actson q)$ is the M-concave function $\temptrafonew{G}{\trop(q)}{\deg(A)}$ arising from $\trop(q)$ by transformation via $G$. 
\end{theorem}
\begin{proof}
  Assume first that $q$ consists of a single monomial $d_\alpha \cdot w_1^{\alpha_1}w_2^{\alpha_2}\cdots w_n^{\alpha_n}$.
  Reordering yields
  \begin{equation} \label{eq:reorder}
  q(Av) = d_\alpha \cdot \left(\sum_{j=1}^{k}a_{1j} v_j \right)^{\alpha_1} \cdots \left(\sum_{j=1}^{k}a_{nj} v_j \right)^{\alpha_n} = d_\alpha \cdot \sum_{\beta \in \Delta^d_k} \left(\sum_{\substack{\mu \in [n]\times[k] \\ \delta_{[n]}\mu = \alpha \\ \delta_{[k]}\mu = \beta}} \prod_{e \in \mu} a_e \right) v^{\beta} ,
  \end{equation}
  where the coefficient of each $v^\beta$ is the sum of weights of subgraphs satisfying $\delta_{[n]}\mu = \alpha$ and $\delta_{[k]}\mu = \beta$.
  Therefore, the value $\trop(A \actson q)(\beta)$ for $\beta \in \Delta^d_k$ is
  \[
  \deg \left(d_{\alpha} \cdot \sum_{\substack{\mu \in [n]\times[k] \\ \delta_{[n]}\mu = \alpha \\ \delta_{[k]}\mu = \beta}} \prod_{e \in \mu} a_e \right) = 
  \max\SetOf{\deg(d_\alpha) + \sum_{e \in \mu} \deg(a_e)}{\mu \in [n]\times[k] \colon \delta_{[n]}\mu = \alpha \wedge \delta_{[k]}\mu = \beta}
  \]
  where we use that degree is a semiring homomorphism from Observation~\ref{obs:semiring-morphism}.
  The claim now follows by ranging over all $\alpha \in \Delta^d_n$ in the support of~$q$. 
\end{proof}
If $g$ is a valuated matroid, recall from Theorem~\ref{thm:induction-network} that $\tempinducednew{G}{g}{\co}$ is also a valuated matroid. 
In comparison, $\temptrafonew{G}{g}{\co}$ may be an arbitrary M-concave function.
The difference is that the former only allows us to take matchings in the induction process, while the latter allows us to take arbitrary subgraphs.
Restricting $\temptrafonew{G}{g}{\co}$ to its multi-affine part recovers the valuated matroid $\tempinducednew{G}{g}{\co}$.

\begin{corollary} \label{cor:tropical+induction}
  Let $q \in \Lor{d}{n}(\pseries{\R}{t})$ be multi-affine and let $A \in \pseries{\R}{t}_{\geq 0}^{n \times k}$.
  Furthermore, let $G=(V, U; E)$ be the bipartite graph with weight function $\deg(A) \in \R^E$.

  Then $\tempinducednew{G}{\trop(q)}{\deg(A)} = \trop(\MAP(A \actson q))$.
\end{corollary}

\subsection{Limitations of basic constructions}
In this section, we let $\K$ be any real closed field unless explicitly stated.

For an M-convex set $\B \subset \Z_{\geq 0}^n$, its \emph{generating function} is the Lorentzian polynomial
\[
q_{\B} \coloneqq \sum_{\alpha \in \B} \frac{1}{\alpha !}w^{\alpha} \ , \mbox{where } \alpha! \coloneqq \prod_{i=1}^n \alpha_i! \, .
\]
Of particular interest for us is when $\B \subseteq \{0,1\}^n$ i.e., $\B$ forms the set of bases of a matroid.
Let $\cG^d_n \subset \Lor{d}{n}(\K)$ be the set of all generating functions corresponding to rank $d$ matroids on $n$ elements.
For each $k \in \Z_{\geq 0}$, the set $\K_{\geq 0}^{n \times k}$ acts on $\cG^d_n$ by $A \actson q(w) = q(Av) \in \Lor{d}{k}(\K)$ where $A \in \K_{\geq 0}^{n \times k}, q \in \cG^d_n$.
We denote the orbit of this action by $\K_{\geq 0}^{n \times k} \actson \cG^d_n \subseteq \Lor{d}{k}(\K)$. 

\begin{definition}
  We say a Lorentzian polynomial is \emph{matroid induced} if it is contained in the orbit $\K_{\geq 0}^{n \times k} \actson \cG^d_n$ for some $n \geq d$.
\end{definition}

Our main theorem of this section is that over the Puiseux series, the class of matroid induced Lorentzian polynomials is a strict subclass of Lorentzian polynomials.

\begin{theorem} \label{thm:matroid+induced+subclass}
For $k \geq 10$, we have
\[
\bigcup_{n \geq d} (\pseries{\R}{t}_{\geq 0}^{n \times 2k} \actson \cG_n^d) \subsetneq \Lor{d}{2k}(\pseries{\R}{t}) \, .
\]
\end{theorem}

\begin{proof}
Containment is given by Proposition~\ref{propdef:operations+polynomials}.
For strict containment, we assume the converse, that every Lorentzian polynomial is matroid induced.
Let $h \in \cF_k$ be a valuated matroid on the ground set $[2k]$ defined in Definition~\ref{def:class} such that it takes only finite values.
Let $f\colon \Delta_{2k}^d \rightarrow \Trop$ be an M-concave function such that $f$ restricted to $\{0,1\}^{2k}$ is $h$.
By Theorem~\ref{thm:M-concave-Lorentzian}, there exists some $p \in \Lor{d}{2k}(\pseries{\R}{t})$ such that $\trop(p) = f$; furthermore $\trop(\MAP(p)) = h$ by Proposition~\ref{prop:operations+tropical}.
By the assumption that $p$ is matroid induced, there exists a matrix $A \in \pseries{\R}{t}_{\geq 0}^{n \times 2k}$ and some $q \in \cG_n^d$ such that $p = A \actson q$.
By Corollary~\ref{cor:tropical+induction} we have
\[
h = \trop(\MAP(p)) = \trop(\MAP(A \actson q)) = \tempinducednew{G}{\trop(q)}{\deg(A)} \, .
\]
contradicting Proposition~\ref{prop:h-not-R-induced} that $h$ is not an R-induced valuated matroid.
\end{proof}

We would like to extend Theorem~\ref{thm:matroid+induced+subclass} to $\R$ and other real closed fields via Tarski's principle.
However, the statement of the theorem is an infinite union of first-order statements and so cannot be extended immediately.
Instead, we show the weaker statement that for any finite integer $N$, there exists some Lorentzian polynomial over $\R$ that cannot be induced by a matroid on a ground set of size $n \leq N$.

\begin{proposition}\label{prop:lorentzian-restriction-real}
Let $\K$ be a real closed field and $k \geq 10$.
For each $N \in \N$, there exists some Lorentzian polynomial $p \in \Lor{d}{2k}(\K)$ such that
\[
p \notin\bigcup_{n \geq d}^N (\K_{\geq 0}^{n \times 2k} \actson \cG_n^d) \, .
\]
\end{proposition}

\begin{proof}
The containment $\bigcup_{n \geq d}^N (\K_{\geq 0}^{n \times 2k} \actson \cG_n^d)\subseteq \Lor{d}{2k}(\K)$ is given by Proposition~\ref{propdef:operations+polynomials}.
We claim that the sentence
\begin{equation}
p \in \Lor{d}{k}(\K) \ \rightarrow \ p \in \bigcup_{n \geq d}^N (\K_{\geq 0}^{n \times k} \actson \cG_n^d) \, ,
\end{equation}
is a first-order sentence of ordered fields for arbitrary $N \in \N$.
Theorem~\ref{thm:matroid+induced+subclass} shows this statement is false over Puiseux series for large enough even $k$, and so by Tarski's principle is false over the reals as well.

We write $p$ as shorthand for $(p_\beta : \beta \in \Delta^d_k)$,  where $p$ is a degree $d$ polynomial in $k$ variables and $p_\beta$ is the coefficient of $w^\beta$ in $p$.
The sentence $\phi(p)$ that verifies whether $p$ is Lorentzian is first-order by \cite{BraendenHuh:2020}.

Fix some generating function $q_\B$ for some rank $d$ matroid $\B$ on $n$ elements;
recall its coefficients $c_\alpha$ are one if $\alpha \in \B$ and zero otherwise.
We write $\K^{n \times k}_{\geq 0} \actson q_B$ for its orbit in this action.
By a similar reordering as~\eqref{eq:reorder}, the sentence $\psi_n^\B(p)$ that verifies whether $p \in \K^{n \times k}_{\geq 0} \actson q_B$ is given by
\[
\psi_n^\B(p) = (\exists a_{ij} : i \in [n],j \in [k]) :  \bigwedge_{\beta} \bigg(p_\beta = \sum_{\substack{\mu \in [n]\times[k] \\ \delta_{[n]}\mu = \alpha \\ \delta_{[k]}\mu = \beta}} c_\alpha \prod_{(i,j) \in \mu} a_{ij}\bigg) \, .
\]
In particular, it is a first-order sentence in the language of ordered fields.
The sentence $\psi_n(f)$ verifying whether $p \in \K^{n \times k}_{\geq 0} \actson \cG_n^d$ is given by taking the finite union of sentences $\psi_n^\B(p)$ over all rank $d$ matroids on $n$ elements; as this union is finite, it is also a first-order sentence.
Finally, we can take arbitrary finite unions of this sentence to reach the first-order sentence
\[
\phi(p) \rightarrow \left(\bigvee_{n\geq d}^N \psi_n(p)\right) \, ,
\]
proving the claim.
\end{proof}

While Proposition~\ref{prop:lorentzian-restriction-real} gives restrictions on constructing Lorentzian polynomials over $\R$ and other real closed fields, it remains an open question whether Theorem~\ref{thm:matroid+induced+subclass} holds over $\R$ or not.
\begin{question}
Is the class of matroid induced Lorentzian polynomials over $\R$ a strict subclass of Lorentzian polynomials over $\R$?
\end{question}
We note that $\pseries{\R}{t}$ is a non-archimedean field, and so its topological properties behave quite differently to over the reals; see~\cite[6.3 (iii)]{marshall} for a discussion and an example.
As Lorentzian polynomials over $\R$ have an equivalent definition in terms of topological closures~\cite[Definition~2.1]{BraendenHuh:2020}, it is reasonable to believe that the spaces of Lorentzian polynomials over $\R$ and $\pseries{\R}{t}$ could behave very differently.

\section*{Acknowledgements}

The authors thank Andr\'as Frank for clarifying the questions around representability of M-convex functions with induction by networks. 
The authors are grateful to Mario Kummer for pointing out the subtleties in the transfer from Lorentzian polynomials over Puiseux series to those over the reals.

\printbibliography

\newpage
\appendix

\section{All functions in \texorpdfstring{$\cF_n$}{Fn} are valuated matroids}
\label{section:functionH}

For convenience, we restate the definition of the family $\cF_n$, and then show that each of the functions in the family is a valuated matroid.

\family*

The proof of the following lemma is an adaptation of \cite[Lemma~8]{BansalPendavinghVanderPol:2015} that we include for intuition.
\begin{lemma}\label{lemma:sparsePavingB}
  Let $\B_0={V\choose 4}\setminus \cH\ $ and $\B_1=\dom(h)$. Then $\B_0$ and $\B_1$ are sparse paving matroids.
\end{lemma}

\begin{proof}
We first show that $\B_0={V\choose 4}\setminus \cH\ $ forms the bases of a matroid.
Suppose this is not the case, then there exists $B,B' \in \B_0$ and $e \in B\setminus B'$ such that $B - e + f \notin \B_0$ for all $f \in B' \setminus B$.
We observe that $|B' \setminus B| > 1$, else $B' = B - e + f \in \B_0$.
Hence, we let $f,f'$ be distinct elements of $B' \setminus B$ and consider $N = B - e + f$ and $N' = B - e + f'$.
As the basis exchange axiom does not hold, we have $N, N' \in \cH$.
However, we also have $N \cap N' = B - e$ has cardinality three, while elements of $\cH$ can intersect in at most two elements.
This gives a contradiction and it follows that $\B_0$ form the bases of a matroid.

We next show the circuits of $\B_0$ are of cardinality four or more.
Let $X$ be any set with $|X| = 3$.
If $X$ intersects three distinct pairs $P_i,P_j,P_k$, then added any element $x \in V \setminus X$ gives a basis $X + x \in \B_0$.
If $X$ intersects two distinct pairs $P_i, P_j$, as $n \geq 3$ there exists a pair $P_k$ that $X$ does not intersect.
Hence for any element $x \in P_k$, the set $X + x \in \B_0$ is a basis.
As all sets of cardinality three are independent, $\B_0$ is a paving matroid with $\cH$ as the circuits of size four.
As sets in $\cH$ can intersect in at most two elements, it is also sparse paving.
As $\B_1$ is obtained by removing elements of $\cH$, it is also a sparse paving matroid.
\end{proof}

\begin{lemma}\label{lem:F-valmat}
For every $n\ge 3$, all functions in $\cF_n$ are valuated matroids.
\end{lemma}
\begin{proof}
  We need to show each $h \in \cF_n$ satisfies~\eqref{eq:M-concave}. We consider three cases:
  \begin{itemize}
    \item Let $X, Y \in \B_0$ and $i \in X \setminus Y$.
    By Lemma~\ref{lemma:sparsePavingB}, the basis exchange axiom holds within $\B_0$.
    Therefore we can find $j \in Y - i$ such that $X - i + j, Y + i - j$ are both in $\B_0$, taking the value zero and satisfying~\eqref{eq:M-concave}.
    \item Let $X \in \B_0, Y \in \cH$ without loss of generality.
    If there exists $j \in Y \setminus X$ such that $X - i + j \in \B_0$, then $Y + i - j$ is also in $\B_0$ and we satisfy~\eqref{eq:M-concave}.
    If such a $j$ does not exist, there cannot be distinct $j_1, j_2 \in Y \setminus X$, else $X - i + j_1, X - i + j_2$ are both elements of $\cH$ and have intersection of cardinality 3, something elements of $\cH$ cannot have.
    Therefore $Y = X - i + j$ and so~\eqref{eq:M-concave} is satisfied with equality.
    \item Let $X, Y \in \cH$ and $i \in X \setminus Y$.
    As elements of $\cH$ can intersect in at most two elements, picking any $j \in Y \setminus X$ to exchange yields two sets in $\B_0$ with value zero, satisfying~\eqref{eq:M-concave}.\qedhere
  \end{itemize}
\end{proof}

\begin{remark}
  We can extend the above construction of the valuated matroid $h$ to any sparse paving matroid $\B$, where $\cH = {V \choose 4} \setminus \B$ is the set of circuits of rank 4.
  The proof of Lemma~\ref{lem:F-valmat} generalizes as it only uses the property that elements of $\cH$ cannot intersect in three elements, as stated in \cite[Lemma 19]{PendavinghVanDerPol:2015}.
\end{remark}

\section{Valuated matroid operations}\label{sec:VM-operations}
We prove in this section that valuated matroids are closed under the operations introduced in Section~\ref{sec:preliminaries}.

\operations*

\begin{lemma} \label{lem:contraction-dual+deletion}
  $f/Y = (f^* \setminus Y)^*$
\end{lemma}
\begin{proof}
  At first, observe that the independence of $Y$ in $\dom(f)$ implies that it is contained in a basis. 
  Hence, $V - Y$ has full rank in $\dom(f^*) = \dom(f)^*$ and we can actually apply the deletion operation. 
  
  Let $X \in \binom{V-Y}{d-|Y|}$.
  Then, as the codomain of $(f^* \setminus Y)$ is $V - Y$, we get $(f^* \setminus Y)^*(X) = (f^* \setminus Y)(V - (Y \cup X))$.
  Note that $X$ and $Y$ are disjoint by definition.
  Furthermore, from $V - (Y \cup X) \subseteq V - Y$ we obtain $(f^* \setminus Y)(V - (Y \cup X)) = f^*(V - (Y \cup X)). $
  Since the codomain of $f^*$ is $V$, this yields $f^*(V - (Y \cup X)) = f(X \cup Y)$. 
\end{proof}

\begin{lemma} \label{lem:truncation-extension+contraction}
  $f^{(1)} = f^{\mathbf{0}} / \{p\}$, where $\mathbf{0}$ is the zero vector and $p$ is the element added in the principal extension. 
\end{lemma}
\begin{proof}
  As $p$ is not a loop of $\dom(f^{\mathbf{0}})$ one can form this contraction and $\rank(\{p\}) = 1$.
  Now the claim follows directly from the definition of contraction and truncation. 
\end{proof}

The valuated truncation is further studied in~\cite{Murota:1997}.
It is shown that this actually gives rise to a valuation on all independent sets such that this forms a generalized valuated matroid. 

\begin{lemma} \label{lem:induction-principle+extension}
  Let $G = (V,U;E)$ be a bipartite graph with weight function $\co \in \R^E$ and $g$ be a valuated matroid on $U$.
  Then $\tempinducednew{G}{g}{{\co}} = ((\dots(g^{\co_1})\dots)^{\co_{|V|}})\setminus U$, the iterated principal extension of $g$ by $\{\co_1, \dots, \co_{|V|}\} \subset (\Trop)^{U}$, where $\co_i$ is the function $\co$ restricted to the edges incident with $i \in V$ extended with value $-\infty$ where it is not defined. 
  Furthermore, these principal extensions commute. 
\end{lemma}

\begin{proof}
  The claim follows by induction.
  We start with the bipartite graph $G_0 = (U',U;E_{0})$ where $U'$ is a copy of $U$, and $E_{0}$ consists of edges $(u',u)$ between copies of elements.
  Furthermore the weight function $d_{0}$ takes the value zero on all elements of $E_{0}$.
  
  We inductively define $G_{i} = (V_i, U; E_{i})$ where $V_i = U' \cup \{1,\dots,i\}$ by adding the node $i \in V$ to $G_{i-1}$ with edges $(i,u)$ for all $u \in U$.
  Furthermore the weight function $d_{i}$ takes the value of $d_{i-1}$ for all edges in $E_{i-1}$, and the value $\co_{iu}$ on the new edges $(i,u)$.
  These graphs are displayed in Figure~\ref{fig:induction-principle+extension}.
  \begin{figure}
\centering
\includegraphics[width=0.8\textwidth]{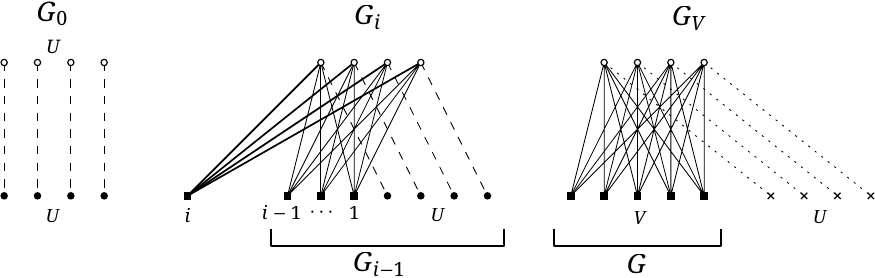}
\caption{The inductive construction of graphs corresponding to principal extension from Lemma~\ref{lem:induction-principle+extension}.}
\label{fig:induction-principle+extension}
\end{figure}
  We claim that $\tempinducednew{G_{i}}{g}{d_{i}} = (\dots(g^{\co_1})\dots)^{\co_i}$.
  
  Note that for the base case, we have that $\tempinducednew{G_0}{g}{d_0} = g$, as all edges in $G_0$ have weight zero.
  
  For the general case, consider $\tempinducednew{G_i}{g}{d_i}$ and let $X$ be a $d$-subset of $V_i = U \cup \{1,\dots,i\}$.
  If $i \notin X$, then $\tempinducednew{G_i}{g}{d_i} = \tempinducednew{G_{i-1}}{g}{d_{i-1}}$ as the graphs $G_i$ and $G_{i-1}$ are the same outside of node $i$.
  If $X = i \cup Y$, then 
  \begin{align*}
  \tempinducednew{G_i}{g}{d_i}(X) &= \max\left(\sum_{(k,v) \in \mathcal{P}} d_i(k,v) + g(Z) \; \middle| \; \partial_{V_i}(\mathcal{P}) = X \, , \, \partial_U(\mathcal{P}) = Z\right) \\
  &= \max\left(\co(i,u) + \sum_{(k,v) \in \mathcal{P'}} d_i(k,v) + g(Z' \cup u) \; \middle| \; \partial_{V_i}(\mathcal{P}') = Y \, , \, \partial_U(\mathcal{P}) = Z' \right) \\
  &= \max_{u \in V_i \setminus Y}\left(\co_{iu} + \tempinducednew{G_{i-1}}{g}{d_i}(Y \cup u)\right) \, .
  \end{align*}
  Note that for $u \notin U$, we define $c_{iu} = -\infty$, therefore this maximum will only be achieved for some $u \in U$ unless no matching $\mathcal{P}$ exists.
  This is precisely the principal extension of $\tempinducednew{G_{i-1}}{g}{d_{i-1}}$ with respect to $\co_i$.
  By the inductive hypothesis, this implies $\tempinducednew{G_{i}}{g}{d_i} = (\dots(g^{\co_1})\dots)^{\co_{i}}$.
  
  The final observation is that $G$ is obtained from the graph $G_V$ by deleting the copy of $U$ that shares no edges with $V$.
  As they share no edges, deleting these nodes is equivalent to deletion on the level of valuated matroids, therefore $\tempinducednew{G}{g}{\co} = \tempinducednew{G_V}{g}{d_V} \setminus U$.
  
  Finally, we note that as elements of $V$ share no edges, we can inductively build the graph $G_V$ by adding nodes in any order.
  On the level of valuated matroids, this implies the principal extensions commute.
\end{proof}

Let $V_1$ and $V_2$ be the respective (not necessarily disjoint) ground sets for the valuated matroids $f_1$ and $f_2$ with ranks $d_1$ and $d_2$ and let $V = V_1 \cup V_2$. 
We define a bipartite graph $G = (V, V_1 \dot{\cup} V_2;E)$ where one colour class is $V$ and the other colour class is the disjoint union of copies of $V_1$ and $V_2$.
The edge set $E$ consists of edges $(v,v)$ connecting a node to any of its copies, all weighted zero by weight function $\co$; in particular, a node of $V$ has degree two if and only if it represents an element in $V_1 \cap V_2$.
This graph is displayed in Figure~\ref{fig:union-induction+sum}.

\begin{figure}
\centering
\includegraphics[width=0.3\textwidth]{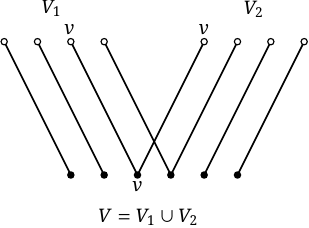}
\caption{The graph $G$ that induces the union $f_1 \vee f_2$, as described before Lemma~\ref{lem:union-induction+sum}.}
\label{fig:union-induction+sum}
\end{figure}

\begin{lemma} \label{lem:union-induction+sum}
  The union $f_1 \vee f_2$ can be written as an induction $\tempinducednew{G}{f_1 \oplus f_2}{\co}$.
\end{lemma}
\begin{proof}
  Any matching $M$ such that $\partial_V(M) = X$ corresponds to a decomposition $X = X_1 \dot{\cup} X_2$ where $X_i \subseteq V_i$.
  Therefore
  \begin{align*}
  \tempinducednew{G}{f_1 \oplus f_2}{\co}(X) = \max\SetOf{(f_1 \oplus f_2)(X)}{X_1 \in \binom{V_1}{d_1} \, , \, X_2 = X \setminus X_1 \in \binom{V_2}{d_2} } \, ,
  \end{align*}
  which is precisely the definition of $f_1 \vee f_2$.
\end{proof}

\begin{proof}[Proof of Theorem~\ref{thm:VM+closed+operations}]
  Deletion, dualization and direct sum are all covered by \cite[Theorem 6.13]{Murotabook}, parts (6), (2) and (8) respectively.
  Lemma~\ref{lem:contraction-dual+deletion} implies closure under contraction.
  Lemma~\ref{lem:union-induction+sum} and Remark~\ref{rem:principle+extension} show matroid union and principal extension are special cases of induction by networks, which valuated matroids are closed under via Theorem~\ref{thm:induction-network}.
  Finally, Lemma~\ref{lem:truncation-extension+contraction} implies closure under truncation.
\end{proof}

Finally, we show that induction by networks is a special case of induction by bipartite graphs with contraction.

\networkbipartite*

\begin{proof}
  Let $N = (T,A)$ be the weighted directed network such that the valuated matroid $f$ on the subset $V$ of $T$ is the induction of the valuated matroid $g$ on the subset $U$ of $T$ through $N$.
  Let $W = T \setminus (V \cup U)$ and $W'$ a disjoint copy of $W$.
  We define the bipartite graph $G = (V \cup W, U \cup W';E)$ with weight function $\co \in \R^E$ where for each arc $(a,b) \in A$, we add the edge $(a,b)$ if $b \in U$ or $(a,b')$ if $b \in W$ to $E$ with weight $d(a,b)$.
  Furthermore, we add the zero weight edges $(w,w')$ for all $w \in W$ with copy $w'$.
  An example of this construction is displayed in Figure~\ref{fig:network-bipartite+contraction}.

  \begin{figure}
\centering
\includegraphics[width=0.7\textwidth]{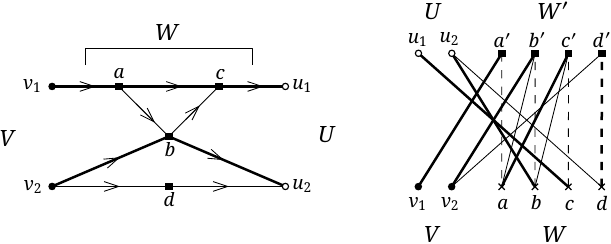}
\caption{An example of the construction from Lemma~\ref{lem:network-bipartite+contraction}: a network $N$ and the corresponding bipartite graph $G$.
A set of node-disjoint paths in $N$ correspond to a matching in $G$, both displayed in bold.}
\label{fig:network-bipartite+contraction}
\end{figure}
  
  Let $X \subseteq V$ and $Y \subseteq U$ be subsets of equal cardinality.
  We observe that node disjoint paths from $X$ to $Y$ in $N$ are in bijection with matchings from $X \cup W$ to $Y \cup W'$ in $G$, and furthermore preserve weights.
  Let $\mathcal{P}$ be a set of node disjoint paths in $N$, the edges of $G$ corresponding to arcs of $\mathcal{P}$ form a matching of equal weight on a subset of the nodes from $X \cup W$ to $Y \cup W'$.
  For any nodes $w \in W$ that are not used in $\mathcal{P}$, we add the corresponding zero weight edge $(w, w')$ to the matching: this gives a perfect matching from $X \cup W$ to $Y \cup W'$ of the same weight at $\mathcal{P}$. 
  Conversely, any perfect matching $\mu$ from $X \cup W$ to $Y \cup W'$ gives rise to a set of node disjoint paths by contracting the $(w,w')$ in $G$ for all $w \in W$.
  This precisely recovers the network $N$ from $G$, and the matching $\mu$ becomes a set of node disjoint paths from $X$ to $Y$ in $N$.
  
  We let $h$ be the valuated matroid $g \oplus \free{W'}$ as defined in Example~\ref{ex:extension-by-coloops}.
  Consider $f(X)$ for some $X \subseteq V$.
  As node disjoint paths from $X$ in $N$ are in bijection with matchings on $X \cup W$ in $G$, we can replace $N$ with $G$ in the definition of $f$:
  \[
  f(X) = \max\SetOf{\sum_{e \in \mu} \co(e) + g(Y) }{\text{ matching } \mu \text{ in } G \colon \partial_{V\cup W}(\mu) = X\cup W \wedge \partial_{U\cup W'}(\mu) = Y \cup W'} .
  \]
  Furthermore, by definition of $h$ we can replace $g(Y)$ with $h(Y \cup W')$ in the above equation.
  This implies that $f(X) = \tempinducednew{G}{h}{\co}(X \cup W)$; furthermore this holds for arbitrary $X$ and so $f = \tempinducednew{G}{h}{\co} / W$.
\end{proof}

\section{From valuated generalized matroids to valuated matroids}
\label{sec:vgm-to-vm}

By definition, valuated matroids are defined only on a layer of the ground set, but it is easy to check 
that each valuated matroid is also a valuated generalized matroid if we set the function value to $-\infty$ outside of the layer.
Another way to obtain a valuated generalized matroid from a valuated matroid is by truncation (introduced in Section~\ref{sec:preliminaries}) and elongation.
The interested reader is referred to~\cite{Murota:1997}, in particular Theorem~3.2.

\smallskip

Here, we demonstrate how to go in the other direction, i.e., how to represent a valuated generalized matroid as a valuated matroid.
This construction also appears in \cite[Lemma 8.5]{Leme2017} and \cite[Proposition 3]{MurotaShioura:2018-b}.  
Then we show an explicit construction for the case of R\nat-minor valuated generalized matroids.

\medskip
Let $f:2^{V_1} \to \Trop$ be a valuated generalized matroid. 
Denote with $n$ the size of $V_1$ and let $V_2$ be a copy of $V_1$. 
We define a function $g_f : {{V_1\cup V_2} \choose n} \to \Trop$ for $X\in {{V_1\cup V_2} \choose n}$ as
\[
g_f(X) := f(X\cap V_1)\,.
\]
Then, it is a straightforward check via the valuated (generalized) matroid axioms that the function $g_f$ is a valuated matroid. 
Note that given such a function $g_f$, we can recover $f$ as $f(X) = g_f(X\cup Y)$ for any $Y\subseteq V_2$ of size $n - |X|$.

\medskip
Starting with an R\nat-induced or an R\nat-minor valuated generalized matroid, a similar construction gives rise to an R-minor valuated matroid. 
Let $f:2^{V_1}\to \Trop$ be an R\nat-minor valuated generalized matroid represented by $(G_1, \M_1,\co, W)$ where $G_1=(V_1\cup W, U_1; E)$.
For $n = |V_1|$, let $V_2, U_2$ be two disjoint sets, each with $n$ elements, and disjoint from $V_1\cup W \cup U_1$.
Let $\M_2$ be the free matroid on $U_2$.
Consider the R-minor valuated matroid $g$ defined by the bipartite graph $G = \left(\, (V_1\cup V_2)\cup W,\, U_1\cup U_2 ;\, E'\right)$, 
matroid $\M$ on $U_1\cup U_2$, $\co'\in \R^{E'}$, and $W$; where
\begin{itemize}
  \item $\M$ is obtained by truncating $\M_1 \oplus \M_2$ to the size $|W| + n$,
  \item $E'$ is obtained from $E$ by adding all possible edges $(i,j)$, for $i \in V_2$, $j\in U_1\cup U_2$,
  \item $\co'$ extends $\co$ to $E'$ by weighting all edges in $E' \setminus E_0$ by zero.
\end{itemize}
Then, a maximal independent matching in $G$ on $X \cup W$ must come from a maximal independent matching in $G_1$ with additional zero weight edges adjacent to all nodes in $X \cap V_2$, verifying that $g$ is the same valuated matroid as $g_f$ defined in the previous paragraph.

\section{The size of R-induced representations}
\label{sec:size-bound-R-rep}
We show that any R-induced valuated matroid has an R-induced representation where the bipartite graph has size $O(|V|\cdot d)$, 
where $d$ is its rank.
A corollary is that not all valuated matroids are R-induced.

\begin{lemma}\label{lem:informationBound}
  Let $f:{V\choose d}\to \Trop$ be an R-induced valuated matroid with representation $G=(V, U; E)$, $\M = (U, r)$ and $\co\in \R^E$.
  Then, there is an R-induced representation of $f$ with $G'=(V, U'; E')$, $\M' = (U', r')$ and $\co'\in \R^{E'}$ 
  such that $|\Gamma_{G'}(v)|\le d$ for all $v\in V$.
  In particular, $|E'| + |U'|+|V| \in O(|V|\cdot d)$.
\end{lemma}
\begin{proof}
  Consider an arbitrary node $v\in V$, and the set of its neighbours $\Gamma_G(v)$ in $U$.
  Let us define a weight function $\omega$ over $\Gamma_G(v)$ as $\omega(u) = c_{vu}$ for $u \in \Gamma_G(v)$.
  Let $S$ be a maximum weight basis in the matroid $\M$ restricted to $\Gamma_G(v)$ with respect to the weights $\omega$.
  As $\M$ has rank $d$ it follows that $|S|:=s \le d$.

  To prove the lemma, it suffices to show that for any set $X\in \dom(f)$ with $v\in X$, 
  in any maximum weight independent matching $\mu^X$ defining $f(X)$
        the edge incident to $v$ can be switched to have the other end point in $S$. 

  Let $\mu^X$ be an independent matching covering $X$ where $X\in \dom(f)$ and $v\in X$.
  Denote with $u$ the node in $U$ matched to $v$ by $\mu^X$. 
  If $u\in S$, there is nothing to show. 
  So, assume $u\not \in S$.
  Let $T$ be the set of all other endpoints of $\mu^X$ in $U$. 
  That is, the set of endpoints of $\mu^X$ in $U$ is exactly $T\cup \{u\}$, where $u\not \in T$ and $|T|=|X|-1$.
  We show that we can swap $(v,u)$ by an edge $(v, u')$ for $u' \in S$ without decreasing the weight of the matching. 

  Denote the elements of the neighbourhood $\Gamma_G(v)$ by $u_1, \dots, u_s$ such that $\omega(u_1)\ge \dots \ge \omega(u_s)$.
  Since $S$ is a maximum weight basis, there is a $k\in [s]$ such that $\omega(u_1)=c_{v u_1}\ge  \dots \ge \omega(u_k)=c_{v u_k} \ge \omega(u)=c_{v u}$ and $u\in \cl[\{u_1,\ldots,u_k\}]$ (by the greedy algorithm for finding a maximum weight basis in a matroid).

  If we can replace $(v,u)$ by an edge $(v, u_t)$ for $t\in [k]$ in $\mu^X$, 
   we get a new independent matching with weight at least as much as the weight of $\mu^X$.
  On the other hand, suppose that for any $t\in [k]$ the set $\mu^X\cup\{(v, u_t)\}\setminus \{(v, u)\}$ is not an independent matching.
  Then, it must be the case that $\{u_1, \dots, u_k\} \subseteq \cl[T]$.
  Since, $u\in \cl[\{u_1, \dots, u_k\}]$ it follows that $u\in \cl[T]$.
  A contradiction. 
  It follows that we can always swap $(v, u)\in \mu^X$ for an edge $(v, u')$ where $u'\in S$, to obtain a matching with weight at least the weight of $\mu^X$.
  The lemma follows.
\end{proof}

\paragraph{Information-theoretic separation}
We use the above lemma to give an alternative proof that not all valuated matroids are R-induced. 
Note that this is also proved in Proposition~\ref{prop:h-not-R-induced}.

Let $f : {V\choose 4} \to \Trop$ be an R-induced valuated matroid and consider its R-induced representation $(G,\M,c)$ given by Lemma~\ref{lem:informationBound}; in particular, $G = (V, U; E)$ where $|E|\le |V|\cdot \rank(f) = 4|V|$.
Let $C = \{c_{ij} : (i,j) \in E\}$ be the set of weights appearing on the edges; note that we trivially have $|C|\le 4|V|$.
For any set $X\in {V\choose 4}$, the value $f(X)$ is either $-\infty$ or a sum of precisely four numbers in $C$.
This implies that the set of function values is contained in the $\Q$-vector space generated by~$C$.
In particular, the dimension of this vector space is bounded above by $|C|$.

We now exhibit a family of valuated matroids for which the $\Q$-vector space generated by its attained values has dimension greater than $4|V|$.
Recall from Definition~\ref{def:class} and Appendix~\ref{section:functionH} the sparse paving matroid with bases $\binom{V}{4} \setminus \cH$, where $\cH$ is the set of pairs $P_i \cup P_j$ where at least one of $i,j$ are even.
We define a valuated matroid by
\[
h(X) = \begin{cases} 0 & X \in \binom{V}{4} \setminus \cH \\ \alpha_X & X \in \cH \end{cases} \ , \ \alpha_X < 0 \, . 
\]
In particular, the values $\alpha_X$ for $X\in {\cH}$ can be assigned freely.
Consider such a function for which the set $A = \{\alpha_X : X\in \cH\}$ is a set of linearly independent real numbers over $\Q$.
Therefore the $\Q$-vector space generated by values of $h$ has dimension at least $|A|$.
By definition of $\cH$, we have $|A| = \binom{n}{2} - \binom{\lfloor n/2 \rfloor}{2}$; in particular, this grows quadratically as opposed to $|C|$ which grows linearly.
For $n \ge 23$, we have that $|A| > 4\cdot 2n = 4 |V|$. 
Hence, such a function $h$ is not an R-induced valuated matroid.

\medskip
Finally we mention that with a similar proof, it is easy to show an analogous lemma for R\nat-induced valuated generalized matroids. 
\begin{lemma}
  Let $f:2^V\to \Trop$ be an R\nat-induced valuated matroid with representation $G=(V, U; E)$, $\M = (U, r)$ and $c\in \R^E$.
  Then, there is an $\Rnat$-induced representation of $f$ with $G'=(V, U'; E')$, $\M' = (U', r')$ and $c'\in \R^{E'}$ 
  such that $|\Gamma_{G'}(v)|\le \min\{n, r(\M)\}$ for all $v\in V$.
  In particular, $|E'| + |U'|+|V| \in O(|V|^2)$.
\end{lemma}

\end{document}